\documentclass[11pt]{article}
\usepackage[utf8x]{inputenc}
\usepackage{microtype}
\usepackage[T1]{fontenc}
\usepackage{ae,aecompl}
\usepackage{times}
\usepackage{mathrsfs}  
\usepackage{verbatim,amsmath,amsthm,amsfonts,amssymb,latexsym,graphicx,mathtools,extpfeil,color} 
\usepackage{epstopdf,pinlabel}
\usepackage[all]{xy}
\usepackage{graphicx}
\usepackage{caption}
\usepackage{subcaption}
\usepackage{tikz}
\DeclareMathAlphabet{\mathpzc}{OT1}{pzc}{m}{it}
\usepackage[colorlinks,pagebackref,hypertexnames=false]{hyperref} \usepackage[alphabetic,backrefs,msc-links]{amsrefs}

\usepackage{aliascnt}
\numberwithin{equation}{section}

\newcommand{\bbL}{{\mathbb L}}
\newcommand{\bbH}{{\mathbb H}}

\newcommand{\rD}{{\rm D}}

\newcommand{\bB}{{\bf B}}
\newcommand{\bC}{{\bf C}}

\newcommand{\bJ}{{\bf J}}

\newcommand{\bN}{{\bf N}}

\newcommand{\bP}{{\bf P}}
\newcommand{\bQ}{{\bf Q}}
\newcommand{\bR}{{\bf R}}

\newcommand{\bX}{{\bf X}}

\newcommand{\bZ}{{\bf Z}}

\newcommand{\sL}{\mathscr{L}}

\newcommand{\fc}{{\mathfrak c}}

\newcommand{\fp}{{\mathfrak p}}
\newcommand{\fq}{{\mathfrak q}}

\newcommand{\fD}{{\mathfrak D}}

\newcommand{\fL}{{\mathfrak L}}

\newcommand{\fP}{{\mathfrak P}}
\newcommand{\fQ}{{\mathfrak Q}}
\newcommand{\fR}{{\mathfrak R}}
\newcommand{\fS}{{\mathfrak S}}

\newcommand{\N}{\bN}
\newcommand{\Z}{\bZ}

\newcommand{\R}{\bR}
\newcommand{\C}{\bC}

\newcommand{\SO}{{\rm SO}}

\newcommand{\SU}{{\rm SU}}

\DeclareMathOperator{\ad}{ad}

\DeclareMathOperator{\tr}{tr}

\renewcommand{\div}{\operatorname{div}}

\newcommand{\dvol}{{\rm dvol}}

\renewcommand{\epsilon}{\varepsilon}

\def\({\mathopen{}\left(}
\def\){\right)\mathclose{}}
\def\<{\mathopen{}\left<}
\def\>{\right>\mathclose{}}

\usepackage{multicol, color}

\definecolor{gold}{rgb}{0.85,.66,0}
\definecolor{cherry}{rgb}{0.9,.1,.2}
\definecolor{burgundy}{rgb}{0.8,.2,.2}
\definecolor{orangered}{rgb}{0.85,.3,0}
\definecolor{orange}{rgb}{0.85,.4,0}
\definecolor{olive}{rgb}{.45,.4,0}
\definecolor{lime}{rgb}{.6,.9,0}
\definecolor{green}{rgb}{.2,.7,0}
\definecolor{grey}{rgb}{.4,.4,.2}
\definecolor{brown}{rgb}{.4,.3,.1}

\def\makeautorefname#1#2{\AtBeginDocument{\expandafter\def\csname#1autorefname\endcsname{#2}}}

\newcommand{\mynewtheorem}[2]{
  \newaliascnt{#1}{equation}          
  \newtheorem{#1}[#1]{#2}
  \aliascntresetthe{#1}
  \makeautorefname{#1}{#2}
}
\mynewtheorem{theorem}{Theorem}
\mynewtheorem{prop}{Proposition}
\mynewtheorem{cor}{Corollary}
\mynewtheorem{construction}{Construction}
\mynewtheorem{lemma}{Lemma}
\mynewtheorem{conjecture}{Conjecture}

\numberwithin{substep}{step}
\makeautorefname{step}{Step}
\makeautorefname{substep}{Step}

\numberwithin{subcase}{case}
\makeautorefname{case}{Case}
\makeautorefname{subcase}{case}

\theoremstyle{remark}
\mynewtheorem{remark}{Remark}

\theoremstyle{definition}
\mynewtheorem{definition}{Definition}
\mynewtheorem{example}{Example}
\mynewtheorem{exercise}{Exercise}
\mynewtheorem{convention}{Convention}
\newtheorem*{convention*}{Convention}
\newtheorem*{conventions*}{Conventions}
\mynewtheorem{question}{Question}
        
\makeautorefname{chapter}{Chapter}
\makeautorefname{section}{Section}
\makeautorefname{subsection}{Section}
\makeautorefname{subsubsection}{Section}

\theoremstyle{introthm}
\newtheorem{introthm}{Theorem}

\title{Lagrangians, SO(3)-instantons and Mixed Equation}
\author{\bf \sc \large Aliakbar Daemi\thanks{The work of AD was supported by NSF Grant DMS-1812033 and NSF FRG Grant DMS-1952762.} \hspace{1cm} Kenji Fukaya\thanks{The work of KF was supported by the Simons Foundation through its Homological Mirror Symmetry Collaboration grant.} \hspace{1cm} Maksim Lipyanskiy}
\date{}
\begin{document}
\maketitle

\begin{abstract}
	The {\it mixed equation}, defined as a combination of the anti-self-duality equation in gauge theory and Cauchy-Riemann equation in symplectic geometry, is studied. In particular, regularity and Fredholm properties are established for the solutions of this equation, and it 
	is shown that the moduli spaces of solutions to the mixed equation satisfy a compactness property which combines Uhlenbeck and Gormov compactness theorems. The results of this paper are used in a sequel to study the Atiyah-Floer conjecture.
\end{abstract}

{
  \hypersetup{linkcolor=black}
  \tableofcontents
}
\newpage
\section{Introduction}

The {\it Cauchy-Riemann equation} and the  {\it anti-self-duality} equation provide two important geometric partial differential equations. For any Riemann surface $S$ and an almost complex manifold $M$, we may define the CR equation on the space of maps from $S$ to $M$. In the case that the target manifold $M$ is a symplectic manifold, the moduli of solutions of this equation admits a nice compactification known as {\it stable map compactification}. Such moduli spaces have been the essential ingredient in the development of various important tools in symplectic topology. For instance, Lagrangian Floer homology, which is a homology group associated to a pair of Lagrangians in a symplectic manifold $M$, is defined using the solutions of the CR equation for the space of maps from the strip $S=[-1,1]\times \R$ with Lagrangian boundary condition \cite{Fl:LFH,Oh:LFH,FOOO:HF1,FOOO:HF2}. Given a vector bundle $V$ over a Riemannian 4-manifold $X$, the ASD equation can be defined on the space of connections on the vector bundle $V$. The moduli of solutions to this equation play a key role in the definition of {\it Donaldson invariants} \cite{Don:Polynomial,DK} and {\it instanton Floer homology} \cite{floer:inst1,Don:YM-Floer} which are respectively powerful invariants of 4- and 3-manifolds.

Atiyah-Floer conjecture states that instanton Floer homology and Lagrangian Floer homology are related to each other (see \cite{At:AT-Fl,floer:inst1}). More specifically, the instanton Floer homology of a 3-manifold is isomorphic to Lagrangian Floer homology of appropriate Lagrangians in the space $M$ of flat connections on a vector bundle over a Riemann surface. One motivation for this conjecture is due to a relation between the ASD and CR equations. In fact, the CR equation with the target space $M$ can be regarded as an adiabatic limit of the ASD equation (see \cite{At:AT-Fl}). This observation was used in the remarkable work \cite{DS:At-Fl} to prove an instance of the Atiyah-Floer conjecture for 3-manifolds which are mapping tori. In this paper and its companion, we follow a different approach toward the Atiyah-Floer conjecture. We study another geometric PDE, called the {\it mixed equation}, which is defined by combining the CR and ASD equations in the third author's unpublished work \cite{Max:GU-comp}. In the sequel, we use the results of the current paper on the analytical properties of the mixed equation to prove the generalization of \cite{DS:At-Fl} for admissible bundles on arbitrary 3-manifolds.

\subsection*{Mixed Equation}
Suppose $X$ is a 4-manifold with boundary $\gamma \times \Sigma$ where $\Sigma$ is a possibly disconnected closed Riemann surface and $\gamma$ is an oriented connected 1-manifold. Thus, $\gamma$ is diffeomorphic to either $S^1$ or $\R$. Suppose $V$ is an $\SO(3)$-bundle over $X$. For each connected component $\Sigma_0$ of $\Sigma$, we require that the restriction of $V$ to $\gamma \times \Sigma_0$ is the pull-back of the non-trivial $\SO(3)$-bundle over $\Sigma_0$. In particular, the restriction of $V$ to $\gamma\times \Sigma$ is the pull-back of an $\SO(3)$-bundle $F$ on $\Sigma$. 
We fix a Riemannian metric on $X$ such that the restriction of the metric to a collar neighborhood of the boundary is given by
\begin{equation} \label{metric-near-bdry}
  ds^2+d\theta^2+g_\Sigma,
\end{equation}
for a fixed metric $g_\Sigma$ on $\Sigma$. Here we identify a collar neighborhood of the boundary of $X$ with $(-1,0] \times \gamma \times \Sigma$, and $s$, $\theta$ are respectively the coordinates on $(-1,0]$, $\gamma$.

Suppose $\mathcal A(\Sigma,F)$ denotes the space of connections on $F$. This space is an affine Banach space after Banach completion, and the automorphisms of $F$ acts on it by taking pullback. The moduli space $\mathcal M(\Sigma,F)$ is the quotient of flat connections in $\mathcal A(\Sigma,F)$ by the action of {\it determinant one} automorphisms of $F$. The Hodge operator $*_2$, defined using the conformal structure of $\Sigma$, acts on the space of $1$-forms and it gives rise to complex structures on $\mathcal A(\Sigma,F)$ and $\mathcal M(\Sigma,F)$. We denote the latter complex structure on $\mathcal M(\Sigma,F)$ by $J_*$. Define 
\[
  \mathcal L(\Sigma,F)=\{(\alpha,[\beta])\in \mathcal A(\Sigma,F)\times \mathcal M(\Sigma,F) \mid \text{$\alpha$ is flat and represents the class $[\beta]$}\}.
\]
The spaces $\mathcal A(\Sigma,F)$ and $\mathcal M(\Sigma,F)$ admit symplectic forms $\Omega$ and $\omega_{\rm fl}$, and $\mathcal L(\Sigma,F)$ defines a Lagrangian correspondence from $\mathcal A(\Sigma,F)$ to $\mathcal M(\Sigma,F)$. Motivated by this, $\mathcal L(\Sigma,F)$ is called the {\it matching Lagrangian correspondence}.

Suppose $S$ is a compact oriented Riemann surface whose boundary is
\begin{equation} \label{bdry-S}
  \partial S=\eta_1 \cup \dots \cup \eta_k \cup -\gamma,
\end{equation}
where $\eta_i$ is a connected $1$-manifold and $-\gamma$ denotes $\gamma$ with the reverse orientation. Throughout the paper, we use a similar notation to indicate reversing orientation on a manifold. For each boundary component $\eta_i$ of $S$, we fix a Lagrangian submanifold $L_i$ of the moduli space of flat connections $\mathcal M(\Sigma,F)$. We write $\mathbb L$ for the collection $(L_1,\dots,L_s,\mathcal L(\Sigma,F))$.

Following \cite{Max:GU-comp}, the mixed equation is associated to any quintuple of the form
\begin{equation}\label{matching-quintuple}
  (X,V,S,\mathcal M(\Sigma,F),\mathbb L).
\end{equation}
A pair of a connection $A$ on the bundle $V$ and a map $u:S\to \mathcal M(\Sigma,F)$ is a solution of the mixed equation if it satisfies the equations
\begin{equation} \label{mixed}
	\left\{
		\begin{array}{ll}
		F^+(A)=0,\\
		\overline \partial_{J_*} u=0, \\		
		\end{array}
	\right.
\end{equation}
and the {\it boundary} and {\it matching} conditions
\begin{equation} \label{config-space}
	\left\{
		\begin{array}{ll}
		u(x)\in L_i& x\in \eta_i,\\
		(A|_{\{x\}\times \Sigma},u(x))\in \mathcal L(\Sigma,F) & x\in \gamma.\\		
		\end{array}
	\right.
\end{equation}
The term $F^+(A)$ in \eqref{mixed} is the self-dual part of the curvature $F_A$ of the connection $A$. That is to say, the first equation requires that $A$ satisfies the ASD equation, which is also known as the instanton equation. The holomorphic curve equation $\overline \partial_{J_*} u=0$ in \eqref{mixed} is defined using the conformal structure on $S$ and the complex structure $J_*$ on $\mathcal M(\Sigma,F)$. More generally, we may define the mixed equation when $\mathcal M(\Sigma,F)$ is replaced by an arbitrary symplectic manifold $(M,\omega)$ with a compatible almost complex structure $J$, and $\mathcal L(\Sigma, F)$ is replaced by a {\it canonical Lagrangian correspondence} $\mathcal L$ from $\mathcal A(\Sigma,F)$ to $M$. We then call $ (X,V,S,M,\mathbb L)$ a {\it quintuple}, where $\mathbb L$ is the data of the canonical Lagrangian correspondence $\mathcal L$ and the Lagrangians $L_i\subset M$ associated to the boundary component $\eta_i$ of $S$. A quintuple of the special form in \eqref{matching-quintuple} is called a {\it matching quintuple}. See Section \ref{symp-lag-cor} for more details.

\subsection*{Regularity}

The solutions of the mixed equation enjoy {\it regularity} properties similar to those of the ASD equation and the Cauchy-Riemann equation. That is to say, if $(A,u)$ is a solution of the mixed equation satisfying some initial regularity, then $(A,u)$ is $C^\infty$ smooth. The precise statement of regularity requires some care because the mixed equation is invariant with respect to automorphisms of the $\SO(3)$-bundle $V$, and we may obtain a non-smooth solution by pulling back $A$ using a non-smooth automorphism of $V$. To avoid this issue, we assume that the connection $A$ of the mixed pair $(A,u)$ is in {\it Coulomb gauge} with respect to a smooth connection $A_0$ on $V$, which means that it satisfies 
\begin{equation}\label{Coulomb-gauge}
	  d_{A_0}^*(A-A_0)=0,\,\hspace{2cm}*(A-A_0)\vert_{\partial X}=0.
\end{equation}
Moreover, since regularity is a local phenomenon, we assume that $(A,u)$ is a solution to the mixed equaion associated to the quintuple 
\begin{equation}\label{Q(r)-quintuple-intro}
  \fQ(r):=(D_-(r)\times \Sigma,D_-(r)\times F,D_+(r),M,\mathcal L).
\end{equation}
Let $\bbH_+$ and $\bbH_-$ denote the half planes $s\geq 0$ and $s\leq 0$ in the $(s,\theta)$ plane. Then $D_+(r)\subset \bbH_+$, $D_-(r)\subset \bbH_-$ in \eqref{Q(r)-quintuple-intro} are respectively the open subspaces $B_{r}(0)\cap \bbH_+$, $B_{r}(0)\cap \bbH_-$ with $B_{r}(z)$ being the ball of radius $r$ centered at the point $z\in \R^2$. For the statement of our regularity result, we may work with an arbitrary symplectic manifold $M$ and a canonical Lagrangian correspondence $\mathcal L$ from $\mathcal A(\Sigma,F)$ to $M$.

\begin{introthm}\label{regularity-thm}
	Suppose $p>2$ and $(A,u)$ is an $L^p_1$ solution of the mixed equation associated to $\fQ(r)$. Suppose $A$ satisfies \eqref{Coulomb-gauge} with respect to a smooth connection $A_0$.
	Then $(A,u)$ is smooth. 
\end{introthm}

In the case that $(A,u)$ is initially only in $L^p_1$, then we can guarantee that $A|_{\{x\}\times \Sigma}$ is an $L^p$ connection for any $x\in \gamma$. Thus, we need to take $L ^p$ completion of $\mathcal A(\Sigma,F)$ to make sense of the second condition in \eqref{config-space}. This in turn implies that, we are forced to define the the space of flat connections in $\mathcal A(\Sigma,F)$ in the weak sense as in \cite{W:Ban-elliptic}. 

Theorem \ref{regularity-thm} can be used to prove regularity for solutions $(A,u)$ of the mixed equation for more general quintuples. By picking an appropriate smooth connection $A_0$ which is close enough to $A$ in the $L^p_1$ norm, we may assume that $A$ is in the Coulomb gauge with respect to $A_0$ after applying a gauge transformation of the bundle $V$. Then Theorem \ref{regularity-thm} can be used to prove regularity of $A$ and $u$ in a neighborhood of the boundary components $\gamma\times \Sigma$ of $X$ and $\gamma$ of $S$. Then standard regularity of the solutions of ASD equation and holomorphic curve equation can be employed to show interior regularity of $A$ and $u$.

There is a sequential version of Theorem \ref{regularity-thm} which shall be useful for our purposes.
\begin{introthm}\label{regularity-sequential}
	Suppose $p>2$ and $\{(A_i,u_i)\}$ is a sequence of $L^p_1$ solutions of the mixed equation associated to $\fQ(r)$ which is $L^p_1$-convergent to $(A,u)$. 
	Suppose $A_0$ is a smooth connection on $D_-(r)\times F$ and $A_i$ is in Coulomb gauge with respect to $A_0$.
	Then $(A_i,u_i)$ is $C^\infty$ convergent to $(A,u)$.
\end{introthm}

\subsection*{Compactness}

Solutions of the mixed equation for the matching quintuple satisfies a compactness property which generalizes the Uhlenbeck compactness for the solutions of ASD equation \cite{Uh:com,Uh:remov-sing} and the Gromov compactness for holomorphic curves in the symplectic manifold $\mathcal M(\Sigma,F)$ \cite{Gro:comp}.

\begin{introthm}\label{GU-comp}
	There is a constant $\hbar$ such that the following holds. Suppose $\{(A_i,u_i)\}$ 
	is a sequence of smooth solutions of the mixed equation for a matching quintuple $\fq$ as in \eqref{matching-quintuple} such that 
	\begin{equation}\label{energy-bound}
	  |\!|F_{A_i}|\!|^2_{L^2(X)}+|\!|du_i|\!|_{L^2(S)}^2\leq \kappa
	\end{equation}
	for a fixed constant $\kappa$. Then there are
	\vspace{-5pt}
	\begin{itemize}
		\item[(i)]  a subsequence $\{(A_i^\pi,u_i^\pi)\}$ of $\{(A_i,u_i)\}$,
		\item[(ii)] a solution of the mixed equation $(A_0,u_0)$ for the quintuple $\fq$,
		\item[(iii)] finite sets $\sigma_-\subset {\rm int}(X)$, $\sigma_\partial \subset \gamma$ and 
		$\sigma_+\subset S\setminus \gamma$,
	\end{itemize}
	\vspace{-5pt}
	such that the following holds.
	\vspace{-5pt}
	\begin{itemize}
		\item[(i)] The pair $(A_0,u_0)$ satisfies the energy bound 
		\[|\!|F_{A_0}|\!|^2_{L^2(X)}+|\!|d u_0|\!|_{L^2(S)}^2\leq \kappa-\hbar.\] 
		\item[(ii)] $u_i^\pi$ is $C^\infty$-convergent to $u_0$ on any compact subspace of $S\setminus (\sigma_+\cup \sigma_\partial)$.
		\item[(iii)] There are gauge transformations $g_i^\pi$ defined over $X\setminus (\sigma_\partial\times \Sigma\cup \sigma_-)$ such that $(g_i^\pi)^*A_i^\pi$ is $C^\infty$ convergent to
		$A_0$ on any compact subspace of  $X\setminus (\sigma_\partial\times \Sigma\cup \sigma_-)$.
	\end{itemize}
\end{introthm}

An important ingredient in the proof Theorem \ref{GU-comp} is an a priori estimate in Subsection \ref{energy-quant=subsec} which asserts that if we have a solution $(A,u)$ of the mixed equation satisfying the $L^2$ bound in \eqref{energy-bound} for a constant $\kappa$ less than $\hbar$, then for an appropriate choice of $p$, the $L^p_1$ norm of $(A,u)$ can be controlled. Another important input for Theorem \ref{GU-comp} is a {\it removability of singularity} result in Subsection \ref{remove-sing}, which is the analogue of corresponding result for the solutions of the ASD and CR equations.

\begin{introthm}\label{remove-sing-thm-intro}
	Let $(A,u)$ be a solution of the mixed equation for the quintuple 
	\[
	  ((D_-(r)\setminus\{0\})\times \Sigma,(D_-(r)\setminus\{0\})\times F,D_+(r)\setminus\{0\},\mathcal M(\Sigma,F),\mathcal L(\Sigma,F))
	\]
	such that 
	\[
	  |\!|F_{A}|\!|^2_{L^2(X)}+|\!|du|\!|_{L^2(S)}^2<\infty.
	\]
	Then the followings hold.
	\vspace{-5pt}
	\begin{itemize} 
		\item[(i)] There exists a gauge transformation $g$ over $(D_-(r)\setminus\{0\})\times \Sigma$ such that $g^*A$ extends to a smooth connection 
		$\widetilde A$ on $D_-(r)\times \Sigma$.
		\item[(ii)] $u$ can be extended to a smooth map $\widetilde u: D_+(r)\to  \mathcal M(\Sigma,F)$.
	\end{itemize}
	In particular, $(\widetilde A,\widetilde u)$ is a solution of the mixed solution associated to the quintuple 
	\[
	  (D_-(r)\times \Sigma,D_-(r)\times F,D_+(r),\mathcal M(\Sigma,F),\mathcal L(\Sigma,F)).
	\]
\end{introthm}

\subsection*{Fredholm theory}

The moduli spaces of the solutions of the mixed equation generically are expected to be finite dimensional smooth manifolds once appropriate decay conditions are prescribed on the non-compact ends of $X$ and $S$. The routine approach to achieve this is to establish a Fredholm theory for the linearization of the mixed equation. Fredholm theory of the linearized operator can be turned into a local problem by a cut and paste method. Given the local nature of this property, we focus on the special case of the mixed equation associated to a {\it cylinder quintuple}
\begin{equation}\label{mixed-cylinder}
	\fc_I:=(Y\times I,E\times I,[0,1]\times I,M,\{\mathcal L,L\}).
\end{equation}
where $I$ is an open interval in $\R$, $Y$ is a compact Riemannian $3$-manifold with boundary $\Sigma$, $M$ is a symplectic manifold, $\mathcal L$ is a canonical Lagrangian correspondence from $\mathcal A(\Sigma,F)$ to $M$ and $L$ is a Lagrangian in $M$. The assumption on the topological types of the bundles imply that $\Sigma$ has even number of connected components. The Riemannian metric on $Y$ induces the product metric on $Y\times I$. We also fix a family of compatible almost complex structures $\{J_{s,\theta}\}_{(s,\theta)\in [0,1]\times I}$ on $M$. The variable $\theta$ denotes the coordinate on the interval $I$ and $s$ denotes the coordinate on the factor $[0,1]$ of the region $[0,1]\times I$. We also orient $Y\times I$ using the volume form $\dvol_X=\dvol_Y\wedge d\theta$. Using the metric and the orientation on $Y\times I$, we define the first equation in \eqref{mixed}, and the second part of the mixed equation is given by the CR equation defined with respect to domain dependent almost complex structures $J_{s,\theta}$.

Given a smooth mixed pair $(A,u)$ associated to $\fc_I$, we may form an operator $\mathcal D_{(A,u)}$ which is called the {\it mixed operator}. If $(A,u)$ is the solution of the mixed equation, then the local behavior of the moduli of solutions to the mixed equation around $(A,u)$ is governed by the mixed operator $\mathcal D_{(A,u)}$. For any integer $k\geq 1$, the linearization operator can be regarded as a bounded linear map with the domain $E^k_{(A,u)}(I)$ consisting of pairs $(\zeta,\nu)$ where 
\begin{equation}\label{dom-tuple}
  \zeta\in L^2_k(Y\times I,\Lambda^1\otimes E),\hspace{1cm}\nu\in L^2_k([0,1]\times I,u^*TM),
\end{equation}
such that 
\begin{equation}\label{bdry-conditions}
  *\zeta|_{\Sigma\times I}=0,\hspace{1cm} (\zeta\vert_{\Sigma\times \{\theta\}},\nu(0,\theta))\in T\mathcal L,\hspace{1cm}\nu(1,\theta)\in TL.
\end{equation}
To be a bit more detailed, the middle condition, called the matching condition, asserts that $(\zeta\vert_{\Sigma\times \theta},\nu(0,\theta))$ belongs to the tangent space of $\mathcal L$ at the points $(A|_{\Sigma\times \{\theta\}},u(0,\theta))$ for any $\theta$. (See Section \ref{Fredholm-section} for an elaboration on this condition, especially in the case that $k=1$.) Similarly, the last condition, called the boundary condition, implies that for any $\theta$, the vector $\nu(1,\theta)$ is tangent to the Lagrangian $L$ at $u(1,\theta)$. The target of $\mathcal D_{(A,u)}$ consists of triples $(\mu,\xi,z)$ such that
\begin{equation}\label{triple-tar}
  \mu\in L^2_{k-1}(Y\times I,\Lambda^+\otimes E),\hspace{1cm}\xi\in L^2_{k-1}(Y\times I,E),\hspace{1cm}z\in L^2_{k-1}([0,1]\times I,u^*TM).
\end{equation}
The map $\mathcal D_{(A,u)}$ is a degree one differential operator and an explicit formula for this operator is given in Section \ref{Fredholm-section}. This operator is defined by linearizing the mixed equation and then including a component that is related to the first equation in \eqref{Coulomb-gauge}.

We can also consider the formal adjoint $\mathcal D_{(A,u)}^*$ of $\mathcal D_{(A,u)}$. The domain of $\mathcal D_{(A,u)}^*$, denoted by $K^k_{(A,u)}$, consists of triples $(\mu,\xi,z)$ as in \eqref{triple-tar} where $k-1$ is replaced with $k$, and the following additional conditions hold. Since $Y\times I$ is equipped with the product metric, the self-dual form $\mu$ has the form $\frac{1}{2}(d\theta\wedge b-*_3b)$ where $b$ is a section of the pullback of $T^*Y\otimes E$ to $Y\times I$. We have the following additional requirements on $(\mu,\xi,z)$:
\begin{equation}\label{bdry-conditions}
	*b|_{\Sigma\times I}=0,\hspace{1cm} (b\vert_{\Sigma\times \{\theta\}},z(0,\theta))\in T\mathcal L,\hspace{1cm}z(1,\theta)\in TL.
\end{equation}
The target of the adjoint operator $\mathcal D_{(A,u)}^*$ consists of tuples as in \eqref{dom-tuple}, where $k$ is replaced with $k-1$.  By definition, $\mathcal D_{(A,u)}^*$ is the unique operator which satisfies
\begin{equation}\label{adjoint}
   \langle \mathcal D_{(A,u)}^*(\mu,\xi,z),(\zeta,\nu)\rangle_{L^2}=\langle (\mu,\xi,z),\mathcal D_{(A,u)}(\zeta,\nu)\rangle_{L^2},
\end{equation}
for any $(\mu,\xi,z)\in K^k_{(A,u)}$ and any smooth $(\zeta,\nu)$ where $\zeta$ is compactly supported in the interior of $Y\times I$ and $\nu$ is compactly supported in the interior of $[0,1]\times I$. As it is explained in more details in Section \ref{Fredholm-section}, $\mathcal D_{(A,u)}^*$ essentially has the same form as $\mathcal D_{(A,u)}$.

\begin{introthm}\label{Fred-mixed-op}
	For any open interval $J$ that its closure is a compact subset of $I$ the following holds. 
	\vspace{-5pt}
	\begin{itemize}
	\item[(i)]Suppose $(\zeta,\nu)\in E^1_{(A,u)}(I)$ and $\mathcal D_{(A,u)}(\zeta,\nu)$ is in $L^2_{k-1}$. Then $(\zeta,\nu)\in E^k_{(A,u)}(J)$. Moreover, there is a constant $C$, independent of $(\zeta,\nu)$, such that
	\begin{equation}
		|\!|(\zeta,\nu)|\!|_{L^2_{k}(J)}\leq C\left(|\!|\mathcal D_{(A,u)}(\zeta,\nu)|\!|_{L^2_{k-1}(I)}+|\!|(\zeta,\nu)|\!|_{L^2(I)}\right).
	\end{equation}
	Similarly, suppose $(\mu,\xi,z)\in K^1_{(A,u)}(I)$ and $\mathcal D_{(A,u)}^*(\mu,\xi,z)$ is in $L^2_{k-1}$. Then $(\mu,\xi,z)\in K^k_{(A,u)}(J)$. Moreover, there is a constant $C$, independent of $(\mu,\xi,z)$, such that 
	\begin{equation}
		|\!|(\mu,\xi,z)|\!|_{L^2_{k}(J)}\leq C\left(|\!|\mathcal D_{(A,u)}^*(\mu,\xi,z)|\!|_{L^2_{k-1}(I)}+|\!|(\mu,\xi,z)|\!|_{L^2(I)}\right).
	\end{equation}	
	\item[(ii)] Suppose $(\mu,\xi,z)$ is as in \eqref{triple-tar} for $k=1$, and there is a constant $\kappa$ such that
	\[
	   \left\vert \langle (\mu,\xi,z),\mathcal D_{(A,u)}(\zeta,\nu)\rangle\right\vert\leq \kappa |\!|(\zeta,\nu)|\!|_{L^2(I)}
	\]
	for any smooth $(\zeta,\nu)$ in $E^1_{(A,u)}(I)$ with compact support. Then $(\mu,\xi,z)\in K^1_{(A,u)}(J)$. Moreover, there is a constant $C$, independent of $(\mu,\xi,z)$, such that
	\begin{equation}
		|\!|(\mu,\xi,z)|\!|_{L^2_{1}(J)}\leq C\left(|\!|\mathcal D_{(A,u)}^*(\mu,\xi,z)|\!|_{L^2(I)}+|\!|(\mu,\xi,z)|\!|_{L^2(I)}\right).
	\end{equation}	
	Similarly, suppose $(\zeta,\nu)$ is as in \eqref{dom-tuple} for $k=0$, and there is a constant $\kappa$ such that
	\[
	   \left\vert \langle (\zeta,\nu),\mathcal D_{(A,u)}^*(\mu,\xi,z)\rangle\right\vert\leq \kappa |\!|(\mu,\xi,z)|\!|_{L^2(I)}
	\]
	for any smooth $(\mu,\xi,z)$ in $K^1_{(A,u)}(I)$ with compact support. Then $(\zeta,\nu)\in E^1_{(A,u)}(J)$. Moreover, there is a constant $C$, independent of $(\zeta,\nu)$, such that
	\begin{equation}
		|\!|(\zeta,\nu)|\!|_{L^2_{1}(J)}\leq C\left(|\!|\mathcal D_{(A,u)}(\zeta,\nu)|\!|_{L^2(I)}+|\!|(\zeta,\nu)|\!|_{L^2(I)}\right).
	\end{equation}	
	\end{itemize}
\end{introthm}

Although Theorem \ref{Fred-mixed-op} does not explicitly assert Fredholmness of any mixed operator, it is the key ingredient to show that mixed operators are Fredholm in various contexts. For instance, it is straightforward to use this theorem to show that the mixed operator is Fredholm if $X$ and $S$ are compact. (The definition of the mixed operator for cylinder quintuples adapts to more general quintuples in the obvious way.) In the sequel paper, we use Theorem \ref{Fred-mixed-op} to obtain Fredholmness of the mixed operator in a case that $X$ and $S$ are non-compact but appropriate decay conditions are fixed on the non-compact ends.  

\subsection*{Outline and Conventions}

The precise definition of a canonical Lagrangian correspondence from $\mathcal A(\Sigma,F)$ to a symplectic manifold is given in Section \ref{symp-lag-cor}. We also review some technical results about such Lagrangians and the special case of the matching Lagrangian correspondence. The proof of the regularity and compactness results are respectively given in Sections \ref{regularity} and \ref{compactness-sec}. Our treatment here is essentially the same as the third author's unpublished work \cite{Max:GU-comp} with some minor modifications, most of them in exposition. Section \ref{Fredholm-section} of the paper is devoted to the proof of Theorem \ref{Fred-mixed-op} on Fredholm property of the mixed equation. In Appendices \ref{elliptic-reg-sec} and \ref{Banach-space-valued}, we collect some mostly standard analytical results, which are used throughout the paper.

The mixed equation has two predecessors in the existing literature. This equation is closely related to the ASD equation with Lagrangian boundary conditions introduced and developed in \cite{Weh:Lag-bdry-ana,Weh:Lag-bdry-ana2,SW:I-bdry}. In fact, the method of the current paper is inspired by these works and our treatment owes a great deal on these works on the analytical aspects of the ASD equation with Lagrangian boundary condition. An older relative of the mixed equation is introduced in \cite{Fuk:ASD-deg} by the second author, which is defined using the ASD equation with respect to a special degenerate metric. In fact, the mixed equation can be regarded as a limiting version of such equations. Although compactness and removability of singularity are already established for such equations \cite{Fuk:ASD-deg}, the Fredholm property seems to be a technically more difficult problem.

Throughout the paper, we use the following conventions to denote $\SO(3)$-bundles and connections on them unless otherwise stated. For any closed oriented 2-manifold $\Sigma$, there is a unique (up to isomorphism) $\SO(3)$-bundle on $\Sigma$, whose restriction to each connected component of $\Sigma$ is not trivializable. This bundle is denoted by $F$. Connections on this bundle are denoted by greek letters such as $\alpha$ and $\beta$. We write $E$ for a typical $\SO(3)$-bundle on a 3-manifold $Y$. A typical connection on this bundle is denoted by $B$. Finally, an $\SO(3)$-bundle on a 4-manifold is denoted by $V$, and a typical notation for a connection on $V$ is $A$.

The Euclidean space $\R^3$ with the standard cross product defines a Lie algebra, which is equivariant with respect to the standard $\SO(3)$ action. This $\SO(3)$-Lie algebra is isomorphic to $\mathfrak{so}(3)$, linear space of skew-adjoint endomorphisms of $\R^3$, and $\mathfrak{su}(2)$, the linear space of trace free skew-Hermitian endomorphisms of $\C^2$. Conjugation defines the $\SO(3)$ action on $\mathfrak{so}(3)$ and $\mathfrak{su}(2)$. Throughout this paper, we use this isomorphism to identify an $\SO(3)$ vector bundle $V$ with the bundle of skew adjoint endomorphisms of $V$. In particular, the curvature of a connection on $V$ can be regarded as a 2-form with values in $V$. 

Let $\tr:\R^3\times \R^3\to \R$ be the bi-linear form given by $-\frac{1}{2}$ of the standard inner product. Using the identification with $\mathfrak{su}(2)$, this bi-linear form can be identified with $\tr:\mathfrak{su}(2) \times \mathfrak{su}(2)\to \R$ which maps a pair of a skew-Hermitian matrices $A$ and $B$ to $\tr(AB)$. The bi-linear form $\tr$ induces a bi-liner form on sections of any $\SO(3)$-vector bundle $V$, which is denoted by the same notation. If $\alpha$ and $\beta$ are two general $k$-forms on a Riemannian manifold $M$ with values in an $\SO(3)$ vector bundle $V$, we use
\begin{equation}\label{inner-prod-diff-forms}
  \langle \alpha,\beta\rangle :=-\int_M \tr(\alpha\wedge *_M\beta)
\end{equation}
to define their inner products, where $*_M$ is the Hodge $*$-operator on $M$. 

\newpage
{\it Acknowledgements.} The third author would like to thank Tomasz Mrowka, Dusa McDuff and Dennis Sullivan for useful conversations. The authors are grateful to the Simons Center for Geometry and Physics for providing a stimulus environment where part of this paper and the sequel one were being completed at different stages. The second author is thankful to Washington University in St. Louis for an opportunity to visit the first author to discuss some aspects of the present paper.

\section{Symplectic manifolds and canonical Lagrangian correspondences} \label{symp-lag-cor}

The space of all connections on $F$ is an affine space modeled on $\Omega^1(\Sigma,F)$, the space of 1-forms with values in $F$. This space admits a symplectic form given by
\[
  \hspace{3cm}\Omega(a,b)=-\int_\Sigma \tr(a \wedge b),\hspace{1cm} \text{ for } a,b \in \Omega^1(\Sigma,F) .
\]
For $p>2$, let $\mathcal A^p(\Sigma, F)$ denote the completion of this affine space with respect to the $L^p$ norm. The symplectic form $\Omega$ clearly extends to $\mathcal A^p(\Sigma, F)$.
There is also an action of a Banach Lie group $\mathcal G^p_1(\Sigma,F)$ on $\mathcal A^p(\Sigma, F)$. The Lie group $\SO(3)$ acts on $\SU(2)$ by the conjugation action $\ad$, and this action determines a fiber bundle on $\Sigma$ given an
\begin{equation}\label{Auto-fib-bdle}
	{\rm Fr}(F)\times_{\ad}\SU(2).
\end{equation}
where ${\rm Fr}(F)$ denotes the framed bundle of $F$. Then $\mathcal G^p_1(\Sigma,F)$ is the space of sections $g$ of this bundle such that $\nabla_{\alpha_0}g$ has a finite $L^p$ norm where $\nabla_{\alpha_0}$ is defined using a smooth connection $\alpha_0$ on $F$. Any element of $\mathcal G^p_1(\Sigma,F)$ is continuous and pulling back connections with respect to the elements of $\mathcal G^p_1(\Sigma,F)$ gives rise to an action of $\mathcal G^p_1(\Sigma,F)$ on $\mathcal A^p(\Sigma, F)$. The symplectic form $\Omega$ is invariant with respect to this action.

The curvature of an element of $\mathcal A^p(\Sigma, F)$ is not necessarily well-defined. However, we can define the subspace $\mathcal A_{\rm fl}(\Sigma, F)$ of connections in $\mathcal A^p(\Sigma, F)$ which are weakly flat (see \cite{W:Ban-elliptic}). First fix a smooth flat connection $\alpha_0$. For $a\in L^p(\Sigma, \Lambda^1\otimes F)$, the $L^p$-connection $\alpha_0+a$ is an element of $\mathcal A_{\rm fl}(\Sigma, F)$, if 
\[
  \int_\Sigma \tr(a \wedge(d_{\alpha_0}\psi-\psi a))=0
\]
holds for any smooth section $\psi$ of the bundle $F$. This space is invariant with respect to the action of $\mathcal G^p_1(\Sigma,F)$ and determines a Banach submanifold of $\mathcal A^p(\Sigma,F)$. Any element of this space belongs to the orbit of a smooth flat connection (see \cite{W:Ban-elliptic}). We may form a neighborhood in $\mathcal A^p(\Sigma,F)$ of a smooth connection $\alpha \in \mathcal A_{\rm fl}(\Sigma, F)$ by taking connections of the form
\begin{equation}\label{nbhd-alpha}
  g^*(\alpha'+*d_{\alpha'}\zeta)
\end{equation}
where $\alpha'$ is a smooth flat connection on $\Sigma$ such that it satisfies the Coulomb gauge fixing condition $d_\alpha^*(\alpha'-\alpha)=0$, $|\alpha-\alpha'|<\epsilon$, $g\in \mathcal G^p_1(\Sigma,F)$ with $\vert\!\vert\nabla_{\alpha}g\vert\!\vert_{L^p}<\epsilon$ and $\zeta\in L^p_1(\Sigma, F)$ with $\vert\!\vert\nabla_{\alpha}\zeta\vert\!\vert_{L^p_1}<\epsilon$. The subspace $\zeta=0$ of this open set describes the intersection with $\mathcal A_{\rm fl}(\Sigma,F)$. The Hodge $*$-operator on $\Sigma$, denoted by $*_2$, induces a $\mathcal G^p_1(\Sigma,F)$-invariant complex structure on $\mathcal A^p(\Sigma,F)$. This complex structure is compatible with $\Omega$ and the induced metric on $\mathcal A^p(\Sigma,F)$ is the standard one.

The quotient $\mathcal A_{\rm fl}(\Sigma, F)/\mathcal G^p_1(\Sigma,F)$ can be identified with the moduli space of flat connections $\mathcal M(\Sigma,F)$. The symplectic form $\Omega$ on $\mathcal A^p(\Sigma,F)$ gives rise to the standard symplectic structure $\omega_{\rm fl}$ on $\mathcal M(\Sigma,F)$. The tangent space of $\mathcal M(\Sigma,F)$ to the class of a flat connection $\alpha$ can be identified with
\begin{equation}
	\mathcal H^1(\Sigma;\alpha)=\{a\in \Omega^1(\Sigma,F) \mid d_\alpha a=0,\,d_\alpha^*a=0 \}.
\end{equation}
The complex structure $*_2$ on $\mathcal A^p(\Sigma,F)$ induces a $\omega_{\rm fl}$-compatible complex structure $J_*$ on $\mathcal H^1(\Sigma;\alpha)$.

\begin{definition} \label{can-lag-cor-def}
        Suppose $(M,\omega)$ is a symplectic manifold. A Banach submanifold 
        $\mathcal L\subset \mathcal A^p(\Sigma,F) \times M$ 
        is called a {\it canonical Lagrangian correspondence} from $\mathcal A^p(\Sigma,F)$ to $M$ 
        if it satisfies the following properties:
        \begin{enumerate}
        		\item[(i)] $\mathcal L$ is invariant with respect to the action of the gauge group $\mathcal G^p_1(\Sigma,F)$ 
        		on $\mathcal A^p(\Sigma,F) \times M$.
	        	\item[(ii)] The first component of any element of $\mathcal L$ belongs to $\mathcal A_{\rm fl}(\Sigma,F)$.
        		\item[(iii)] $\mathcal L$ is isotropic with respect to $(-\Omega)\oplus\omega$, i.e., 
        			if $a$ and $b$ are two tangent vectors to 
	        		$\mathcal L$, then $\((-\Omega)\oplus\omega\)(a,b)=0$.
        		\item [(iv)] $\mathcal L$ is co-isotropic with respect to $(-\Omega)\oplus\omega$, i.e., if $a$ is a tangent vector to 
		        	$\mathcal A^p(\Sigma,F) \times M$ at a point $(\alpha,x)$ and $\((-\Omega)\oplus\omega\)(a,b)=0$
        			for any $b\in T_{(\alpha,x)}\mathcal L$, then $a$ is tangent to $\mathcal L$.		       
        \end{enumerate}
\end{definition}

There is a correspondence between canonical Lagrangian correspondences from $\mathcal A^p(\Sigma,F)$ to $M$ and Lagrangians in the (finite dimensional) symplectic manifold $\mathcal M(\Sigma,F)\times M$ equipped with the symplectic form $(-\omega_{\rm fl})\times \omega$. Given any canonical Lagrangian correspondences from $\mathcal A^p(\Sigma,F)$ to $M$, we may form a subspace of $\mathcal M(\Sigma,F)\times M$ by taking the quotient $\mathcal L/\mathcal G^p_1(\Sigma,F)$. This subspace is in fact a Lagrangian in $\mathcal M(\Sigma,F)\times M$. This follows from the following standard lemma on Hodge decomposition associated to twisted Laplace operators.
\begin{lemma}\label{Hodge-decom}
	Suppose $k\geq 0$, $q>1$ and $\alpha$ is a smooth flat connection on $F$. Then we have the following splitting of $L^2_k(\Sigma,\Lambda^1\otimes F)$ into 
	a sum of closed subspaces:
	\begin{equation}\label{Hodge-splitting}
	  L^q_k(\Sigma,\Lambda^1\otimes F)=\mathcal H^1(\Sigma;\alpha)
	  \oplus {\rm image} (d_\alpha) \oplus {\rm image} (*d_\alpha),
	\end{equation}
	where ${\rm image} (d_\alpha)$ and ${\rm image} (*d_\alpha)$ are the images of the operators
	\[
	  d_\alpha:L^q_{k+1}(\Sigma,F)\to L^q_k(\Sigma,\Lambda^1 \otimes F),\hspace{1cm}
	  *d_\alpha:L^q_{k+1}(\Sigma,F)\to L^q_k(\Sigma,\Lambda^1 \otimes F).
	\]
\end{lemma}

\begin{proof}
	This is a standard result which follows from the fact that the twisted laplacian
	\[
	  d_\alpha d_\alpha^*+d_\alpha^*d_\alpha: L^q_{k+2}(\Sigma,\Lambda^1\otimes F)\to 
	  L^q_k(\Sigma,\Lambda^1 \otimes F)
	\]
	is an elliptic operator with cokernel being $\mathcal H^1(\Sigma;\alpha)$.
\end{proof}

The splitting \eqref{Hodge-splitting} in the case that $q=p$ and $k=0$ gives a splitting of the tangent space of $\mathcal A^p(\Sigma,F)$ at smooth elements of $\mathcal A_{\rm fl}(\Sigma,F)$. The first two summands describe the tangent space to $\mathcal A_{\rm fl}(\Sigma,F)$. For any canonical Lagrangian correspondence $\mathcal L$ and any $z=(\alpha,x)\in \mathcal L$,  $T_z\mathcal L$ contains ${\rm image} (d_\alpha)$ and is $L^2$-orthogonal to ${\rm image} (*d_\alpha)$. Therefore, there is a  subspace $V_z$ of the finite dimensional symplectic vector space $\mathcal H^1(\Sigma;\alpha) \oplus T_xM$ such that
\begin{equation}\label{L-tan-decom}
	T_z\mathcal L=V_z\oplus {\rm image} (d_\alpha)\subset L^p(\Sigma,\Lambda^1\otimes F) \oplus T_xM .
\end{equation}
where the domain of $d_\alpha$ is $L^p_1(\Sigma,F)$. The definition of $\mathcal L$ is equivalent to say that $V_z$ is a Lagrangian subspace of $\mathcal H^1(\Sigma;\alpha) \oplus T_xM$. Consequently, $\mathcal L/\mathcal G^p_1(\Sigma,F)$ is a Lagrangian submanifold of $\mathcal M(\Sigma,F)\times M$. This presentation also gives a useful description for the closure of \eqref{L-tan-decom} with respect to $L^q$ norms with $1<q<p$: this closure given by the same direct sum decomposition where $d_\alpha$ should be a regarded as a map acting on $L^q_1(\Sigma,F)$. In particular, we will use this in Section \ref{Fredholm-section} in the case that $q=2$.

\begin{example} \label{matching-lag-cor}
	Let $\mathcal L(\Sigma,F)$ be the following Banach submanifold of $\mathcal A^p(\Sigma,F) \times \mathcal M(\Sigma,F)$:
	\[
	  \mathcal L(\Sigma,F)=\{(B,[B])\mid B\in \mathcal A_{\rm fl}(\Sigma,F)\}.
	\]
	This space is diffeomorphic to $\mathcal A_{\rm fl}(\Sigma,F)$ and defines a canonical 
	Lagrangian correspondence from $\mathcal A(\Sigma,F)$ to $\mathcal M(\Sigma,F)$,
	which is called the {\it matching Lagrangian correspondence}. 
	The corresponding Lagrangian in $\mathcal M(\Sigma,F)\times \mathcal M(\Sigma,F)$ is the identity Lagrangian correspondence form $\mathcal M(\Sigma,F)$ to itself.
\end{example}

Let $J$ be an almost complex structure on $M$ compatible with the symplectic form $\omega$. This induces an almost complex structure $\bJ$ on $\mathcal A^p(\Sigma,F) \times M$ which acts on $(a,v)\in  L^p(\Sigma,\Lambda^1\otimes F )\oplus T_xM$ as
\begin{equation}\label{bJ}
 \bJ(a,v)=(-*_2 a,Jv).
\end{equation}
For any $z=(\alpha,x)\in \mathcal L$, property (iii) of $\mathcal L$ implies that $T_z\mathcal L\cap \bJ(T_z\mathcal L)$ is trivial. Moreover, (iv) implies that $V_z$ and $\bJ V_z$ generate the finite dimensional symplectic space $\mathcal H^1(\Sigma;\alpha)\oplus T_xM$. In particular, we have
\[
  T_z\mathcal L\oplus\bJ(T_z\mathcal L)=L^p(\Sigma,\Lambda^1\otimes F)\oplus T_xM.
\]
The following lemma gives a suitable chart for the complex structure $\bf J$ in a neighborhood of a point in $\mathcal L$.

\begin{lemma}\label{chart-complex-str}
	Suppose $\mathcal L$ is a canonical Lagrangian correspondence from $\mathcal A^p(\Sigma,F)$ to a symplectic manifold $M$. Suppose an almost complex structure $\bf J$ on $\mathcal A^p(\Sigma,F)\times M$ is defined as in \eqref{bJ}. 
	Suppose $B_p$ is the Banach space $L^p_1(\Sigma,F)\oplus \R^{n-3\chi(\Sigma)/2}$ where $2n$ is the dimension of $M$. Then for any $z=(\alpha,x)\in \mathcal L$, there is an open neighborhood $U$ of 
	the origin of $B_p\oplus B_p$, and a diffeomorphism $\Phi_p$ from $U$ onto some open subspace
	of $\mathcal A^p(\Sigma,F)\times M$ with $\Phi_p(0)=z$ such that 
	\begin{itemize}
		\item[(i)] $\Phi_p^{-1}(\mathcal L)$ is the intersection of $0\oplus B_p$ with $U$;
		\item[(ii)] for any $x\in \mathcal L\cap {\rm im}(\Phi_p)$, the pullback of the almost complex
		structure $\bJ(x)$ is the standard complex structure 
		\[
		  (v_1,v_2) \to (-v_2,v_1);
		\]
		\item[(iii)] if $q>p$, then $\Phi_p$ maps $(B_q\oplus B_q) \cap U$ to $(\mathcal A^q(\Sigma,F)\times M)\cap \rm{image}(\Phi_p)$.
	\end{itemize}	
\end{lemma}

\begin{proof}
	We may assume that the connection $\alpha\in \mathcal A_{\rm fl}(\Sigma,F)$ is smooth. 
	Let $\breve z=([\alpha],x)$ be obtained by projecting $z$ to $\mathcal M(\Sigma,F)\times M$.
	The quotient of $\mathcal L$ by $\mathcal G^p_1(\Sigma,F)$ determines a smooth submanifold 
	$\breve {\mathcal L}$ of the finite dimensional manifold $\mathcal M(\Sigma,F)\times M$, 
	which is in fact Lagrangian with respect to the symplectic form $(-\omega_{\rm fl})\times \omega$.
	Using neighborhood theorems for Lagrangian submanifolds (of finite dimensional symplectic manifolds),
	there is a chart
	\[\breve \Phi:\breve U \to \mathcal M(\Sigma,F)\times M\]
	such that $\breve U$ is an open neighborhood of the origin in $\R^{2n-3\chi(\Sigma)}$, 
	$\breve \Phi(0)=\breve z$, $\breve \Phi^{-1}(\breve {\mathcal L})$ is the intersection of 
	$\breve U$ with 
	$\{0\}\times \R^{n-3\chi(\Sigma)/2}\subset \R^{n-3\chi(\Sigma)/2}\times \R^{n-3\chi(\Sigma)/2}$.
	The pull back of the complex structure on 
	$\mathcal M(\Sigma,F)\times M$, given as $(-*_2 a,Jv)$,
	determines a complex structure on $\breve U$, and we may pick $\breve \Phi$ such that 
	for any point in $\breve \Phi^{-1}(\breve {\mathcal L})$ this complex structure is the standard one
	\[
	  (v_1,v_2) \in \R^{n-3\chi(\Sigma)/2}\times \R^{n-3\chi(\Sigma)/2} \to (-v_2,v_1).
	\]

	The chart $(\breve \Phi,\breve U)$ can be used to define a chart 
	for $\mathcal A^p(\Sigma,F)\times \mathcal M(\Sigma,F)$.
	Let $\breve \Phi=(\breve \Phi_1,\breve \Phi_2)$ where $
	\breve \Phi_1$ and $\breve \Phi_2$ are respectively 
	maps from $\breve U$ to $\mathcal M(\Sigma,F)$ and $M$. By shrinking the open set $\breve U$, we may  
	assume that the elements in the image of $\breve \Phi_1$ are lifted to smooth elements of 
	$\mathcal A_{\rm fl}(\Sigma,F)$ which satisfy gauge fixing condition with respect to the flat connection 
	$\alpha$. With a slight abuse of notation, this lift of $\breve \Phi_1$ to a map with target 
	$\mathcal A_{\rm fl}(\Sigma,F)$ is still denoted by $\breve \Phi_1$.
	Define a map 
	\begin{equation}\label{varphi}
	  L^p_1(\Sigma,F)\times L^p_1(\Sigma,F)\times \breve U\to \mathcal A^p(\Sigma,F)\times M,
	\end{equation}
	as
	\[
	  (\zeta,\xi,v) \to \left(\exp(\zeta)^*\breve \Phi_1(v)-*_2\(\left.\frac{d}{dt}\right\vert_{t=0}\exp(\zeta+t\xi)^*(\breve \Phi_1(v))\),\breve \Phi_2(v)\right).
	\]
	By taking $U$ to be a small enough neighborhood of the origin in 
	$L^p_1(\Sigma,F)\times L^p_1(\Sigma,F)\times \breve U$ and $\Phi_p$ being the restriction of 
	\eqref{varphi}, inverse function theorem allows us to obtain the desired chart.
\end{proof}

\section{Regularity}\label{regularity}

The main goal of this section is to prove Theorems \ref{regularity-thm} and \ref{regularity-sequential} on regularity of solutions of the mixed equation. For $p>2$, suppose $(A,u)$ is an $L^p_1$ solution of the mixed equation for the quintuple $\fQ(r)$ in \eqref{Q(r)-quintuple-intro}, which we copy here again:
\begin{equation}\label{quintuple-reg}
	(X:=D_-(r)\times \Sigma,V:=D_-(r)\times F,S:=D_+(r),M,\mathcal L).
\end{equation}
Here $M$ is a symplectic manifold with a symplectic form $\omega$ and a compatible almost complex structures $J$. The space $\mathcal L$ is a canonical Lagrangian correspondence from $\mathcal A^p(\Sigma,F)$ to $M$. The mixed equation for the pair $(A,u)$ has the form 
\begin{equation} \label{mixed-Q-r}
	\left\{
		\begin{array}{ll}
		F^+(A)=0,\\
		\frac{du}{d\theta}-J(u)\frac{du}{ds}=0. \\		
		\end{array}
	\right.
\end{equation}
We write $U_\partial(r)$ for the intersection of the half discs $D_+(r)$ and $D_-(r)$. We also assume that $A$ is in Coulomb gauge with respect to a smooth connection $A_0$:
\begin{equation}\label{Coulomb-gauge-A0}
	  d_{A_0}^*(A-A_0)=0,\,\hspace{2cm}*(A-A_0)\vert_{U_\partial\times \Sigma}=0.
\end{equation}
Then a more precise statement of Theorem \ref{regularity-thm} is given as follows.
\begin{theorem}\label{regularity-thm-detailed}
	Any $(A,u)$ as above is smooth. 
\end{theorem}

The proof of Theorem \ref{regularity-thm-detailed} is performed in several steps where the regularity of $(A,u)$ is improved in each step. The proof is slightly more involved in the case that $p<4$. In this case, first we show that one can improve regularity by increasing the value of $p$. Let $\{q_i\}_{0\leq i\leq N}$ be an increasing finite sequence of real numbers such that $q_0=p$, $q_N>4$ and 
\begin{equation}\label{def-q-i}
  \hspace{2cm}q_{i+1}=\frac{2q_i}{4-q_i},\hspace{1cm}\text{for $0\leq i\leq N-1$}.
\end{equation}
We shall show that if the assumptions of Theorem \ref{regularity-thm-detailed} hold for $p=q_i$, then it also holds for $p=q_{i+1}$. 
In the case that $p>4$, we shall show that one can obtain $L^{p/2}_2$ regularity from $L^p_1$ regularity. 
In the case that $p>2$ and $k\geq 2$, a similar argument as above shows that if $(A,u)$ is in $L^p_k$, then it also belongs to $L^p_{k+1}$. Subsection \ref{regularity-thm-sub} is devoted to the proof of these claims. 

The following theorem is a more detailed version of Theorem \ref{regularity-sequential} and its proof will be discussed in Subsection \ref{regularity-thm-sequential-sub}.
\begin{theorem}\label{regularity-sequentia-detailed}
	Any $\{(A_i,u_i)\}_i$ is a sequence of smooth solutions of \eqref{mixed-Q-r} which satisfy \eqref{Coulomb-gauge-A0}. For $p>2$, suppose $(A_i,u_i)$ is $L^p_1$ convergent to $(A,u)$.
	Then $(A_i,u_i)$ is $C^\infty$ convergent to $(A,u)$.
\end{theorem}

\subsection{Proof of Theorem \ref{regularity-thm}}\label{regularity-thm-sub}

Suppose $(A,u)$ is an $L^p_1$ solution of \eqref{mixed-Q-r} associated to the quintuple $\fQ(r)$ with $2<p<4$ that satisfies \eqref{Coulomb-gauge-A0} for a smooth connection $A_0$. Suppose $A-A_0$ has the form
\[
  A-A_0=a+\phi ds+\psi d\theta
\]
with respect to the coordinate system on $D_-(r)\times \Sigma$. Since the connection $A$ satisfies the ASD equation, we have
\begin{equation}\label{dA0+-q}
	d_{A_0}^+(A-A_0)=-F(A_0)^++Q(A-A_0)
\end{equation}
where $Q(A-A_0)$ is defined to be the quadratic term $-((A-A_0)\wedge (A-A_0))^+$.  

We list some inequalities and identities here which will be used in various stages of the proof. For any $q<4$, Sobolev embedding implies that 
\begin{equation*}\label{quad-bound-pre}
	  |\!|A-A_0|\!|_{L^{\frac{4q}{4-q}}(X)} \leq C |\!|A-A_0|\!|_{L^q_1(X)}.
\end{equation*}
Since we have
\[
  \frac{4-q}{4q}+\frac{1}{q}=\frac{8-q}{4q},
\]
the H\"older inequality gives 
\begin{equation}\label{quad-bound}
	  |\!|Q(A-A_0)|\!|_{L^{\frac{4q}{8-q}}_1(X)} \leq C |\!|A-A_0|\!|_{L^q_1(X)}^2.
\end{equation}
Similarly, if we fix $q\geq 1$, then for any positive integer $k$ with $qk>4$ there is a constant $C_k$ such that
\begin{equation}\label{quad-bound-higher}
	  |\!|Q(A-A_0)|\!|_{L^{q}_k(X)} \leq C_k |\!|A-A_0|\!|_{L^q_k(X)}^2.
\end{equation}
For each $(s,\theta)\in D_-(r)$, let $\beta(s,\theta)$ (respectively, $\beta_0(s,\theta)$) denote the restriction of $A$ (respectively, $A_0$) to $\Sigma\times\{(s,\theta)\}$. In particular, we have $\beta=\beta_0+a$. Since $\beta(0,\theta)$ is flat, we have 
\begin{equation} \label{B-B0-der-flat}
  d_{\beta_0(0,\theta)}\(\beta(0,\theta)-\beta_0(0,\theta)\) =-F(\beta_0(0,\theta))-(\beta(0,\theta) -\beta_0(0,\theta) )\wedge (\beta(0,\theta) -\beta_0(0,\theta)).
  \end{equation}

We wish to use Lemma \ref{elliptic-reg} to improve regularity of the components $\phi$ and $\psi$ of $A-A_0$ over $D_-(r')$ with $r'<r$. Let $\rho:D_-(r)\to \R^{\geq 0}$ be a compactly supported function which is equal to $1$ on $D_-(r'')$ where $r'<r''<r$. As the first step, note that the second identity of the Coulomb gauge condition \eqref{Coulomb-gauge-A0} implies that for any $\xi$ in $\Gamma_c(D_-(r)\times \Sigma,V)$, the space of compactly supported smooth sections of $V$ over $X=D_-(r)\times \Sigma$, we have
\begin{align}
  \int_{X} \langle \rho (A-A_0), d_{A_0}\xi \rangle=&\int_{X} \langle d_{A_0}^*(\rho (A-A_0)), \xi\rangle \nonumber\\
  =&\int_{X}  \langle\nabla\rho\cdot (A-A_0), \xi\rangle.\label{1-identity-A12}
\end{align}
Here $\nabla\rho\cdot (A-A_0)$ is an expression which is linear in $A-A_0$ and the derivative $\nabla\rho$ of $\rho$. In particular, we observe that the $L^p_1$ norm of $\nabla\rho\cdot (A-A_0)$ is bounded by the $L^p_1$ norm of $A-A_0$.

For any $\eta$ in $\Gamma_\tau(D_-(r)\times \Sigma,V)$, the space of sections of $V$ over $D_-(r)\times \Sigma$ with vanishing restriction to $U_\partial(r)\times \Sigma$, we have
\begin{align}
	\int_{X} \langle \rho (A-A_0), d_{A_0}^*&d_{A_0} (\eta ds) \rangle
	=2\int_{X} \langle \rho (A-A_0), d_{A_0}^*d_{A_0}^+(\eta ds)\rangle+
	\int_{X}\langle \rho (A-A_0),*[F(A_0), \eta ds] \rangle \nonumber\\
	=&2\int_{X} \langle d_{A_0}^+ (\rho (A-A_0)), d_{A_0}^+(\eta ds)\rangle+2
	\int_{\Sigma\times U_\partial}\tr( \rho (A-A_0)\wedge d_{A_0}^+(\eta ds))\nonumber\\ 
	&+\int_{X}\langle \rho (A-A_0),*[F(A_0), \eta ds] \rangle \nonumber\\
	=&2\int_{X} \langle d_{A_0}^+ (\rho (A-A_0)), d_{A_0}(\eta ds)\rangle
	+\int_{X}\langle *[\rho (A-A_0), F(A_0)], \eta ds \rangle \label{2-identity-A12}
\end{align}
where in the last identity we use the assumption on $\eta$ that it vanishes on $ U_\partial \times \Sigma$ to drop the boundary term. 
The identity in \eqref{dA0+-q} and the inequality in \eqref{quad-bound} imply that 
\begin{equation}\label{dAp=A-A}
  |\!|d_{A_0}^+ (\rho (A-A_0))|\!|_{L^{p_1}_{1}(X)}\leq C (|\!|A-A_0|\!|_{L^p_1(X)}^2+1),
\end{equation}
where $p_1=\frac{4p}{8-p}$. Identities \eqref{1-identity-A12} and \eqref{2-identity-A12} and the inequality in \eqref{dAp=A-A} allow us to apply Lemma \ref{elliptic-reg} in the case that $\alpha= \rho (A-A_0)$, $k=1$, $r=p_1$ and the vector field $\sigma$ equals $\frac{\partial}{\partial s}$. This implies that 
\begin{equation} \label{phi-reg}
  |\!|\rho \phi|\!|_{L^{p_1}_{2}(X)}\leq C (|\!|A-A_0|\!|_{L_{1}^p(X)}^2+1).
\end{equation}

In order to improve the regularity of $\psi$, let $\eta\in \Gamma_\nu(D_-(r)\times \Sigma,V)$ where $\Gamma_\nu(D_-(r)\times \Sigma,V)$ is the space of sections of $V$ over $D_-(r)\times \Sigma$ with vanishing normal covariant derivate on $U_\partial \times \Sigma$ with respect to $A_0$. Then
\begin{align}\label{ASD-psi}
	\int_{X} \langle \rho (A-A_0), d_{A_0}^*d_{A_0} (\eta d\theta) \rangle
	=&2\int_{X} \langle \rho (A-A_0), d_{A_0}^*d_{A_0}^+(\eta d\theta)\rangle+
	\int_{X}\langle \rho (A-A_0),*[F(A_0), \eta d\theta] \rangle \nonumber\\
	=&\int_{X}\langle \rho (A-A_0),*[F(A_0), \eta d\theta] \rangle+2\int_{X} \langle d_{A_0}^+ (\rho (A-A_0)), d_{A_0}^+(\eta d\theta)\rangle
	\nonumber\\ 
	&+2\int_{\Sigma\times U_\partial}\tr( \rho (A-A_0)\wedge d_{A_0}^+(\eta d\theta)) \nonumber\\
	=&\int_{X}\langle *[\rho (A-A_0), F(A_0)], \eta ds \rangle+2\int_{X} \langle d_{A_0}^+ (\rho (A-A_0)), d_{A_0}(\eta d\theta)\rangle
	 \nonumber\\
	&+\int_{U_\partial}\int_{\Sigma}\tr( \rho  (\beta-\beta_0)\wedge d_{\beta_0}\eta) d\theta.
\end{align}
By the Stokes theorem and \eqref{B-B0-der-flat}, the last term can be rewritten as
\begin{equation}\label{last-term-bdry}
	-\int_{U_\partial}\int_{\Sigma}\tr( \rho  F(\beta_0)\wedge \eta) d\theta-\int_{U_\partial}\int_{\Sigma}\tr( \rho  (\beta-\beta_0)\wedge(\beta-\beta_0)\wedge \eta) d\theta.
\end{equation}
As in \eqref{quad-bound}, the quadratic term $(\beta-\beta_0)\wedge(\beta-\beta_0)$, regarded as a 2-form on the $4$-manifold $X$, satisfies
\begin{equation}\label{psi-boundary-control}
  \vert\!\vert (\beta-\beta_0)\wedge(\beta-\beta_0)\vert\!\vert_{L^{p_1}_1(X)}\leq C |\!|A-A_0|\!|_{L^p_1(X)}^2.
\end{equation}
Thus, Lemma \ref{elliptic-reg} with $\sigma=\frac{\partial}{\partial \theta}$ and the same $\alpha$, $k$ and $r$ as in the previous case, together with \eqref{dAp=A-A} and \eqref{psi-boundary-control} gives
\begin{align}\label{psi-reg}
  |\!|\rho \psi|\!|_{L^{p_1}_{2}(X)}\leq C (|\!|A-A_0|\!|^2_{L_{1}^p(X)}+1).
\end{align}
Using Sobolev embedding theorem, we may assume that \eqref{phi-reg} and \eqref{psi-reg} hold if the $L^{p_1}_2$ norm on the left hand side is replaced with $L^{q_1}_1$ where $q_1=\frac{2p}{4-p}$.

For each $(s,\theta)\in D_-(r)$, Coulomb gauge condition \eqref{Coulomb-gauge-A0} implies that 
\begin{equation}\label{dbetasa}
	|\!|d_{\beta_0}^*a|\!|_{L^{p_1}_1(\{(s,\theta)\}\times \Sigma)}= |\!|\partial_s^{A_0}\phi+\partial_\theta^{A_0}\psi|\!|_{L^{p_1}_1(\{(s,\theta)\}\times \Sigma)}.
\end{equation}
We have
\[
  d_{A_0}(A-A_0)=d_{\beta_0}a+(d_{\beta_0}\phi-\partial_s^{A_0}a)\wedge ds+(d_{\beta_0}\psi-\partial_\theta^{A_0}a)\wedge d\theta+(\partial_s^{A_0}\psi-\partial_\theta^{A_0}\phi)ds\wedge d\theta,\
\]
where $d_{\beta_0}a$ denotes the exterior derivative of $a$ in the $\Sigma$ direction. This identity can be used to show
\begin{equation}\label{dbetaa}
	|\!|d_{\beta_0}a|\!|_{L^{p_1}_1(\{(s,\theta)\}\times \Sigma)}\leq C(|\!|d_{A_0}^+(A-A_0)|\!|_{L^{p_1}_1(\{(s,\theta)\}\times \Sigma)}+
	|\!|\partial_\theta^{A_0}\phi-\partial_s^{A_0}\psi|\!|_{L^{p_1}_1(\{(s,\theta)\}\times \Sigma)}).
\end{equation}
Therefore, we can use Lemma \ref{weak-Dirac-lemma} and the inequalities \eqref{phi-reg} and \eqref{psi-reg} to show
\begin{align}\label{slice-a-reg}
  |\!|\nabla_{\Sigma}a|\!|_{L^{p_1}_1(D_-(r'')\times \Sigma)}\leq& C (|\!|A-A_0|\!|^2_{L_{1}^p(X)}+1).
\end{align}
We may again assume that the same inequality holds if the $L^{p_1}_1$ norm on the left hand side is replaced with $L^{q_1}$ norm. In particular, $a$ belongs to the Sobolev spaces $L^{p_1}_2(\Sigma,\Lambda^1\otimes L^{p_1}(D_-(r'')))$ and $L^{q_1}_1(\Sigma,\Lambda^1\otimes L^{q_1}(D_-(r'')))$.

Next, we improve regularity of $\partial_s a$, $\partial_\theta a$ and $u$. Define $\fp:D_+(r'')\to  L^p(\Sigma,\Lambda^1\otimes F) \times M$ as
\begin{equation}\label{fp}
  \fp(s,\theta)=(a(-s,\theta), u(s,\theta)).
\end{equation}
Following \eqref{bJ},  the almost complex structures $J$ on $M$ induces an almost complex structure $\bJ$ on $L^p(\Sigma,\Lambda^1\otimes F)  \times M$ given as $(-*_2,J)$. Using the assumption $p>2$ and by decreasing the value of $r$ if necessary, we may assume that $\fp$ takes values in a chart where the pullback of $\bJ$ has the special form given in Lemma \ref{chart-complex-str}. The ASD and CR equations in \eqref{mixed-Q-r} implies that 
\begin{align}
  \(\partial_\theta\fp-\bJ(\fp)\partial_s\fp\)&(s,\theta)=(\(\partial_\theta a-*_2\partial_sa\)(-s,\theta),\(\partial_\theta u-J_{s,\theta}\partial_su\)(s,\theta))\nonumber\\
  &=(\(d_{\beta_0}\psi-*_2 d_{\beta_0}\phi -[\psi+\psi_0,a]+*_2[\phi+\phi_0,a]+F_\theta^+(A_0)\) (-s,\theta),0),\label{CR-op-Banach}
\end{align}
where $\phi_0$ and $\psi_0$ are respectively the components of $A_0$ in the $s$ and $\theta$ directions (that is $A_0=\beta_0+\phi_0 d\theta+\psi_0 d\theta$) and $F_\theta^+(A_0)$ is the projection of $2F^+(A_0)$ to  the summand $F\otimes \Lambda^1(\Sigma)\wedge d\theta$ of $F\otimes \Lambda^2(X)$. Using \eqref{phi-reg}, \eqref{psi-reg} and an application of H\"older inequality analogous to \eqref{quad-bound}, we may conclude that the first entry of \eqref{CR-op-Banach} is in $L^{p_1}_1(D_-(r'')\times \Sigma)$, and hence in $L^{q_1}(D_-(r'')\times \Sigma)$ by Sobolev embedding. Therefore, $ \partial_\theta\fp-\bJ(\fp)\partial_s\fp$ is an element of 
\[
  L^{p_1}_1\(D_-(r''),L^p(\Sigma,\Lambda^1\otimes F)\)\cap L^{q_1}\(D_-(r''),L^{q_1}(\Sigma,\Lambda^1\otimes F)\).
\]
Lemma \ref{chart-complex-str} allows us to apply Proposition \ref{Ban-val-reg-k=0} and conclude that $\fp$ is in $L^{q_1}_1(D_-(r'),L^{q_1}(\Sigma,\Lambda^1\otimes F))$ and 
\begin{equation}\label{pl4}
  \vert\!\vert  \fp\vert\!\vert_{L^{q_1}_1(D_-(r'),L^{q_1}(\Sigma,\Lambda^1\otimes F))} \leq C\( 1+|\!|A-A_0|\!|_{L_{1}^p(X)}^2+|\!|d u|\!|_{L^p(S)}\).
\end{equation}
In summary, $(A,u)$ is in $L^{q_1}_1$, and in fact the $L^{q_1}_1$ norm of the restriction of $(A,u)$ to $D_-(r')\times \Sigma$ can be controlled using the inequalities in \eqref{phi-reg}, \eqref{psi-reg}, \eqref{slice-a-reg} and \eqref{pl4}. By iterating this process, we can prove a similar result where $q_1$ is replaced with $q_i$ of \eqref{def-q-i}. In particular, we can reduce the proof of the regularity to the case that $p>4$.

The rest of the proof of regularity can be addressed in a similar way. For $p>4$, we may obtain \eqref{phi-reg}, \eqref{psi-reg}, \eqref{slice-a-reg} where $p_1$ can be replaced with $p$ because we can use \eqref{quad-bound-higher} with $k=1$ in this case instead of \eqref{quad-bound}. In particular, we obtain
\begin{equation}\label{improved-reg-phi-phi-a-Sigma-step-2}
  \phi,\,\psi\in L^{p}_{2}(D_-(r')\times \Sigma,F),\hspace{1cm}a\in L^{p}_2(\Sigma,\Lambda^1\otimes L^{p}(D_-(r''))).
\end{equation}
In the last step of the above proof where we improve the regularity of $\fp$, we need to use Proposition \ref{Ban-val-reg} instead of Proposition \ref{Ban-val-reg-k=0} to conclude that $\fp$ is in $L^{p/2}_2(D_-(r'),L^{p}(\Sigma,\Lambda^1\otimes F))$. This in addition to \eqref{improved-reg-phi-phi-a-Sigma-step-2} implies that $(A,u)$ is in $L^{p/2}_2$ (see \eqref{fun-space-Ban}). Thus the proof of regularity is reduced to the show that if $(A,u)$ is in $L^p_k$ with $p>2$ and $k\geq 2$, then $(A,u)$ is in $L^p_{k+1}$. The proof of this claim follows the same strategy. Following the first three steps of the above proof, we obtain 
\begin{equation}\label{improved-reg-phi-phi-a-Sigma-step-2}
  \phi,\,\psi\in L^{p}_{k+1}(D_-(r')\times \Sigma,F),\hspace{1cm}a\in L^{p}_{k+1}(\Sigma,\Lambda^1\otimes L^{p}(D_-(r''))).
\end{equation}
In the last step of the proof, Proposition  \ref{Ban-val-reg} allows us to conclude that $\fp\in L^{p}_2(D_-(r'),L^{p}(\Sigma,\Lambda^1\otimes F))$. This complete the proof of smoothness of $(A,u)$. In each step of the proof, we can bound the given Sobolev norm of $(A,u)$ over any region $D_-(r')\times \Sigma$ with $r'<r$ using a polynomial function of $|\!|A-A_0|\!|_{L_{1}^p(X)}$ and $|\!|d u|\!|_{L^p(S)}$ where the coefficients of this polynomial depend only on $A_0$ and $r'$.

\subsection{Proof of Theorem \ref{regularity-sequential}}\label{regularity-thm-sequential-sub}

The proof of Theorem \ref{regularity-sequential} can be verified with a similar argument as in the previous section. Given a sequence $(A_i,u_i)$ as in the statement of Theorem \ref{regularity-sequentia-detailed}, let 
\[
  A_i-A_0=a_i+\phi_i ds+\psi_i d\theta.
\]
The instances of \eqref{dA0+-q} for $A_i$ and $A_j$ imply that 
\begin{equation}\label{dA0+-q-ij}
	d_{A_0}^+(A_i-A_j)=Q(A_i-A_0)-Q(A_j-A_0)
\end{equation}
As an analogue of \eqref{quad-bound} and \eqref{quad-bound-higher}, we have
\begin{equation}\label{quad-bound-ij}
  |\!|Q(A_i-A_0)-Q(A_j-A_0)|\!|_{L^{\frac{4q}{8-q}}_1(X)} \leq C |\!|A_i-A_j|\!|_{L^q_1(X)}(|\!|A_i-A_0|\!|_{L^q_1(X)}+|\!|A_j-A_0|\!|_{L^q_1(X)})
\end{equation}
for $q< 4$, and 
\begin{equation}\label{quad-bound-higher-ij}
  |\!|Q(A_i-A_j)|\!|_{L^q_k(X)} \leq C |\!|A_i-A_j|\!|_{L^q_k(X)}(|\!|A_i-A_0|\!|_{L^q_k(X)}+|\!|A_j-A_0|\!|_{L^q_k(X)})
\end{equation}
when $qk>4$ and $k$ is a positive integer. Similarly, if $\beta_i(s,t)$ denotes the restriction of $A$ $\Sigma\times\{(s,\theta)\}$, then we have 
\begin{align}
  d_{\beta_0(0,\theta)}\(\beta_i(0,\theta)-\beta_j(0,\theta)\) =&(\beta_j(0,\theta) -\beta_0(0,\theta) )\wedge (\beta_j(0,\theta) -\beta_0(0,\theta))\nonumber\\
  &-(\beta_i(0,\theta) -\beta_0(0,\theta) )\wedge (\beta_i(0,\theta) -\beta_0(0,\theta)). \label{B-B0-der-flat-ij}
  \end{align}

Now, by following the steps of the previous section and replacing \eqref{dA0+-q}, \eqref{B-B0-der-flat}, \eqref{quad-bound} and \eqref{quad-bound-higher} with the above identities and inequalities, we can inductively show that the $L^p_1$ convergence of $(A_i,u_i)$ can be improved to higher regularities. As the starting point, \eqref{1-identity-A12} implies that for any $\xi\in \Gamma_c(D_-(r)\times \Sigma,V)$ we have 
\begin{equation}
  \int_{X} \langle \rho (A_i-A_j), d_{A_0}\xi \rangle=\int_{X}  \langle\nabla\rho\cdot (A_i-A_j), \xi\rangle,\label{1-identity-A12-ij}
\end{equation} 
\eqref{2-identity-A12} implies that for any $\eta\in\Gamma_\tau(D_-(r)\times \Sigma,V)$ we have 
\begin{align}
	\int_{X} \langle \rho (A_i-A_j), d_{A_0}^*&d_{A_0} (\eta ds) \rangle\nonumber\\
	=&2\int_{X} \langle d_{A_0}^+ (\rho (A_i-A_j)), d_{A_0}(\eta ds)\rangle
	+\int_{X}\langle *[\rho (A_i-A_j), F(A_0)], \eta ds \rangle, \label{2-identity-A12-ij}
\end{align}
and \eqref{ASD-psi} and \eqref{last-term-bdry} imply that for any $\eta\in \Gamma_\nu(D_-(r)\times \Sigma,V)$ we have
\begin{align}
	\int_{X} \langle \rho (A_i-A_j), &d_{A_0}^*d_{A_0} (\eta d\theta) \rangle
	=\int_{X}\langle *[\rho (A_i-A_j), F(A_0)], \eta ds \rangle+2\int_{X} \langle d_{A_0}^+ (\rho (A_i-A_j)), d_{A_0}(\eta d\theta)\rangle
	 \nonumber\\
	&-\int_{U_\partial}\int_{\Sigma}\tr( \rho  \((\beta_i-\beta_0)\wedge(\beta_i-\beta_0)-(\beta_j-\beta_0)\wedge(\beta_j-\beta_0)\)\wedge d_{\beta_0}\eta) d\theta.\label{ASD-psi-ij}
\end{align}
From these identities we obtain
\begin{equation} \label{phi-reg-ij}
  |\!|\rho (\phi_i-\phi_j)|\!|_{L^{p_1}_{2}(X)}\leq C |\!|A_i-A_j|\!|_{L^p_1(X)}(|\!|A_i-A_0|\!|_{L^p_1(X)}+|\!|A_j-A_0|\!|_{L^p_1(X)}+1),
\end{equation}
and 
\begin{equation} \label{psi-reg-ij}
	|\!|\rho (\psi_i-\psi_j)|\!|_{L^{p_1}_{2}(X)}\leq C |\!|A_i-A_j|\!|_{L^p_1(X)}(|\!|A_i-A_0|\!|_{L^p_1(X)}+|\!|A_j-A_0|\!|_{L^p_1(X)}+1).
\end{equation}

Similar inequalities can be obtained for the terms $a_i-a_j$ and the distance between $u_i$ and $u_j$. First note that we have
\begin{equation}\label{dbetasa-ii}
	|\!|d_{\beta_0}^*(a_i-a_j)|\!|_{L^{p_1}_1(\{(s,\theta)\}\times \Sigma)}= |\!|\partial_s^{A_0}(\phi_i-\phi_j)+\partial_\theta^{A_0}(\psi_i-\psi_j)|\!|_{L^{p_1}_1(\{(s,\theta)\}\times \Sigma)}.
\end{equation}
and 
\begin{align}
	|\!|d_{\beta_0}(a_i-a_j)|\!|_{L^{p_1}_1(\{(s,\theta)\}\times \Sigma)}\leq C\left(|\!|d_{A_0}^+(\right.&A_i-A_j)|\!|_{L^{p_1}_1(\{(s,\theta)\}\times \Sigma)}+\nonumber\\
	&\left.+|\!|\partial_\theta^{A_0}(\phi_i-\phi_j)-\partial_s^{A_0}(\psi_i-\psi_j)|\!|_{L^{p_1}_1(\{(s,\theta)\}\times \Sigma)}\right).\label{dbetaa-ij}
\end{align}
as the counterparts of \eqref{dbetasa} and \eqref{dbetaa}. Thus we obtain the following inequality analogous to \eqref{slice-a-reg}:
\begin{align}\label{slice-a-reg-ii}
  |\!|\nabla_{\Sigma}(a_i-a_j)|\!|_{L^{p_1}_1(D_-(r'')\times \Sigma)}\leq& C |\!|A_i-A_j|\!|_{L^p_1(X)}(|\!|A_i-A_0|\!|_{L^p_1(X)}+|\!|A_j-A_0|\!|_{L^p_1(X)}+1).
\end{align}
We define $\fp_i$ using $a_i$ and $u_i$ as in \eqref{fp}. Since $p>2$, the maps $\fp_i$ are $C^0$ convergent to $\fp$ associated to $(A,u)$. Thus by decreasing the value of $r$ if necessary and for large enough values of $i$, the map $\fp_i$ takes values in a chart where the pullback of $\bJ$ has the special form given in Lemma \ref{chart-complex-str}. We may simplify $\partial_\theta\fp_i-\bJ(\fp_i)\partial_s\fp_i$ as in \eqref{CR-op-Banach}. In particular, a similar argument as in the previous steps can be used to show that the difference between the $L^p(\Sigma,\Lambda^1\otimes F)$-coordinates of $\partial_\theta\fp_i-\bJ(\fp_i)\partial_s\fp_i$ and $\partial_\theta\fp_j-\bJ(\fp_j)\partial_s\fp_j$ is given by an $L^{p_1}_1$ section of the bundle $\Lambda^1\otimes F$ over $D_-(r'')\times \Sigma$ whose $L^{p_1}_1$ norm is bounded by the same term as in the right hand side of \eqref{phi-reg-ij} if $C$ is chosen appropriately. By applying Proposition \ref{Ban-val-reg-k=0} to the sequence $\fp_i':=\Phi_p^{-1}\circ \fp_i$ with $\Phi_p$ being given by Lemma \ref{chart-complex-str}, we conclude that $\fp_i$ is convergent to $\fp$ as elements of $L^{q_1}_1(D_-(r'),L^{q_1}(\Sigma,\Lambda^1\otimes F)$. Combining our results we conclude that the restriction of $A_i$ (resp. $u_i$) to $D_-(r')\times \Sigma$ (resp. $D_+(r')$) for any $r'<r$ is $L^{q_1}_1$ convergent to $A$ (resp. $u$). Analogous to the proof of Theorem \ref{regularity-thm-sub}, iterating this argument allows us to show that $A_i$ (resp. $u_i$) to $D_-(r')\times \Sigma$ (resp. $D_+(r')$) for any $r'<r$ is $L^{p}_k$ convergent to $A$ (resp. $u$) for any $p$ and positive integer $k$. This completes the proof of Theorem \ref{regularity-thm-sequential-sub}.

\section{Compactness}\label{compactness-sec}
In this section, we study compactness properties of the moduli space of mixed solutions. We specialize to the case that our target symplectic manifold for the mixed equation is the moduli space $\mathcal M(\Sigma,F)$ of flat connections on $F$ and consider the quintuples of the form 
\begin{equation}\label{quintuple-comp}
  \fP(r):=(D_-(r)\times \Sigma,D_-(r)\times F,D_+(r),\mathcal M(\Sigma,F),\mathcal L(\Sigma,F)).
\end{equation}
Recall that the Lagrangian correspondence $\mathcal L(\Sigma,F)$ is given in Example \ref{matching-lag-cor}. We fix a positive real number $r_0$ and we drop $r$ from our notation if $r=r_0$ in the following. In this section $C$ is a constant depending only on $r_0$ and the metric $g$ on $\Sigma$ such that its value might increase from each line to the next line.

\subsection{Energy quantization}\label{energy-quant=subsec}

Suppose $(A,u)$ is an element of the configuration space associated to the quintuple $\fP(r_0)$. Let the {\it energy density of $u$}, denoted by $e_u:D_-\to \R^{\geq 0}$, and the {\it $2$-dimensional energy density of $A$}, denoted by $e_{A}:D_-\to \R^{\geq 0}$, be defined as
\[
  e_u(s,\theta)=|du|^2(-s,\theta),\hspace{1cm}e_A(s,\theta)=\int_{\{(s,\theta)\}\times \Sigma}|F_A|^2,
\]
where $(s,\theta)\in D_-$. The {energy density} of $(A,u)$, denoted by $e_{A,u}:D_-\to \R^{\geq 0}$ is defined as
\[
  e_{A,u}=e_A+e_u.
\]

\begin{theorem}\label{en-quant}
	There exist constants $\kappa$ and $\hbar$ such that the following holds.  
	Let $(A,u)$ be a solution of the mixed equation associated to the quintuple $\fP(r_0)$.
	Let $z\in U_{-}$ and $r$ be given such that $D_r(z):=B_r(z)\cap \bbH_- \subset B_{r_0}(0)$. Let also
	\[
	  \int_{D_r(z)}e_{A,u}\dvol\leq \hbar.
	\]
	Then we have:
	\begin{equation}
		\hspace{1cm} e_{A,u}(z) \leq \kappa \frac{ \int_{D_r(z)}e_{A,u}\dvol}{r^2}.
	\end{equation}
\end{theorem}

The following proposition allows us to obtain interior regularity for the the energy densities of a solution $(A,u)$ of the mixed solution associated to the quintuple $\fP(r_0)$. 
\begin{prop}\label{lap-estimates}
	A solution $(A,u)$ of the mixed equation associated to the quintuple $\fP(r_0)$ satisfies the inequalities
	\begin{equation}\label{lap-pure-energies}
	  \Delta_4(|F_A|^2)\leq C(|F_A|^2+|F_A|^3), \hspace{2cm}\Delta_2(e_u)\leq C(e_u+e_u^2),
	\end{equation}
	\begin{equation}\label{lap-mixed-energy}
	  \Delta_2(e_{A,u})\leq C(e_{A,u}+e_{A,u}^2+e_{A,u}^{\frac{1}{2}}f_A),
	\end{equation}	
	where $\Delta_2$ is the Laplacian on $D_-$, $\Delta_4$ is the Laplace-Beltrami operator associated to $D_-\times \Sigma$, and $f_A:D_-\to \R^{\geq 0}$ is given by
	\begin{equation}\label{fA}
	  f_A(z)=(\int_{\{z\} \times \Sigma}|F_A|^4\,\dvol)^{\frac{1}{2}}.
	\end{equation}
\end{prop}
\begin{proof}
	Let $X$ be a Riemannian 4-manifold and $E$ be a vector bundle over $X$.
	Let $A$ be a unitary connection on $E$ and $\phi$ be a $2$-form with values in $E$. Then the Weitzenb\"ock
	formula states that
	\begin{equation} \label{Weitzenbock}
	  \nabla_A^*\nabla_A \phi-(d_Ad_A^*+d_A^*d_A) \phi=Q_1(R_X,\phi)+Q_2(F_A,\phi),
	\end{equation}
	where $Q_1(R_X,\phi)$ (respectively, $Q_2(F_A,\phi)$) denotes a point-wise smooth bi-linear form of the Riemannian curvature $R_X$
	(respectively, the curvature of the connection $A$) and $\phi$. (See \cite[Theorem 3.2]{BL:SI-YM} for more details.)
	In particular, if we apply this identity to the case that $\phi$ is equal to the curvature $F_A$ of an ASD connection $A$, then Bianchi identity implies that:
	\begin{equation} \label{Weit}
	  \nabla_A^*\nabla_A F_A=Q_1(R_X,F_A)+Q_2(F_A,F_A)
	\end{equation}
	By taking the inner product of \eqref{Weit} with $F_A$, we can conclude that:
	\begin{align}\label{energy-connection}
		\Delta_4 |F_A|^2+2|\nabla_A F_A|^2&=2\langle \nabla_A^*\nabla_A F_A, F_A\rangle \nonumber \\
		&\leq C( |F_A|^2+|F_A|^3),
	\end{align}
	This implies the first claimed inequality of \eqref{lap-pure-energies}.
	
	An analogue of \eqref{Weitzenbock} holds for 1-forms on a Riemannian manifold $X$ for appropriate choices of $Q_1$, $Q_2$ (see, for example,  \cite[Remark 6.40]{ABKL:symp}), and 
	the second inequality in \eqref{lap-pure-energies} can be also proved using this Weitzenb\"ock identity, as we explain next. 
	The differential $du$ of the holomorphic map $u$ can be 
	regarded as a 1-form on $D_+$ with values in the bundle $u^*T\mathcal M(\Sigma,F)$. 
	Let $\nabla$ denote the Levi-Civita connection associated to the metric on $\mathcal M(\Sigma,F)$ given as $\omega_{\rm fl}(\cdot , J_*\cdot)$ by the symplectic form $\omega_{\rm fl}$ and the complex structure $J_*$.
	It is useful to consider the $J_*$-linear connection on $T\mathcal M(\Sigma,F)$
	\begin{equation}\label{Jtilde}
		\widetilde \nabla(v):=\nabla(v)-\frac{1}{2}J_*(\nabla J_*)v
	\end{equation}	
	which is compatible with the metric, and its torsion has vanishing $(1,1)$-component. 
	Therefore, if we let $B$ be the pull back of this connection on $u^*T\mathcal M(\Sigma,F)$, then $d_B(du)=0$. 
	Since $u$ satisfies the Cauchy-Riemann equation,
	it is also straightforward to check that $d^*_B(du)=0$. We apply the 1-form version of \eqref{Weitzenbock} to $\phi=du$ and the 
	connection $B$ on $u^*T\mathcal M(\Sigma,F)$. As in the previous case, taking the inner product of 
	the resulting identity with $du$ implies that:
	\begin{align}\label{energy-string}
		\Delta_2 e_u+2|\nabla_B du|^2&=2\langle \nabla_B^*\nabla_B du, du\rangle \nonumber \\
		&=\langle Q_1(R_{D_+},du),du \rangle+\langle Q_2(F_B,du),du \rangle \nonumber \\
		&\leq  C(e_u+ e_u^2) 
	\end{align}
	Note that in the last inequality we use the fact that the norm of $F_B$ can be controlled by $|du|^2$.
	
	Consider the function on $D_-$ that associates to $z\in D_-$ the integral of $|F_A|^2$ over $\{z\}\times \Sigma$. Inequality \eqref{energy-connection} and Cauchy-Schwarz inequality implies that the Laplacian of this function satisfies
	\begin{align*}
	  \Delta_2 (\int_{\{z\}\times \Sigma} |F_A|^2)&=\int_{\{z\}\times \Sigma} \Delta_4|F_A|^2+\int_{\{z\}\times \Sigma}d_\Sigma (*d_\Sigma|F_A|^2)\\
	  &\leq C(\int_{\{z\}\times \Sigma} |F_A|^2+|F_A|^3)  \\
	  &\leq C\(\int_{\{z\}\times \Sigma} |F_A|^2+(\int_{\{z\}\times \Sigma} |F_A|^2)^{\frac{1}{2}}f_A(z)\).
	\end{align*}
	Note that the Stokes' theorem implies that the second integral on the right hand side of the first line vanishes. This inequality together with \eqref{energy-string} verifies the final inequality of the proposition.
\end{proof}

\begin{prop} \label{normal-est}
	For any point $z:=(0,\theta)\in U_\partial$, the normal derivative of the mixed energy density satisfies the following inequality:
	\begin{equation}\label{normal-bdry-est}
	  \partial_s e_{A,u}(z)\leq Ce_{A,u}(z)^{\frac{3}{2}}
	\end{equation}
\end{prop}
\begin{proof}
	First we pick an appropriate gauge for the connection $A$.
	Decompose the connection $A$ as follows
	\[A=\beta(s,\theta)+\phi(s,\theta) ds+\psi(s,\theta) d\theta\] 
	where $\beta(s,\theta)$ 
	is a connection on $F$ over $\Sigma$ and $\phi(s,\theta)$, $\psi(s,\theta)$ are sections of $F$.
	Fix a gauge for $A$ by firstly taking the parallel transport of a fixed frame at the point $(0,0)$ along $U_\partial$, and then extending the frames on $U_\partial$ to $D_-$ 
	by parallel transport in the $s$-direction. Therefore, $\phi$ and the restriction of $\psi$ to $U_\partial$ vanish. The ASD equation for the connection $A$ implies that $\beta$ and $\psi$ satisfy
	\begin{equation}\label{ASD-renterp}
	  *_{\Sigma}F_\beta+\partial_s\psi=0,\hspace{1cm} -\partial_\theta \beta+*_{\Sigma}\partial_s \beta+d_\beta\psi=0,
	\end{equation}
	where $\partial_s$ and $\partial_\theta$ are defined with respect to the chosen frame. 
	Since $F_\beta=0$ on $U_{\partial}$, the first equation in \eqref{ASD-renterp} implies that $\partial_s \psi$ on the matching line $U_{\partial}$ vanishes. Using this and the second equation in 
	\eqref{ASD-renterp}, we can conclude that:
	\begin{equation}\label{ASD-renterp-imp}
		 \hspace{2cm}\partial_{\theta} \beta=*_{\Sigma}\partial_s\beta,\hspace{1cm}\partial_\theta\partial_{\theta} \beta=*_{\Sigma}\partial_\theta\partial_s\beta
		 =-\partial_s\partial_s\beta
		  \hspace{1cm} \forall z\in U_{\partial}.
	\end{equation}	
	
	The curvature of the connection $A$ with respect to the above gauge has the following form:
	\[
	  F_A=F_\beta+ds\wedge \partial_s\beta+d\theta\wedge \partial_\theta \beta+d_\beta\psi\wedge d\theta+\partial_s \psi ds\wedge d\theta.
	\]
	Thus we have
	\begin{align*}
		\frac{1}{2}\partial_s e_A(0,\theta)
		&=\int_{\{(0,\theta)\}\times \Sigma}\langle \partial_s\beta,\partial_s\partial_s\beta\rangle+
		\langle \partial_\theta \beta,\partial_s\partial_\theta \beta\rangle\\
		&=-2\int_{\{(0,\theta)\}\times \Sigma}{\tr}(\partial_\theta \beta\wedge \partial_\theta\partial_\theta \beta),
	\end{align*}
	where the last identity follows from \eqref{ASD-renterp-imp}. 
	
	We follow a similar strategy to fix a representative for the map $u:D_+\to \mathcal M(\Sigma,F)$. For any $(0,\theta)\in U_\partial$, there is a unique connection $\beta'(0,\theta)$ such that 
	$\beta'(0,0)=\beta(0,0)$, $\beta'(0,\theta)$ represents the flat connection $u(0,\theta)$ and
	$d^*_{\beta'(0,\theta)}\partial_\theta \beta'(0,\theta)=0$. We extend this family of connections to $D_+$ by requiring that $\beta'(s,\theta)$ represents the flat connection $u(s,\theta)$ and 
	$d^*_{\beta'(s,\theta)}\partial_s \beta'(s,\theta)=0$. Since $u$ is a holomorphic map, for each $(s,\theta)\in D_+$, there is a section $\psi'(s,\theta)$ of $F$ such that
	\begin{equation}\label{CR-renterp}
	  -\partial_\theta \beta'+*_{\Sigma}\partial_s \beta'+d_{\beta'}\psi'=0.
	\end{equation}	
	In particular, $d_{\beta'}^*d_{\beta'}\psi'=0$ on $U_\partial$, which implies that $\psi'(0,\theta)=0$. Taking the derivative of \eqref{CR-renterp} along the $\theta$-direction on the matching line $U_\partial$
	implies that
	\begin{equation}\label{CR-renterp-imp}
	  \hspace{2cm}\partial_\theta\partial_\theta \beta'=*_{\Sigma}\partial_\theta\partial_s \beta'\hspace{1cm} \forall z\in U_{\partial}.
	\end{equation}	
	
	For $(0,\theta)\in U_\partial$, the exterior derivatives $d_{\beta'(0,\theta)}$ and $d^*_{\beta'(0,\theta)}$ act trivially on $\partial_\theta \beta'(0,\theta)$, and hence we have
	\[
	  |d u|^2(0,\theta)=2\int_{\{(0,\theta)\}\times \Sigma}|\partial_\theta \beta'(0,\theta)|^2.
	\]
	This together with $d^*_{\beta'(0,\theta)}\partial_\theta \beta'(0,\theta)=0$ gives rise to the following identity for the normal derivative of $e_u:D_-\to \R$ on $U_\partial$:
	\begin{align*}
		\frac{1}{2}\partial_s e_u(0,\theta)
		&=-2\int_{\{(0,\theta)\}\times \Sigma}\langle \partial_s\partial_\theta \beta',\partial_\theta \beta'\rangle\\
		&=2\int_{\{(0,\theta)\}\times \Sigma}{\tr}(\partial_\theta \beta'\wedge \partial_\theta\partial_\theta \beta').
	\end{align*}
	
	The matching condition on $U_\partial$ implies that 
	there is $g_\theta\in \mathcal G(\Sigma,F)$
	such that $\beta'(0,\theta)=g^*_\theta \beta(0,\theta)$ for each $\theta$. Moreover, $g_\theta$ is smooth with respect to $\theta$ and $g_0=1$.
	Let $\zeta_\theta:=g_\theta^{-1}\partial_\theta g_\theta$. Then we have
	\begin{equation}\label{B-B'-zeta}
	  \partial_\theta \beta(0,\theta)=g_\theta \partial_\theta \beta'(0,\theta)g_\theta^{-1}-g_\theta d_{\beta'(0,\theta)}\zeta_\theta g_\theta^{-1}
	\end{equation} 
	Using the extension theorem of Sobolev spaces, we may find $\widetilde \zeta_\theta\in L^2_{\frac{3}{2}}([-1,1]\times \Sigma,[-1,1]\times F)$ 
	such that
	\[
	  \widetilde \zeta_\theta\vert_{\{-1\}\times \Sigma}=0,\hspace{1cm}\widetilde \zeta_\theta\vert_{\{1\}\times \Sigma}=\zeta_\theta,
	\]
	and
	\[
	 |\!|\widetilde \zeta_\theta|\!|_{L^2_{\frac{3}{2}}([-1,1]\times \Sigma,F)}\leq 
	 C |\!|\zeta_\theta|\!|_{L^2_{1}(\Sigma,F)}.
	\]
	Define $\widetilde g_\theta\in \mathcal G([-1,1]\times \Sigma,[-1,1]\times F)$ by $\partial_\theta \widetilde g_\theta=\widetilde g_\theta \widetilde \zeta_\theta$. 
	Then $\widetilde g_\theta\vert_{\{-1\}\times \Sigma}=1$ and $\widetilde g_\theta\vert_{\{1\}\times \Sigma}=g_\theta$.

	Let  $\widetilde B_\theta$ be the connection on $[-1,1]\times \Sigma$ defined as $\widetilde g_\theta^*\beta(0,\theta)$. This connection restricts to $\beta(0,\theta)$, $\beta'(0,\theta)$ on $\{-1\}\times \Sigma$, $\{1\}\times \Sigma$. 
	Since $\widetilde B_\theta$ is flat for each $\theta$, we have:
	\begin{equation}\label{der-A-tilde}
		d_{\widetilde B_\theta}(\partial_\theta \widetilde B_\theta)=0\hspace{1cm}
		d_{\widetilde B_{\theta}}(\partial_\theta\partial_\theta \widetilde B_\theta)=
		-2\partial_\theta \widetilde B_\theta\wedge \partial_\theta \widetilde B_\theta
	\end{equation}
	Here the second identity is obtained by applying $\partial_\theta$ to the first one.
	Stokes theorem and the identities in \eqref{der-A-tilde} imply that we have the following identities for each $\theta$
	\[\int_{\{(0,\theta)\}\times \Sigma}{\tr}(\partial_\theta \beta'\wedge \partial_\theta\partial_\theta \beta')
	-\int_{\{(0,\theta)\}\times \Sigma}{\tr}(\partial_\theta \beta\wedge \partial_\theta\partial_\theta \beta)=\hspace{4cm}\]
	\vspace{-15pt}
	\begin{align*}
		\hspace{4cm}
		&=\int_{[-1,1]\times \Sigma}d{\tr}(\partial_\theta \widetilde B_\theta \wedge \partial_\theta\partial_\theta \widetilde B_\theta )\\
		&=2\int_{[-1,1]\times \Sigma}{\tr}(\partial_\theta \widetilde B_\theta 
		\wedge \partial_\theta \widetilde B_\theta \wedge \partial_\theta \widetilde B_\theta)
	\end{align*}
	Thus for any $(0,\theta)\in U_\partial$, we have:
	\[\partial_s e_{A,u}(0,\theta)\leq C |\!|\partial_\theta \widetilde B_\theta|\!|_{L^3([-1,1]\times \Sigma)}^3 \]
	
	Using the definition of $\widetilde B_\theta$ we can conclude that
	\begin{align}
		|\!|\partial_\theta \widetilde B_\theta|\!|_{L^3([-1,1]\times \Sigma)}
		&\leq C\(|\!|\partial_\theta \beta'(0,\theta)|\!|_{L^3(\Sigma)}+
		|\!|d_{\widetilde B_\theta}\widetilde \zeta_\theta|\!|_{L^3([-1,1]\times \Sigma)}\)\nonumber\\
		&\leq C\(|\!|\partial_\theta \beta'(0,\theta)|\!|_{L^2(\Sigma)}+
		|\!|d_{\widetilde B_\theta}\widetilde \zeta_\theta|\!|_{L^2_{\frac{1}{2}}([-1,1]\times \Sigma)}\)\nonumber\\
		&\leq C\(e_{u}(0,\theta)^{\frac{1}{2}}+
		|\!|\widetilde \zeta_\theta|\!|_{L^2_{\frac{3}{2}}([-1,1]\times \Sigma)}\).\label{bound-L3-A-tilde}
	\end{align}
	In addition to Sobolev embedding inequality, we use the fact that 
	$\partial_\theta \beta'(0,\theta)$ belongs to the kernel of $d_{\beta'(0,\theta)}$ and $d^*_{\beta'(0,\theta)}$
	 to obtain the second inequality.
	Our choice of $\widetilde \zeta_\theta$ allows us to conclude that its $L^2_{\frac{3}{2}}$ norm is bounded by 
	$C|\!|\zeta_\theta|\!|_{L^2_{1}(\Sigma)}$, which in turn is bounded by $C|\!|d_{\beta'(0,\theta)}\zeta_\theta|\!|_{L^2(\Sigma)}$.
	The last claim holds because the kernel of $d_{\beta'(0,\theta)}$ acting on the space of $0$-forms is trivial, and 
	$\beta'(0,\theta)$ is a representative for an element of the compact space $\mathcal M(\Sigma,F)$. Since \eqref{B-B'-zeta} implies that 
	$|\!|d_{\beta'(0,\theta)}\zeta_\theta|\!|_{L^2(\Sigma)}$ is controlled  by 
	$|\!|\partial_{\theta}\beta(0,\theta)|\!|_{L^2(\Sigma)}+|\!|\partial_{\theta}\beta'(0,\theta)|\!|_{L^2(\Sigma)}$, we conclude that
	the $L^2_{\frac{3}{2}}$ norm of $\widetilde \zeta_\theta$ is bounded by $Ce_{A,u}(0,\theta)^{\frac{1}{2}}$. Therefore, 
	this observation and \eqref{bound-L3-A-tilde} give us the 
	inequality \eqref{normal-bdry-est}.
\end{proof}

The following proposition is a weaker version of Theorem \ref{en-quant}.
\begin{prop} \label{en-quant-weak}
	There exist constants $\kappa'$ and $\hbar'$ such that the following holds. Let $(A,u)$ be a solution of the mixed equation 
	associated to the quintuple $\fP(r_0)$.
	Let $z\in D_{-}$ and $r$ be given such that $D_{r}(z)\subset B_{r_0}(0)$, and
	\begin{equation}\label{ptwise-bd}
	  \hspace{2cm} e_{A,u}(w)\leq \hbar' r^{-2} \hspace{1cm}\forall w\in D_{r}(z).
	\end{equation}
	Then we have
	\begin{equation}\label{Main-ineq-weak}
		\hspace{1cm} e_{A,u}(z) \leq \kappa' \frac{ \int_{D_{r}(z)}e_{A,u}\dvol}{r^2}.
	\end{equation}
\end{prop}

Before delving into the proof of Proposition \ref{en-quant-weak}, we show that the assumption of this proposition allows us to obtain appropriate $L^2_1$ bounds on $F_A$ and $du$:
\begin{lemma}\label{L21-energy}
	There is a constant $\hbar_0$ such that the following holds. 
	Suppose $(A,u)$ is a solution of the mixed equation 
	associated to the quintuple $\fP(r_0)$, and $z\in D_{-}$ and $r$ are given
	such that $D_{r}(z)\subset B_{r_0}(0)$ and \eqref{ptwise-bd} holds for $\hbar'=\hbar_0$. Then
	\begin{equation}\label{L21-bd-energy}
		|\!|\nabla_AF_A|\!|_{L^2(D_{\frac{r}{2}(z)}\times \Sigma)}^2+|\!|\nabla_B (d u)|\!|_{L^2(D^+_{\frac{r}{2}}(z))}^2
		\leq C \frac{ \int_{D_{r}(z)}e_{A,u}\dvol}{r^2}.
	\end{equation}	
	where $B$ is the connection introduced in the proof of Proposition \ref{lap-estimates} and $D^+_{\frac{r}{2}}(z)$ denotes the reflection of $D_{\frac{r}{2}}(z)$ with respect to $U_\partial$.
\end{lemma}
\begin{proof}
	Fix a smooth function on $\rho:\C\to \R$ which is supported in $B_1(0)$ and its value on $B_{\frac{1}{2}}(0)$ is equal to
	$1$. We also define $\rho_r(w):=\rho(\vert\frac{w-z}{r}\vert)$. We have
	\begin{align}
		|\!|\nabla_A(\rho_rF_A)|\!|_{L^2(D_{r}(z)\times \Sigma)}^2
		&=\int_{D_{r}(z)\times \Sigma} \langle d\rho_r \otimes  F_A,d\rho_r\otimes F_A)\rangle+
		\langle \nabla_A (F_A),\nabla_A (\rho_r^2 F_A)\rangle\nonumber\\
		& \leq Cr^{-2} |\!|F_A|\!|_{L^2(D_{r}(z)\times \Sigma)}^2+
		\int_{D_{r}(z)\times \Sigma}\langle \nabla_A^*\nabla_A F_A, \rho_r^2 F_A\rangle\nonumber\\
		&+\int_{D^\partial_{r}(z)} \rho_r^2  \int_{\Sigma}\langle
		 (\nabla_A)_{\partial s} F_A, F_A\rangle\label{A-der-1}.
	\end{align}	
	Here $D^\partial_{r}(z)$ denotes  $D_{r}(z)\cap U_\partial$. Using the inequality in \eqref{energy-connection}, 
	the point-wise assumption \eqref{ptwise-bd}, Cauchy-Schwarz
	and Sobolev embedding theorem we have:
	\begin{align}
		\int_{D_{r}(z)\times \Sigma}\langle \nabla_A^*\nabla_A F_A, \rho_r^2 F_A\rangle&
		\leq C\int_{D_{r}(z)\times \Sigma} \rho_r^2|F_A|^2+\rho_r^2|F_A|^3\nonumber\\
		 &\leq C(|\!|F_A|\!|_{L^2(D_{r}(z)\times \Sigma)}^2+
		(\int_{D_{r}(z)\times \Sigma} |F_A|^2)^{\frac{1}{2}}(\int_{D_{r}(z)\times \Sigma} 
		|\rho_rF_A|^4)^{\frac{1}{2}})\nonumber\\		 
		&\leq C\(
		|\!|F_A|\!|_{L^2(D_{r}(z)\times \Sigma)}^2+
		\hbar_0
		|\!|\nabla_A(\rho_rF_A)|\!|_{L^2(D_{r}(z)\times \Sigma)}^{2}\)\nonumber.
	\end{align}
	Combining the above inequality and \eqref{A-der-1}, we obtain:
\begin{align}
		|\!|\nabla_A(\rho_rF_A)|\!|_{L^2(D_{r}(z)\times \Sigma)}^2
		& \leq C((r^{-2}+1) |\!|F_A|\!|_{L^2(D_{r}(z)\times \Sigma)}^2+
		\hbar_0|\!|\nabla_A(\rho_rF_A)|\!|_{L^2(D_{r}(z)\times \Sigma)}^{2}\nonumber\\
		&+\int_{D^\partial_{r}(z)} \rho_r^2  \int_{\Sigma}\langle
		 (\nabla_A)_{\partial s} F_A, F_A\rangle)\nonumber
	\end{align}
	If the constant $\hbar_0$ is small enough, then we can absorb the term containing 
	$|\!|\nabla(\rho_rF_A)|\!|_{L^2(D_{2r}(z)\times \Sigma)}^{2}$ on the right hand side and obtain
	\begin{equation}\label{e-A-der}
		|\!|\nabla_A(\rho_rF_A)|\!|_{L^2(D_r(z)\times \Sigma)}^2\leq  
		Cr^{-2} |\!|F_A|\!|_{L^2(D_{r}(z)\times \Sigma)}^2+
		\int_{D^\partial_{r}(z)} \rho_r^2  \int_{\Sigma}\langle
		 (\nabla_A)_{\partial s} F_A, F_A\rangle.
	\end{equation}
	We follow a similar strategy to bound $|\!|\nabla_B (d u)|\!|_{L^2(D^+_{\frac{r}{2}}(z))}$. 
	\begin{align}
		|\!|\nabla_B(\rho_rd u)&|\!|_{L^2(D^+_{r}(z))}^2
		=\int_{D^+_{r}(z)\times \Sigma} \langle d\rho_r \otimes d u,d\rho_r\otimes d u)\rangle+
		\langle \nabla_Bd u,\nabla_B(\rho_r^2d u)\rangle\nonumber\\
		& \leq Cr^{-2} |\!|du|\!|_{L^2(D^+_{r}(z))}^2+
		\int_{D^+_{r}(z)}\langle \nabla_B^*\nabla_B du, \rho_r^2 du\rangle-
		\int_{D^\partial_{r}(z)} \rho_r^2\langle
		 (\nabla_B)_{\partial s} du, du\rangle\nonumber\\
		&\leq C\left(r^{-2} |\!|du|\!|_{L^2(D^+_{r}(z))}^2+
		\int_{D^+_{r}(z)} \rho_r^2(|du|^2+|du|^4)\right)
		-\int_{D^\partial_{r}(z)} \rho_r^2\langle
		 (\nabla_B)_{\partial s} du, du\rangle\nonumber\\
		&\leq Cr^{-2} |\!|du|\!|_{L^2(D^+_{r}(z))}^2
		-\int_{D^\partial_{r}(z)} \rho_r^2\langle
		 (\nabla_B)_{\partial s} du, du\rangle \label{e-u-der}.
	\end{align}
	Here the second inequality is obtained using \eqref{energy-string} and 
	we use the point-wise assumption on $du$ in \eqref{ptwise-bd} to produce the last inequality.

	Proposition \ref{normal-est} asserts that for any point $(0,\theta) \in D^\partial_{r}(z)$, we have:
	\begin{align*}
	  \(\int_{\Sigma}\langle (\nabla_A)_{\partial s} F_A, F_A\rangle\)
	  -\langle(\nabla_B)_{\partial s} du, du\rangle&\leq Ce_{A,u}(0,\theta)^{\frac{3}{2}}\\
	  &\leq C\hbar_0^{\frac{1}{2}}r^{-1}e_{A,u}(0,\theta)
	\end{align*}
	Therefore, we have:
	\begin{equation} \label{bdry-term-ineq}
	  \int_{D_{r}^\partial(z)}\rho_r^2\(\int_{\Sigma}\langle (\nabla_A)_{\partial s} F_A, F_A\rangle\)
	  -\rho_r^2 \langle(\nabla_B)_{\partial s} du, du\rangle
	  \leq C\hbar_0^{\frac{1}{2}}r^{-1}\int_{D_{r}^\partial(z)}\rho_r^2e_{A,u}(0,\theta)
	\end{equation}
	
	Suppose $f:D_{r}(z)\to \R$ is a compactly supported function. Then Sobolev embedding theorem implies that:
	\[
	  |\!|f|\!|_{L^2(D^\partial_{r}(z))}
	  \leq C_0r^{\frac{1}{2}}|\!|df|\!|_{L^2(D_{r}(z))},	
	\]
	where the constant $C_0$ is independent of $r$. By applying this inequality to the functions 
	$f_+(s,\theta):=|\rho_r du|(-s,\theta)$
	and $f_-(s,\theta):=(\int_{\{(s,\theta)\}\times\Sigma} |\rho_r F_A|^{2})^{\frac{1}{2}}$, we conclude that
	\begin{equation*}
		\int_{D^\partial_{r}(z)}|\rho_r du|^2
		\leq Cr|\!|\nabla_B(\rho_rd u)|\!|_{L^2(D_{r}^+(z))}^2,
	\end{equation*}
	and
	\begin{equation*}
		\int_{D^\partial_{r}(z)\times \Sigma} |\rho_r F_A|^2
		\leq Cr|\!|\nabla_A(\rho_rF_A)|\!|_{L^2(D_r(z)\times \Sigma)}^2.
	\end{equation*}	
	In order to derive the second inequality, we use the inequality
	\begin{equation} \label{der-con-energy}
	  |d(\int_{\{(s,\theta)\}\times\Sigma} |\rho_r F_A|^{2})^{\frac{1}{2}}|\leq 
	  |(\int_{\{(s,\theta)\}\times\Sigma} |\nabla_A(\rho_r F_A)|^{2})^{\frac{1}{2}}|.
	\end{equation}
	By adding up these two inequalities and using the inequality in \eqref{bdry-term-ineq}, we conclude that
	\begin{align}
		\int_{D_{r}^\partial(z)}\rho_r^2\(\int_{\Sigma}\langle (\nabla_A)_{\partial s} F_A, F_A\rangle\)&
	  -\rho_r^2 \langle(\nabla_B)_{\partial s} du, du\rangle\nonumber\\
	  &\leq C\hbar_0^{\frac{1}{2}}(|\!|\nabla_B(\rho_rd u)|\!|_{L^2(D_{r}(z))}^2
	  +|\!|\nabla_A(\rho_rF_A)|\!|_{L^2(D_r\times \Sigma)}^2).\label{bdry-term-ineq-2}
	\end{align}
	Summing up inequalities in\eqref{e-A-der}, \eqref{e-u-der} and \eqref{bdry-term-ineq-2} gives rise to
	\begin{align*}
	  |\!|\nabla_B(\rho_rd u)|\!|_{L^2(D^+_{r}(z))}^2+|\!|\nabla_A(\rho_rF_A)&|\!|_{L^2(D_r(z)\times \Sigma)}^2
	  \leq
	  C\left(r^{-2} |\!|F_A|\!|_{L^2(D_{r}(z)\times \Sigma)}^2+r^{-2}|\!|du|\!|_{L^2(D^+_{r}(z))}^2\right.\\
	  &\left.+\hbar_0^{\frac{1}{2}}\(|\!|\nabla_B(\rho_rd u)|\!|_{L^2(D^+_{r}(z))}^2
	  +|\!|\nabla_A(\rho_rF_A)|\!|_{L^2(D_r(z)\times \Sigma)}^2)\)\right).
	\end{align*}
	Therefore, if $\hbar_0$ is small enough, we can infer the claimed inequality in \eqref{L21-bd-energy}.	
\end{proof}
\begin{proof}[Proof of Proposition \ref{en-quant-weak}]
	We present the proof in $4$ steps:
	
	{\bf Step 1:} {\it Inverting the Laplacian.}
	Let $G(w)$ be the Green function $-\frac{1}{2\pi}\ln(|w-z|)+\frac{1}{2\pi}\ln(\frac{r}{4})$ of the Laplacian $\Delta_2$. 
	Note that $G(w)$ vanishes on $\partial D_{\frac{r}{4}}(z):=\partial B_{\frac{r}{4}}(z)\cap  \bbH_-$, the boundary of $D_{\frac{r}{4}}(z)$.
	As before,  $D_{\frac{r}{4}}^\partial (z)$ denotes $B_{\frac{r}{4}}(z)\cap  U_\partial$.
	We multiply \eqref{lap-mixed-energy} in 
	Proposition \ref{lap-estimates} by $G(w)$ and integrate over $D_{\frac{r}{4}}(z)$. Green's identity implies that:
	\begin{align}
		e_{A,u}(z)&\leq C \int_{D_{\frac{r}{4}}(z)}G(w)(e_{A,u}+e_{A,u}^2+e_{A,u}^{\frac{1}{2}}f_A)+
		\int_{\partial D_{\frac{r}{4}}(z)\sqcup D^\partial_{\frac{r}{4}}(z)}e_{A,u}\partial_{\nu}G-G\partial_{\nu}e_{A,u}\label{Green-int-pre}\\
		&\leq C |\!|G|\!|_{L^2(D_{\frac{r}{4}}(z))}\((1+|\!|e_{A,u}|\!|_{L^\infty})|\!|e_{A,u}|\!|_{L^2(D_{\frac{r}{4}}(z))}+
		|\!|e_{A,u}|\!|^{\frac{1}{2}}_{L^\infty}|\!|f_A|\!|_{L^2(D_{\frac{r}{4}}(z))}\)\label{Green-int}\\
		&+C\(r^{-1}|\!|e_{A,u}|\!|_{L^1(\partial D_{\frac{r}{4}}(z))}+|\!|e_{A,u}\partial_{\nu}G|\!|_{L^1(D^\partial_{\frac{r}{4}}(z))}+|\!|G|\!|_{L^2(D^\partial_{\frac{r}{4}}(z))}
		|\!|e_{A,u}|\!|_{L^2(D^\partial_{\frac{r}{4}}(z))}|\!|e_{A,u}|\!|_{L^\infty}^{\frac{1}{2}}\)\nonumber
	\end{align}
	Recall that $f_A$ in \eqref{Green-int-pre} is given in \eqref{fA}. Here $|\!|e_{A,u}|\!|_{L^\infty}$ is the $L^{\infty}$ norm of $e_{A,u}$ over $D_{r}(z)$, which is less than $\hbar'r^{-2}$
	by assumption. In order to bound the last term in \eqref{Green-int-pre}, we use Proposition \ref{normal-est} and Cauchy-Schwarz inequality.
	A straightforward computation shows
	\[
	  |\!|G|\!|_{L^2(D_{\frac{r}{4}}(z))}\leq C r,\hspace{1cm}|\!|G|\!|_{L^2(D^\partial_{\frac{r}{4}}(z))}\leq C r^{\frac{1}{2}}.
	\]
	Thus we deduce from \eqref{Green-int} that
	\begin{align}
		e_{A,u}(z)\leq &C\left(r^{-1}|\!|e_{A,u}|\!|_{L^2(D_{\frac{r}{4}}(z))}+r^{-1}|\!|e_{A,u}|\!|_{L^1(\partial D_{\frac{r}{4}}(z))}+
		r^{-\frac{1}{2}}|\!|e_{A,u}|\!|_{L^2(D^\partial_{\frac{r}{4}}(z))}\right.\nonumber\\
		&\left.+|\!|f_A|\!|_{L^2(D_{\frac{r}{4}}(z))}+|\!|e_{A,u}\partial_{\nu}G|\!|_{L^1(D^\partial_{\frac{r}{4}}(z))}\right).\label{des-ineq}
	\end{align}
	
	{\bf Step 2:} {\it Establishing the following bounds on various Sobolev norms of $e_{A,u}$:
	\begin{equation}\label{step3-2}
		\begin{array}{c}
			\displaystyle|\!|e_{A,u}|\!|_{L^2(D_{\frac{r}{4}}(z))}\leq C\frac{\int_{D_r(z)}e_{A,u}}{r},\\
		\displaystyle|\!|e_{A,u}|\!|_{L^2(D^\partial_{\frac{r}{4}}(z))}\leq C\frac{\int_{D_r(z)}e_{A,u}}{r^{\frac{3}{2}}},\hspace{1cm}
		|\!|e_{A,u}|\!|_{L^1(\partial D_{\frac{r}{4}}(z))}\leq C\frac{\int_{D_r(z)}e_{A,u}}{r}.
		\end{array}	
	\end{equation}
	}
	
	Sobolev embedding and a straightforward change of variable imply that there is a constant $C_0$, independent of $r$, such that for any function 
	$f:D_{\frac{r}{2}}(z)\to \R$ 
	\begin{equation}\label{r-ind-sob-emb}
	  (\int_{D_{\frac{r}{4}}(z)}f^4)^{\frac{1}{2}}\leq C_0(r\int_{D_{\frac{r}{2}}(z)}|df|^2+r^{-1}\int_{D_{\frac{r}{2}}(z)}f^2).
	\end{equation}
	Applying \eqref{r-ind-sob-emb} to the functions $f(s,\theta):=|du|(-s,\theta)$
	and $g(s,\theta):=(\int_{\{(s,\theta)\}\times\Sigma} |F_A|^{2})^{\frac{1}{2}}$ implies that
	\begin{equation}\label{step3-map}
		(\int_{D^+_{\frac{r}{4}}(z)}|du|^4)^{\frac{1}{2}}\leq C(r\int_{D^+_{\frac{r}{2}}(z)}|d|du||^2+
		r^{-1}\int_{D^+_{\frac{r}{2}}(z)}|du|^2),
	\end{equation}
	\begin{equation}\label{step3-conn}
		\left(\int_{D_{\frac{r}{4}}(z)}(\int_{\{(s,\theta)\}\times\Sigma}|F_A|^2)^2\right)^{\frac{1}{2}}\leq 
		C(r\int_{D_{\frac{r}{2}}(z)\times \Sigma}|\nabla_AF_A|^2+
		r^{-1}\int_{D_{\frac{r}{2}}(z)\times \Sigma}|F_A|^2).
	\end{equation} 
	In order to obtain \eqref{step3-conn}, we also used the inequality in \eqref{der-con-energy}. 
	Using Lemma \ref{L21-energy} and Kato's inequality, we derive the first inequality in \eqref{step3-2}.
	The remaining inequalities in \eqref{step3-2} can be verified in a similar way using the following variations of 
	\eqref{r-ind-sob-emb}:
	\begin{equation*}\label{r-ind-sob-emb-2}
	  \int_{\partial D_{\frac{r}{4}}(z)}f^2\leq C_0(r\int_{D_{\frac{r}{2}}(z)}|df|^2+r^{-1}\int_{D_{\frac{r}{2}}(z)}f^2),
	\end{equation*}	
	\begin{equation*}\label{r-ind-sob-emb-3}
	  (\int_{D_{\frac{r}{4}}^{\partial}(z)}f^4)^{\frac{1}{2}}\leq C_0(r^{\frac{1}{2}}\int_{D_{\frac{r}{2}}(z)}|df|^2+
	  r^{-\frac{3}{2}}\int_{D_{\frac{r}{2}}(z)}f^2).
	\end{equation*}	

	{\bf Step 3:} {\it 
	$
	  \displaystyle |\!|f_A|\!|_{L^2(D_{\frac{r}{4}}(z))}\leq Cr^{-2}\int_{D_{r(z)}}e_{A,u}.
	$}
	
	For $r\leq r_0$, let a function $f:D_{\frac{r}{2}}(z) \times \Sigma \to \R$ be given. Then there is a constant $C_0$, independent of $r$, such that
	\begin{equation}\label{step2-map}
		(\int_{D_{\frac{r}{4}}(z)\times \Sigma}f^4)^{\frac{1}{2}}\leq C_0(\int_{D_{\frac{r}{2}}(z)\times \Sigma}|df|^2+
		r^{-2}\int_{D_{\frac{r}{2}}(z)\times \Sigma}f^2),
	\end{equation}	
	This inequality can be verified by considering $\rho_{\frac{r}{2}}\cdot f$ and applying the Sobolev embedding inequality for functions defined 
	on $D_{2r_0} \times \Sigma$. By definition, $|\!|f_A|\!|_{L^2(D_{\frac{r}{4}}(z))}$ is equal to $|\!|F_A|\!|_{L^4(D_{\frac{r}{4}}(z)\times \Sigma)}^2$.
	Thus we can employ \eqref{step2-map} for $f=|F_A|$ and Kato's inequality to obtain the following inequality:
	\begin{align*}
		|\!|f_A|\!|_{L^2(D_{\frac{r}{4}}(z))}
		&\leq C\(\int_{D_{\frac{r}{2}}(z)\times \Sigma}|\nabla_AF_A|^2+r^{-2}\int_{D_{\frac{r}{2}}(z)\times \Sigma}|F_A|^2\)\\
		&\leq C r^{-2}\int_{D_{r}(z)} e_{A,u}.		 
	\end{align*}	
	The second inequality follows from Lemma \ref{L21-energy}.	
	
	{\bf Step 4:} {\it Completing the proof.} We have appropriate bounds on all terms in \eqref{des-ineq} to obtain the desired result except the term
	$|\!|e_{A,u}\partial_{\nu}G|\!|_{L^1(D^\partial_{\frac{r}{4}}(z))}$. In the case that $z\in U_\partial$, this term vanishes. Therefore, we obtain the inequality in 
	\eqref{Main-ineq-weak} for such choices of $z$. This preliminary case, allows us to complete the proof because for a general $z$ we have:
	\[|\!|e_{A,u}\partial_{\nu}G|\!|_{L^1(D^\partial_{\frac{r}{4}}(z))}\leq C|\!|e_{A,u}|\!|_{L^{\infty}(D^\partial_{r}(z))}.\]  
\end{proof}

Theorem \ref{en-quant} is a consequence of Proposition \ref{en-quant-weak} and the following elementary lemma.
\begin{lemma}
	Suppose $X$ is a compact metric space and $f:X\to \R^{\geq 0}$ is a continuous function which 
	satisfies the following properties for 
	positive constants $\hbar'$ and $\kappa'$. For any $z\in X$ and positive real number $r$ satisfying
	\[
	  \hspace{2cm} f(w)\leq \hbar' r^{-2} \hspace{1cm}w\in D_{r}(z),
	\]
	we have
	\begin{equation}\label{ineq-assumption}
		\hspace{1cm} f(z) \leq \kappa' \frac{ \int_{D_{r}(z)}f\dvol}{r^2}.
	\end{equation}
	Then there are constants $\hbar$ and $\kappa$ such that if for 
	any $z\in X$ and a positive $r$ we have the inequality 
	\[
	  \int_{D_r(z)}f\dvol\leq \hbar,
	\]
	then:
	\[
	  \hspace{1cm} f(z) \leq \kappa \frac{ \int_{D_r(z)}f\dvol}{r^2}.
	\]
\end{lemma}
\begin{proof}
	We claim that $\hbar=\hbar'/(17\kappa')$ and $\kappa=4\kappa'$ satisfy the required properties. 
	To that end, we assume that $z$ and $r$ are given such that
	\begin{equation}\label{r-z0}
		\int_{D_{r}(z)}f\leq \frac{\hbar'}{17\kappa'}.
	\end{equation}
	Define $\phi: D_r(z) \to \R^{\geq 0}$ as 
	\[
	  \phi(w):=(r-|w-z|)^2f(w).
	\]
	This continuous function extends to the boundary of $D_r(z)$ by zero. 
	Therefore, $\phi$ achieves its maximum in an interior point $w_0$. First let $\phi(w_0)\leq \hbar'/4$. In this case, for
	any $w \in D_{\frac{r}{2}}(z)$, we have:
	\begin{align*}
	  f(w) &\leq f(w)\frac{(r-|w-z|)^2}{(\frac{r}{2})^2}\\
	   &\leq f(w_0)\frac{(r-|w_0-z|)^2}{(\frac{r}{2})^2}\\
	   &\leq \hbar'r^{-2}.
	\end{align*}
	Therefore, the assumption implies that
	\begin{align*}	
	  f(z) &\leq \kappa' \frac{ \int_{D_{\frac{r}{2}}(z)}f}{(\frac{r}{2})^2}\\
	  &=4 \kappa' \frac{ \int_{D_{\frac{r}{2}}(z)}f}{r^2}.
	\end{align*}

	Next, let $\phi(w_0)> \hbar'/4$ and define $s:=\sqrt{\frac{\hbar'}{16 f(w_0)}}$. Then we have
	\[
	  s<\frac{r-|w_0-z|}{2}.
	\]
	This implies that $D_s(w_0)$ is a subset of $D_r(z)$. For any $y\in D_s(w_0)$, we can also write
	\begin{align*}
	  f(y) &\leq f(w_0)\frac{(r-|w_0-z|)^2}{(r-|y-z|)^2}\\
	   &\leq f(w_0)\frac{(r-|w_0-z|)^2}{(r-(|w_0-z|+s))^2}\\
	   &\leq 4f(w_0)\\
	   &\leq \hbar' s^{-2}.
	\end{align*}
	Consequently, \eqref{ineq-assumption} implies 
	\begin{align*}
	  f(w_0) &\leq \kappa' \frac{ \int_{D_{s}(w_0)}f}{s^2}\\
	  &= \frac{16 \kappa' f(w_0)}{\hbar'}\int_{D_{s}(w_0)}f\\
	  &\leq \frac{16 \kappa' f(w_0)}{\hbar'}\int_{D_{r}(z)}f\\
  	  &\leq \frac{16}{17}f(w_0).
	\end{align*}
	The last inequality, which leads to a contradiction, follows from \eqref{r-z0}.	
\end{proof}

\subsection{Removability of singularities}\label{remove-sing}
The main goal of this subsection is to prove a removability of singularity result for the mixed solution. As before, we fix a positive real constant $r_0$. Let $(A,u)$ be a solution of the mixed equation associated to the quintuple $\fP(r_0)$ in \eqref{quintuple-comp} where the ASD connection $A$ is defined on $(D_-\setminus \{0\})\times \Sigma$ and the pseudo-holomorphic map $u$ is defined on $D_+\setminus \{0\}$. In particular, we assume that $A$ and $u$ satisfy the matching condition on $U_\partial\setminus \{0\}$. Then the 2-dimensional energy density $e_{A,u}$ is defined on $D_-\setminus \{0\}$. For any $r\leq r_0$, define
\[
  \mathfrak E_r(A,u):=\int_{D_-(r)\setminus \{0\}}e_{A,u}.
\]

\begin{theorem}
	\label{thm:removal}
	For $(A,u)$ as above, let $\mathfrak E_{r_0}(A,u)$ be finite. Then we have the followings.
	\vspace{-5pt}
	\begin{itemize} 
		\item[(i)] There exists $g\in \mathcal G((D_-\setminus \{0\})\times \Sigma,F)$ such that $g^*A$ extends to a smooth connection 
		$\widetilde A$ on $D_-\times \Sigma$.
		\item[(ii)] $u$ can be extended to a smooth map $\widetilde u: D_+ \to  \mathcal M(\Sigma,F)$.
	\end{itemize}
	In particular, $(\widetilde A,\widetilde u)$ is a solution of the mixed equation associated to $\fP(r_0)$.
\end{theorem}
Recall that for an $\SO(3)$-bundle $V$ over a 4-manifold $X$, $\mathcal G(X,V)$ denotes the space of smooth sections of the fiber bundle ${\rm Fr}(V)\times_{\rm ad} \SU(2)$. Without loss of generality, we may decrease the value of $r_0$ as we wish. In particular, we we may assume that 
\begin{equation}\label{bound-hbar-0}
	\mathfrak E_{r_0}(A,u)<\hbar_0
\end{equation}
where $\hbar_0$ is less than the constant $\hbar$ of Theorem \ref{en-quant}, and $r_0$ is smaller than the injectivity radius of $\Sigma$. 

We use polar coordinates on $B_{r_0}(0)=D_-\cup D_+$ throughout this subsection. Polar coordinate of a typical point is denoted by $(r,\phi)$ where $r \in (0,r_0]$ and  $\phi \in \R$, and it is related to our previous notation by the formula $(s,\theta) = (r\cos \phi, r\sin \phi)$. We also write $S_r^1$ for the set of points in $B_{r_0}(0)$ whose radial coordinate is equal to $r$. The intersection of $S_r^1$ with $D_+$ and $D_-$ are denoted by $S_{r,+}^1$ and $S_{r,-}^1$.

\subsubsection{Strategy of the proof}\label{mainprop}
The key estimate for us is the following proposition.
\begin{prop}\label{decayenergy}
	For $(A,u)$ as in the statement of Theorem \ref{thm:removal}, there exists a positive
	$\beta$ such that
	\[
	  \mathfrak E_r(A,u)  \le C r^{2\beta}.
	\]
\end{prop}
The proof of this estimate, which follows a plan similar to \cite[Sections 3 and 4]{Weh:Lag-bdry-ana2}, will be given in the next two subsections. In the remaining part of this subsection, we explain how Theorem \ref{thm:removal} can be derived from Proposition \ref{decayenergy}.

\begin{cor}\label{prop:5-20}
	For any $r\leq r_0/2$, we have:
	\begin{enumerate}
		\item[(i)]$\displaystyle \sup_{z \in S^1_{r,-}} \Vert F_{A}\vert_{\{z\} \times \Sigma}\Vert_{L^2(\Sigma)} \le C r^{\beta-1}$;
		\item[(ii)] $\Vert F_{A}\vert_{\{(r\cos\phi,r\sin \phi)\} \times \Sigma}\Vert_{L^{\infty}(\Sigma)} \le C (\cos\phi)^{-2} r^{\beta-2}$;
		\item[(iii)] $\vert du(z)\vert \le C r^{\beta-1}$ for $z\in S^1_{r,+}$.
	\end{enumerate}
\end{cor}
\begin{proof}[Proof of Proposition \ref{decayenergy} $\Rightarrow$  Corollary \ref{prop:5-20} ]
	Since $\mathfrak E_{r_0}(A,u) \le \hbar$, Theorem \ref{en-quant} and Proposition \ref{decayenergy} imply that for $z \in S^1_{r,-}$ with $2r \le r_0$, we have 
	the following sequences of inequalities
	\begin{equation}\label{I}
	  \Vert F_{A}\vert_{\{z\} \times \Sigma}\Vert^2_{L^2(\Sigma)} \le e_{A,u}(z) \le \kappa \frac{\mathfrak E_{2r}(A,u)}{4r^2}  \le C r^{2\beta-2}
	\end{equation}
	This verifies $(i)$. To prove $(ii)$, let $p\in \Sigma$ and $z = (r\cos \phi,r\sin\phi)$ with $\phi \in (\pi/2,3\pi/2)$.
	The ball of radius $s= r\vert\cos\phi\vert$ centered at $(z,p)$ is contained in $D_-(r_0) \times \Sigma$. Moreover, the $L^2$ norm of 
	the curvature of $A$ on this ball is estimated by $Cr^{\beta}$. Therefore, $(ii)$ is a consequence of Uhlenbeck's Theorem which is recalled as Lemma \ref{4dimestmate} below.
	Finally, $(iii)$ is also a consequence of Theorem \ref{en-quant} and Proposition \ref{decayenergy}:
	\begin{equation}\label{III}
	  \vert du(z)\vert \le e_{A,u}(z')^{\frac{1}{2}} \le \kappa^{\frac{1}{2}} \frac{\mathfrak E_{2r}(A,u)^{\frac{1}{2}}}{2r}  \le C r^{\beta-1}
	\end{equation}
	where $z=(s,\theta) \in S^1_{r,+}$ and $z' = (-s,\theta)$.
\end{proof}

\begin{lemma}\label{4dimestmate}{\rm (Uhlenbeck)}
	For a large enough positive integer $k$, suppose a $C^k$-compact family of metrics on the 4-dimensional ball $B_{r_0}(0)$ is given.
	There exists $\hbar > 0$ such that the following holds. Suppose $A$ is an ASD connection on $B_r(0)\subset B_{r_0}(0)$ with $\Vert F_A\Vert_{L^2(B_r(0))}\leq \hbar$.
	Then
	\[
	  \vert F_A(0)\vert< C r^{-2} \Vert F_A\Vert_{L^2(B_r(0))}.
	\]
\end{lemma}
\begin{proof}
	By scaling we may assume $r=1$, where it is the standard Uhlenbeck compactness theorem.
\end{proof}

\begin{proof}[Proof of Corollary \ref{prop:5-20} $\Rightarrow$ Theorem \ref{thm:removal}]
	Let $p$ be a real number satisfying $2<p<\frac{4}{2-\beta}$. Properties $(i)$ and $(ii)$ of Corollary \ref{prop:5-20} are the assumptions of \cite[Theorem 5.3 (ii)]{Weh:Lag-bdry-ana2}.
	Therefore, there is $g\in  \mathcal G((D_-\setminus \{0\})\times \Sigma,F)$ such that $\widetilde A:=g^*A$ extends as an $L^p_1$ connection on $D_- \times \Sigma$. 
	By continuity, $\widetilde A$ is an ASD-connection. Using $(iii)$ of Corollary \ref{prop:5-20}, we may conclude that $u$ extends as a Holder continuous function $\widetilde u:D_+\to \mathcal M(\Sigma,F)$ and $d\widetilde u$ belongs to $L^p$. 
	Since $\widetilde A$ and $\widetilde u$ are continuous, they satisfy the matching condition. Now we can appeal to our regularity results of Section
	 \ref{regularity} to complete the proof of Theorem \ref{thm:removal}
\end{proof}

\subsubsection{Energy estimate via the Chern-Simons functional}\label{csestimate}
Let $Y$ be a 3-manifold and $\beta_0$ be a flat connection on an $SO(3)$-bundle $E$ over $Y$\footnote{Here we diverge from our convention that connections on 3-manifolds are denoted by the letter $B$ because soon we will focus on the case that $Y=S^1\times \Sigma$ and $\beta_0$ is the pullback of a connection on $\Sigma$.}.  Let $B = \beta_0+ b$ be a connection on $E$ where $b$ is a section of $\Lambda^1(Y) \otimes E$. The Chern-Simons functional of $B$ is defined as
\begin{equation}
	\aligned
	{CS}_{\beta_0}(B) &:= 
	\int_{Y} \tr \left(b \wedge F_{B} - \frac{1}{3}  b\wedge b\wedge b\right).
	\endaligned
\end{equation}
Equivalently, if $A$ is an arbitrary connection on $[0,1]\times Y$ whose restrictions to $\{0\}\times Y$ and $\{1\}\times Y$ are equal to $\beta_0$ and $B$, then
\begin{equation}\label{form536}
	CS_{\beta_0}(B) = \int_{[0,1]\times Y} \tr\(F_{A} \wedge F_{A}\).
\end{equation}
A consequence of \eqref{form536}, which shall be helpful for us, is that $CS_{\beta_0}(B)=CS_{\beta_1}(B)$, if $\beta_0$ and $\beta_1$ are connected to each other by a path of flat connections. It is also a well-known fact that
\begin{equation}\label{form538}
	\frac{1}{8\pi^2}\(CS_{\beta_0} (g^* A) - CS_{\beta_0} ( A) \) \in \Z,
\end{equation}
for any $g\in \mathcal G(Y,E)$.

We will be interested in the case that $Y=S^1 \times \Sigma$, $E=S^1\times F$ and $\beta_0$ is the pull-back of a flat connection on $F$, which is also denoted by $\beta_0$. Let $B=\beta_0+b$ be a connection on $E$ such that $b=\alpha+zd\phi$ where $z$ and $\alpha$ are sections of $E$ and $\Lambda^1(\Sigma)\otimes E$ over $Y$. Then the Chern-Simons function of $B$ is given as
\begin{equation}\label{CS-S1Sigma}
  CS_{\beta_0}(B)=\int_{S^1}\int_\Sigma\tr(\partial_\phi \alpha \wedge \alpha+2F(\beta_0+\alpha)z)\wedge d\phi
\end{equation}
We shall also need an analogue of \eqref{form536} for a connection $A=\beta_0+\alpha+wdr+zd\phi$ on the 4-manifold $X=[0,1]\times [0,1]\times \Sigma$ where $r$ and $\phi$ denote the coordinates on the first and the second intervals. In this case Stokes theorem implies that
\begin{align}\label{CS-energy}
	\int_{X} \tr\(F_{A} \wedge F_{A}\)=&\int_{(\partial [0,1] )\times [0,1]\times \Sigma}
	\tr(\partial_\phi \alpha \wedge \alpha+2F(\beta_0+\alpha)z)\wedge d\phi\nonumber\\
	-&\int_{ [0,1]\times(\partial [0,1] )\times \Sigma}
	\tr(\partial_r \alpha \wedge \alpha+2F(\beta_0+\alpha)w)\wedge dr.
\end{align}

Let $(A,u)$ be as in Theorem \ref{thm:removal}. Suppose $\beta_0^r$ denotes the flat connection on $F$ given by $A\vert_{\{(0,r)\}\times \Sigma}$. For $z\in S^1_{r,+}$ with $2r < r_0$, we can use Theorem \ref{en-quant} and \eqref{bound-hbar-0} as in \eqref{III} to conclude
\begin{equation}\label{form525252}
	\vert du(z)\vert \le \kappa^{\frac{1}{2}} \frac{\mathfrak E_{2r}(A,u)^{\frac{1}{2}}}{2r}  \le \frac{1}{2}\kappa^{\frac{1}{2}} \hbar_0^{\frac{1}{2}} r^{-1}.
\end{equation}
In particular, the diameter of $u(S^1_{r,+})$ is smaller than $C\hbar_0^{1/2}$. Thus by taking $r_0$ and $\hbar_0$ small enough, we may assume that there is  $b_+^r(z) \in \Lambda^1(\Sigma) \otimes F$ for each $z\in S^1_{r,+}$ such that the following conditions hold.
\begin{enumerate}
	\item[$(u.1)$] $[\beta^r_0 + b_+^r(z)] = u(z)$.
	\item[$(u.2)$] $d_{\beta^r_0+b_+^r(z)}^* \partial_\phi b_+^r(z) = 0$.
	\item[$(u.3)$] $b_+^r(r,0) = 0$.
\end{enumerate}
Let $B_{r,+}$ be the connection on $S^1_{r,+} \times F$ defined as $\beta^r_0 + b_+^r$ (with vanishing $d\phi$ component). 

By parallel transport along $S^1_{r,-}$, we may define a connection $B_{r,-}$ on $S^1_{r,-} \times F$ which satisfies the following properties.
\begin{itemize}
	\item[$(A.1)$]  $B_{r,-}$ is gauge equivalent to the restriction of $A$ to $S^1_{r,-}$.
	\item[$(A.2)$]  For each $z\in S^1_{r,+}$ , there is $b_-^r(z)\in \Lambda^1(\Sigma) \otimes F$ such that $B_{r,-} =  \beta^r_0 + b_-^r(z)$. That is to say, $B_{r,-}$ has a vanishing $d\phi$ component.
	\item[$(A.3)$]$b_-^r(r,0) =0$.
\end{itemize}
Similar to \eqref{I}, we have:
\begin{equation}\label{form3}
	\Vert \partial_\phi b_-^r(z)\Vert_{L^2(\Sigma)}\leq 
	re_{A}(z)^{\frac{1}{2}}\leq \kappa^{\frac{1}{2}} \frac{\mathfrak E_{2r}(A,u)^{1/2}}{2}.
\end{equation}

\begin{lemma}\label{estmiate-L2-a-r}
	For any $r$, we have:
	\begin{equation}
	 	\Vert b_+^r(-r,0) \Vert	_{L^2(\Sigma)}+\Vert b_-^r(-r,0) \Vert	_{L^2(\Sigma)}\leq 
		C r\int_{\frac{\pi}{2}}^{\frac{3\pi}{2}}e_{A,u}^{\frac{1}{2}}(r,\phi)
	\end{equation}
	In particular, the left hand side of the above inequality is smaller than $C\mathfrak E_{2r}(A,u)^{1/2}$.
\end{lemma}
\begin{proof}
	Using $(u.2)$, we conclude
	\begin{align*}
		\Vert b_+^r(-r,0)\Vert_{L^2(\Sigma)}
		&=\Vert \int_{-\frac{\pi}{2}}^{\frac{\pi}{2}} \partial_\phi b_+^r(r,\phi)d\phi\Vert_{L^2(\Sigma)}\leq 
		r\left\vert \int_{-\frac{\pi}{2}}^{\frac{\pi}{2}} \vert du \vert (r,\phi)d\phi\right\vert\leq C r\int_{\frac{\pi}{2}}^{\frac{3\pi}{2}}e_{A,u}^{\frac{1}{2}}(r,\phi).
	\end{align*}
	Similarly, \eqref{form3} implies that:
	\begin{equation*}
		\Vert b_-^r(-r,0)\Vert_{L^2(\Sigma)} \leq C r\int_{\frac{\pi}{2}}^{\frac{3\pi}{2}}e_{A,u}^{\frac{1}{2}}(r,\phi).
	\end{equation*}
	The second part of the lemma is a consequence of Theorem \ref{en-quant}.
\end{proof}

The matching condition implies that $\beta^r_0 + b_+^r(-r,0)$ is gauge equivalent to $\beta^r_0+b_-^r(-r,0)$. Namely, there exists $g_r \in \mathcal G(\Sigma,F)$ such that 
\[
  \beta^r_0 + b_+^r(-r,0) = g_r^* (\beta^r_0 + b_-^r(-r,0))
\]
The proof of the following lemma on the extension of the gauge transformation $g_r$ can be found in \cite[Lemma 3.2 (ii)]{Weh:Lag-bdry-ana2}, which is proved based on the results of \cite{HL:Liouv}.
\begin{lemma}\label{lem526}
	There exists $\tilde g_r\in \mathcal G([0,1]\times \Sigma,[0,1] \times F)$ such that: 
	\vspace{-5pt}
	\begin{itemize}
            	\item[(i)] $\tilde g_r\vert_{\{0\} \times \Sigma}$ is identity;
            	\item[(ii)] $\tilde g_r\vert_{\{1\} \times \Sigma}$ is $g_r$;
            	\item[(iii)] let $b^r_{0}$ be the 1-form
				\[(\tilde g_r)^*(\beta^r_0 + b_-^r(-r,0)) -(\beta^r_0 + b_-^r(-r,0))\] 
				on $[0,1]\times \Sigma$. 
				Then $\Vert b^r_{0} \Vert_{L^3([0,1] \times \Sigma)} \le C\Vert b_+^r(r,0)-b_-^r(r,0) \Vert_{L^2(\Sigma)}$.
	\end{itemize}
\end{lemma}

Since the connection $B_{r,0}:=\tilde g_r^*(\beta^r_0 + b_-^r(r,0))$ on $[0,1]\times \Sigma$ restricts to $\beta^r_0+b_-^r(r,0)$ and $\beta^r_0 + b_+^r(r,0)$ on $\{0\}\times \Sigma$ and $\{1\}\times \Sigma$, we can glue $B_{r,-}$, $B_{r,0}$ and $B_{r,+}$ to define a connection $B_{r}$ on the closed 3-manifold
\begin{equation}\label{spaceYr}
	Y_r:=S^1_{r,-} \times \Sigma \cup [0,1] \times \Sigma \cup S^1_{r,+} \times \Sigma.
\end{equation}
Although the connection $B_{r}$ is not smooth, it is clear from \eqref{CS-S1Sigma} that $CS_{\beta_0^r}(B_r)$ is well-defined.
\begin{figure}
\begin{center}
\begin{tikzpicture}[scale=1.5]
   \draw [red,thick,domain=-30:90] plot ({cos(\x)}, {sin(\x)});
   \draw [blue,thick,domain=90:210] plot ({cos(\x)}, {sin(\x)});
   \draw [green,thick,domain=210:330] plot ({cos(\x)}, {sin(\x)});
   \node[right] at (.8,.8) {$S^1_{r,+} \times \Sigma$}; \node[right] at (-1.9,.8) {$S^1_{r,-} \times \Sigma$};
   \node[right] at (-.5,-1.3) {$[0,1] \times \Sigma$};
\end{tikzpicture}
\end{center}
\caption{A schematic picture of the 3-manifold $Y_r$ and its decomposition in \eqref{spaceYr}.}
\end{figure}

\begin{lemma}\label{lem527}
	For any $r$, there is a constant $K$ such that:
	\[
	  \vert CS_{\beta_0^r}(B_r) \vert  \le K r  \frac{d}{d r} \mathfrak E_r(A,u).
	\]
\end{lemma}

\begin{proof}
	We firstly observe that
	\begin{align}
		\left\vert\int_{S_{r,+}^1 \times \Sigma} \tr\left( b_+^r \wedge F_{B_{r,+}}- \frac{1}{3} (b_+^r)^3\right)\right\vert 
		&= \left\vert\int_{S_{r,+}^1 \times \Sigma} \tr\left(b_+^r \wedge\partial_\phi b_+^r\right)\wedge d\phi\right\vert 
		\nonumber\\
		&\leq \int_{\Sigma} \int_{-\frac{\pi}{2}}^{\frac{\pi}{2}} \left\vert\partial_\phi b_+^r\right\vert \(\int_{\phi}^{\frac{\pi}{2}}
		\left\vert\partial_\phi b_+^r \right\vert d\psi\)d\phi\dvol_{\Sigma}\nonumber\\
		&\leq \int_{\Sigma} \(\int_{-\frac{\pi}{2}}^{\frac{\pi}{2}}
		\left\vert\partial_\phi b_+^r \right\vert d\phi\)^2\dvol_{\Sigma}\nonumber\\
		&\leq Cr^2\int_{-\frac{\pi}{2}}^{\frac{\pi}{2}}  \vert du(r,\phi)\vert^2 d\phi
		\leq Cr^2\int_{\frac{\pi}{2}}^{\frac{3\pi}{2}}  e_{A,u}(r,\phi) d\phi.\label{form528}
	\end{align}
	In a similar way, we have
	\begin{align}
		\left\vert\int_{S_{r,-}^1 \times \Sigma} \tr\left( b_-^r \wedge F_{B_{r,-}}- \frac{1}{3}(b_-^r)^3\right)\right\vert 
		&\leq Cr^2\int_{\frac{\pi}{2}}^{\frac{3\pi}{2}}  \Vert F_A(r,\phi)\Vert^2_{L^2(\Sigma)} d\phi b_-^r\nonumber\\
		&\leq Cr^2\int_{\frac{\pi}{2}}^{\frac{3\pi}{2}}  e_{A,u}(r,\phi) d\phi\label{form529}.
	\end{align}
	Finally, Lemma \ref{lem526} and flatness of $B_{r,0}$ give rise to the following estimates:
	\begin{align}
		\left\vert\int_{[0,1] \times \Sigma} \tr\left( b_0^r \wedge F_{B_{r,0}}-\right.\right.&\left.\left. \frac{1}{3} (b_0^r)^3\right)\right\vert 
		= \left\vert\int_{S_{r,+}^1 \times \Sigma} \tr(\frac{1}{3} (b_0^r)^3)\right\vert\nonumber \\
		&\leq \Vert b_+^r(r,0)-b_-^r(r,0) \Vert^3_{L^2(\Sigma)}\leq
		C r^3 (\int_{\frac{\pi}{2}}^{\frac{3\pi}{2}}e_{A,u}^{\frac{1}{2}}(r,\phi)d\phi)^3. \label{1531form-p}
	\end{align}
	For the last inequality we use Lemma \ref{estmiate-L2-a-r}. Since Theorem \ref{en-quant} implies that $r \int_{\frac{\pi}{2}}^{\frac{3\pi}{2}}e_{A,u}^{\frac{1}{2}}(r,\phi)d\phi$
	is  bounded by $C\mathfrak E_{2r}(A,u)^{1/2}$, we may assume that $r \int_0^{\pi}e_{A,u}^{\frac{1}{2}}(r,\phi)d\phi$
	is smaller than $1$ by picking $\hbar_0$ to be small enough. In particular, as a consequence of \eqref{1531form-p} and the Cauchy–Schwarz inequality we have
	\begin{align}
		\left\vert\int_{[0,1] \times \Sigma} \tr\left( b_0^r \wedge F_{B_{r,0}}- \frac{1}{3} (b_0^r)^3\right)\right\vert 
		\leq C r^2 \int_{\frac{\pi}{2}}^{\frac{3\pi}{2}}e_{A,u}(r,\phi)d\phi\label{1531form}.
	\end{align}
	The desired result follows by adding \eqref{form528},  \eqref{form529} and \eqref{1531form}.
\end{proof}

\begin{lemma}\label{lem535}
	$CS_{\beta_0^r}(B_r)  = \mathfrak E_{r}(A,u).$
\end{lemma}

\begin{proof}
	For $0< \delta< r$, we define a 4-dimensional manifold
	\begin{equation}\label{Xdr}
	  X_{\delta,r} =[\delta,r]\times S^1_{r,-} \times \Sigma \cup [\delta,r]\times [0,1] \times \Sigma \cup 
	  [\delta,r]\times S^1_{r,+} \times \Sigma,	
	\end{equation}
	in the same way as in \eqref{spaceYr}. (See Figure \ref{fig:Xdr}.) In particular, this 4-manifold is diffeomorphic to $[\delta,r]\times S^1\times \Sigma$
	and can be written as the union of 3-manifolds $Y_\rho$ for $\rho\in [\delta,r]$.
	The boundary components of $X_{\delta,r}$ are identified with $Y_r$ and $Y_\delta$.
	The pull-back of $F$ on $\Sigma$ induces a bundle on $ X_{\delta,r} $, which we denote by $V_{\delta,r}$.
        \begin{figure}
        \begin{center}
        \begin{tikzpicture}[scale=1.5]
         \filldraw[fill=red, draw=red] ({cos(-30)}, {sin(-30)})--({1.5*cos(-30)}, {1.5*sin(-30)})arc (-30:90:1.5) -- ({cos(90)}, {sin(90)})arc (90:-30:1);
         \filldraw[fill=blue, draw=blue] ({cos(90)}, {sin(90)})--({1.5*cos(90)}, {1.5*sin(90)})arc (90:210:1.5) -- 
         ({cos(210)}, {sin(210)})arc (210:90:1);
          \filldraw[fill=green, draw=green] ({cos(210)}, {sin(210)})--({1.5*cos(210)}, {1.5*sin(210)})arc (210:330:1.5) -- 
          ({cos(330)}, {sin(330)})arc (330:210:1);
             \node[right] at (1.5,.7) {$[\delta,r]\times S^1_{r,+} \times \Sigma$}; 
             \node[right] at (-3.4,.7) {$[\delta,r]\times S^1_{r,-} \times \Sigma$};
	     \node[right] at (-1.0,-1.8) {$[\delta,r]\times [0,1] \times \Sigma$};
	     \node[right] at (-.2,0.7) {$Y_\delta$};
	     \node[right] at (-.2,1.7) {$Y_r$};
        \end{tikzpicture}
        \end{center}
        \caption{A schematic picture of the 4-manifold $X_{\delta,r}$ and its decomposition in \eqref{Xdr}.}
        \label{fig:Xdr}
        \end{figure}
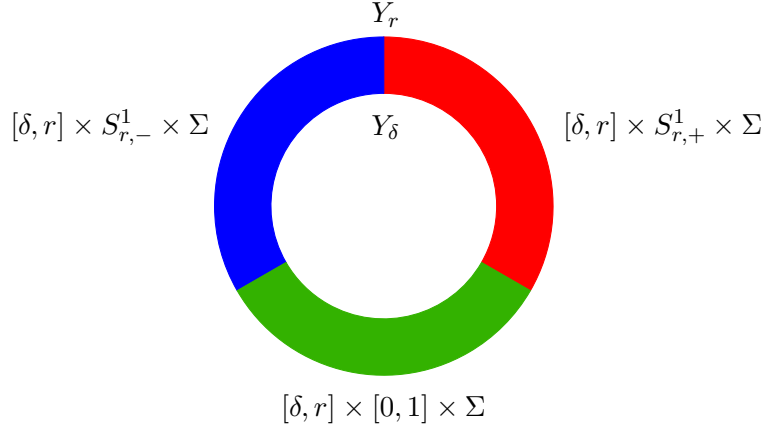

	We consider a connection $A_{\delta,r}$ on $X_{\delta,r}$ which restricts to $B_r$ on the boundary component $Y_r$ 
	and is defined in a similar way. The subspace $[\delta,r]\times S^1_{r,-} \times \Sigma$ of $X_{\delta,r}$ can be identified 
	canonically with 
	$\Sigma \times (D_-(r)\setminus D_-(\delta))$ and the restriction $A_{\delta,r,-}$ of $A_{\delta,r}$ 
	to this region is gauge equivalent to the connection $A$. To define $A_{\delta,r,-}$, we firstly fix a gauge along the ray 
	$\{(\rho,\pi/2)\}_{\delta\leq \rho\leq r}$ by parallel transport and then extend it in the angular directions by parallel transport along 
	the arcs with fixed radial coordinate. In particular, the connection $A_{\delta,r,-}$ has vanishing $d\phi$ component.
	The region $[\delta,r]\times S^1_{r,+} \times \Sigma$ in $X_{\delta,r}$ is identified with 
	$\Sigma \times (D_+(r)\setminus D_+(\delta))$, and analogous to $B_{r,+}$ we require that the restriction $A_{\delta,r,+}$ 
	of $A_{\delta,r}$
	to this region has vanishing $d\phi$ and $dr$ components, the restriction $A_{\delta,r}(z)$ of $A_{\delta,r,+}$ to $\Sigma\times \{z\}$ 
	represents $u(z)$, $d^*_{A_{\delta,r}(z)}\partial_\phi A_{\delta,r}(z)=0$ and the restriction of $A_{\delta,r,+}$ to $[\delta,r]\times \{\pi/2\}\times \Sigma$ agree with the restriction of $A_{\delta,r,-}$. 
	
	Next, we extend the above connection to the remaining region $[\delta,r]\times [0,1] \times \Sigma$ of \eqref{Xdr}.
	The restrictions of $A_{\delta,r,+}$ and $A_{\delta,r,-}$ to $[\delta,r]\times \{\pi/2\}\times \Sigma$ and 
	$[\delta,r]\times \{3\pi/2\}\times \Sigma$ are gauge equivalent to each other. We pick a gauge transformation $\widetilde g$ over 
	$[\delta,r]\times [0,1] \times \Sigma$ such that 
	\begin{itemize}
		\item[(i)] $\widetilde g(z)=1$ for $z\in [\delta,r]\times \{1\}\times \Sigma$;
		\item[(ii)]  $\widetilde g(\rho,0,x)^*A_{\delta,r,+}(\rho,\frac{\pi}{2},x)=A_{\delta,r,-}(\rho,\frac{3\pi}{2},x)$ for $(\rho,0,x)\in [\delta,r]\times \{0\}\times \Sigma$;
		\item[(ii)]  the restriction of $\widetilde g$ to $\{r\}\times [0,1] \times \Sigma$ is equal to $\widetilde g_r$ used in the definition
		of the connection $B_r$.
	\end{itemize}
	Restrict $A_{\delta,r,-}$ to $[\delta,r]\times \{0\}\times \Sigma$ and pull back this connection to 
	$[\delta,r]\times [0,1] \times \Sigma$ via the projection map. Applying the gauge transformation
	$\widetilde g$ gives rise to a connection $A_{\delta,r,0}$ which can be glued to 
	$A_{\delta,r,+}$ and $A_{\delta,r,-}$ to form the desired connection $A_{\delta,r}$ on $X_{\delta,r}$. 
	The restriction of $A_{\delta,r}$ to $Y_r$ agree with $B_r$ and its restriction to $Y_\delta$, 
	denoted by $B_\delta'$, is gauge equivalent to $B_\delta$.
	
	Although the connection $A_{\delta,r}$ is not smooth, its restriction to each of the three regions in 
	\eqref{Xdr} is smooth. We apply \eqref{CS-energy} to each of these regions and add the resulting 
	identities. Since the connection $A_{\delta,r,-}$ satisfies the ASD equation, $A_{\delta,r,0}$ is flat and
	$A_{\delta,r,+}$ represents the holomorphic map $u$, we have
	\[
	  \int_{(D_-(r)\setminus D_-(\delta))\times \Sigma} \vert F_A\vert^2+\int_{(D_+(r)\setminus D_+(\delta))\times \Sigma} \vert du\vert^2
	  = CS_{\beta_0^r}(B_r)-CS_{\beta_0^\delta}(B_\delta').
	\]
	Lemma \ref{lem527} shows that $CS_{\beta_0^\rho}(B_\rho)$ is less than a given positive real number if $\rho$ is small enough. 
	The above identity shows that 
	\begin{equation}\label{difference-CS}
	  CS_{\beta_0^\delta}(B_\delta')=CS_{\beta_0^r}(B_r)+\mathfrak E_r(A,u)-\mathfrak E_\delta(A,u).
	\end{equation}
	Therefore, if $r$ is small enough, we can guarantee that $CS_{\beta_0^\delta}(B_\delta')$ is also less than any given real number.
	Since $B_\delta'$ is gauge equivalent to $B_\delta$, \eqref{form538} implies that 
	$CS_{\beta_0^\delta}(B_\delta')=CS_{\beta_0^\delta}(B_\delta)$. Now the result follows by taking the limit of \eqref{difference-CS}
	as $\delta$ goes to zero.
\end{proof}

\begin{proof}[Proof of Proposition \ref{decayenergy}]
	Let $\beta = \frac{1}{2K}$ where $K$ is given by Lemma \ref{lem527}. Lemmas \ref{lem527} and \ref{lem535} imply
	\[
	  \frac{d}{dr}(r^{-2\beta}\mathfrak E_r(A,u))=-\frac{1}{K} r^{-2\beta-1}\mathfrak E_r(A,u) + r^{-2\beta}\frac{d}{dr} \mathfrak E_r(A,u) \ge 0.
	\]
	Consequently, we have
	\[
	  \mathfrak E_r(A,u) \le r^{2\beta} r_0^{-2\beta}\mathfrak E_{r_0}(A,u)
	\]
	for $r \le r_0$, which completes the proof of Proposition \ref{decayenergy}.
\end{proof}

\subsection{Gromov-Uhlenbeck compactness}\label{compact}

Let $(A_i,u_i)$ be a sequence of solutions of the mixed equation associated to the mixed pair in \eqref{matching-quintuple}, which we copy here again:
\begin{equation}\label{quintuple-gen-comp}
  \fP:=(X,V,S,\mathcal M(\Sigma,F),\bbL).
\end{equation}
The surface $S$ has a distinguished boundary component $U_\partial$ and as usual a Lagrangian in $\mathcal M(\Sigma,F)$ is associated to each of the remaining boundary components. We also require that there is a uniform bound $K$ on the energy of the mixed pairs $(A_i,u_i)$ given as 
\begin{equation}\label{top-energy}
  \mathfrak E(A_i,u_i)=\int_{X}|F_{A_i}|^2\dvol_X+\int_{S}|du|^2\dvol_S.
\end{equation}

\begin{lemma}\label{finite-bubble-points}
	There are finite sets of points $\sigma_- \subset {\rm int} (X)$, $\sigma_+ \subset S\setminus U_\partial$, $\sigma_\partial \subset U_\partial$ and a subsequence of $\{(A_i,u_i)\}$ such that 
	the following holds. For any point $z\in U_\partial \setminus \sigma_\partial$ there is a positive real number $r_z$ such that 
	\begin{equation}\label{energy-mixed}
	  \int_{D_{r_z}(z)\times \Sigma} |F_{A_i}|^2\dvol+ \int_{D_{r_z}^+(z)}|du|^2\dvol\leq \hbar,
	\end{equation}
	for large enough values of $i$. 
	Similarly for any $z\in {\rm int} (X)\setminus \sigma_-$, there is $r_z$ such that
	\begin{equation}\label{energy-conn}
	  \int_{B_{r_z}(z)}|F_{A_i}|^2\dvol\leq \hbar,
	\end{equation}
	if $i$ is large enough, and for any point $z\in S\setminus \(\sigma_+\cup U_\partial\)$ 
	there is a positive real number $r_z$ such that 
	\begin{equation}\label{energy-map}
	  \int_{B_{r_z}(z)}|du|^2\dvol\leq \hbar,
	\end{equation}
	if $i$ is large enough.
\end{lemma}
	
\begin{proof}
	The boundary of $X$ can be identified with $U_\partial \times \Sigma$ and a tubular neighborhood of it is given as $U_-\times \Sigma$ where $U_-=(-1,0]\times U_\partial$.
	Similarly a boundary component of $S$ is $U_\partial$ and we may take a regular neighborhood $U_+$ of this boundary component where $U_+=[0,1)\times U_\partial$.
	For each $i$, we may define a positive measure on $U_-$ where a continuous function $f:U_-\to \R$ with compact support integrates to 
	\[
	  \int_{U_-}f(s,\theta)\(|du|^2(-s,\theta)+\int_{\{(s,\theta)\}\times \Sigma} |F_{A_i}|^2\)\dvol
	\]
	with respect to this measure. Standard compactness theorems for measures imply that there is a subsequence of these measures convergent to a positive measure $\mu_\partial$ on $X$
	in the sense that for any continuous function $f$ with compact support, we have
	\[
	  \int_{U_-}f(s,\theta)\(|du|^2(-s,\theta)+\int_{\{(s,\theta)\}\times \Sigma} |F_{A_i}|^2\)\dvol \to \int_{U_-}f\mu_\partial
	\]
	In particular, the measure of $U_-$ with respect to $\mu_\partial$ is at most $K$, the uniform bound on \eqref{top-energy}. Let $\sigma_\partial$ be the set of points $z\in U_\partial$ such that 
	the $\mu_\partial$-measure of the ball $D_r^-(z)$ for any $r$ is at least $\hbar$.
	Since the measure of $U_-$ is at most $K$, the set $\sigma_\partial$ is finite and has at most $K/\hbar$ elements. From the definition of $\mu_\partial$ it is clear that for any point 
	$z\in U_\partial\setminus \sigma_\partial$, the inequality in \eqref{energy-mixed}
	holds for an appropriate choice of $r_z$ and large enough values of $i$. The sets $\sigma_-$ and $\sigma_+$ can be obtained similarly by firstly defining positive measures $\mu_-$ and 
	$\mu_+$ on $X$ and $S$, and then considering the points with concentrated measure density. 
\end{proof}

For $z\in U_\partial\setminus \sigma_\partial$, let $r_z$ be given as in the lemma, which we denote by $r$ for the ease of notation. Temporarily, we denote the restrictions of $A_i$ and $u_i$ to $D_r(z)\times \Sigma$ and $D_r^+(z)$ with the same notation. Theorem \ref{en-quant} implies that 
\[
   \hspace{2cm} \vert du_i(w)\vert^2 \leq \kappa \frac{\hbar}{r^2} \hspace{1cm}\text{for $w\in D_r^+(z)$}.
\]
Thus, after passing to a subsequence, $u_i$ is $C^0$ convergent to $u_0:D_r^-(z)\to \mathcal  M(\Sigma,F)$. In fact, using Lemma \ref{L21-energy} and Sobolev embedding, we may assume that for a given $p$, the subsequence is chosen such that it is convergent to $u_0$ in the $L^p_1$ norm.

We may use Theorem \ref{en-quant} and Lemma \ref{L21-energy} to obtain the bounds
\[
  |\!|\nabla_{A_i}F_{A_i}|\!|_{L^2(D^-_{r}(z)\times \Sigma)}^2\leq  \kappa \frac{\hbar}{r^2},\hspace{2cm}|\!|F_{A_i}|\!|_{L^2(\{w\}\times \Sigma)}^2\leq  \kappa \frac{\hbar}{r^2}
   \hspace{.5cm}\text{for $w\in D_r^-(z)$}.
\]
In particular, $|\!|F_{A_i}|\!|_{L^4(D^-_{r}(z)\times \Sigma)}$ is unifromly bounded using Kato's inequality and Sobolev embedding theorem. Therefore, we may apply the Uhlenbeck compactness theorem for the manifold $D^-_{r}(z)\times \Sigma$ to conclude that there are $L^4_2$ gauge transformations $g_i$ such that $g_i^*A_i$, after passing to a subsequence, is weakly convergent to $A_0$ in $L^4_1$ \cite{Uh:com,Weh:com}. In particular, the $L^4_1$ norms of the connections $g_i^*A_i$ are uniformly bounded. For large enough values of $i$, we may put the connections $g_i^*A_i$ in the chosen subsequence in the Coulomb gauge with respect to a smooth connection $A_0'$ which is close enough to $A_0$ in the $L^4_1$ norm \cite[Theorem F]{Weh:com}. For the ease of notation, we denote $g_i^*A_i$ after applying the second gauge transformation by $A_i$. The Coulomb gauge condition on $A_i$ asserts that
\begin{equation}\label{Coulomb-gauge-fix}
  d_{A_0'}^*(A_i-A_0')=0,\,\hspace{2cm}*(A_i-A_0')\vert_{U_\partial\times \Sigma}=0.
\end{equation}

We claim that $d_{A_0'}A_i$ is uniformly bounded in $L^2_1$. First note that
\begin{equation}\label{d-0-A-i}
  d_{A_0'}A_i=d_{A_0'}A_0-(A_i-A_0')\wedge (A_i-A_0')+F_{A_i}-F_{A_0'}
\end{equation}
Since $|\!|\nabla_{A_i}F_{A_i}|\!|_{L^2}$, $|\!|F_{A_i}|\!|_{L^4}$ and $|\!|A_i|\!|_{L^4}$ are all uniformly bonded, the term $F_{A_i}$ in \eqref{d-0-A-i} is uniformly bounded in $L^2_1$, too. Similarly, the term $(A_i-A_0')\wedge (A_i-A_0')$ is uniformly bounded in $L^2_1$ because there is a uniform bound on $|\!|A_i|\!|_{L^4_1}$. Thus our claim about $d_{A_0'}A_i$ follows. Sobolev embedding together with uniform boundedness of $\{d_{A_0'}A_i\}_i$ implies that for any $p<4$, the sequence $\{d_{A_0'}A_i\}_i$ is $L^p$ convergent after passing to a subsequence. This together with \eqref{Coulomb-gauge-fix} implies that $A_i$ are $L^p_1$ convergent to $A_0$ over $D^-_{r'}(z)\times \Sigma$ for $r'<r$. Now we use Theorem \ref{regularity-sequential} to show that $(A_i,u_i)$ is in fact $C^\infty$ convergent to $(A_0,u_0)$.

By applying a similar argument to each point $z$ in the complement of $\sigma_-\cup U_\partial\times \Sigma$ in $X$, we may obtain gauge transformations for the restriction of $A_i$ to an open neighborhood $D_r(z)\subset X$ such that after applying these gauge transformations and passing to a subsequence the resulting connections are $C^\infty$ convergent to an ASD connection. On the symplectic side, for any point $z$ in the complement of $\sigma_+\cup U_\partial$ in $S$, there is a disc neighborhood $D_r(z)\subset S$ such that the restriction of $u_i$ to this neighborhood is $C^\infty$ convergent to a holomorphic map from $D_r(z)$ to $\mathcal M(\Sigma,F)$. We may patch together these holomorphic maps together to define a holomorphic map $u_0:S\setminus \fS_+\to \mathcal M(\Sigma,F)$ where $\fS_+=\gamma_+\cup \gamma_\partial$. Then the maps $u_i$ are $C^\infty$ convergent on compact subspaces of $S\setminus \fS_+$ to $u_0$.

We may also define a connection $A_0$ on $X\setminus \fS_-$ where $\fS_-=\gamma_-\cup \gamma_\partial \times \Sigma$. The patching argument of  \cite[Section 4.4.2]{DK} can be used to find a subsequence of $\{A_i\}$ and a gauge transformations $g_i$ defined on $X\setminus \fS_-$ such that $g_i^*A_i$, after passing to the subsequence, is $C^\infty$ convergent to an ASD connection $A_0$ on any compact subspace of $X\setminus \fS_-$. The pair $(A_0,u_0)$ defines a solution of the mixed equation for the quintuple 
\[
   \fP':=(X\setminus \fS_-,V|_{X\setminus \fS_-},S\setminus \fS_+,\mathcal M(\Sigma,F),\mathcal L(\Sigma,F)).
\]
Since \eqref{top-energy} is bounded by $K$, we have $\mathfrak E(A_0,u_0)\leq K$. Moreover, if at least one of $\fS_+$ and $\fS_-$ is non-empty, then $\mathfrak E(A_0,u_0)\leq K-\hbar$. Applying the results of Subsection \ref{remove-sing}, removability of singularity for ASD connections \cite{Uh:remov-sing,DK} and removability of singularity for holomorphic maps \cite{MS:J-holo-symp} implies that $(A_0,u_0)$ can be extended to a solution of the mixed equation for the quintuple 
\begin{equation}
   \fP_0:=(X,V',S,\mathcal M(\Sigma,F),\mathcal L(\Sigma,F)).
\end{equation}
where $V'$ is an $\SO(3)$ bundle over $X$ whose restriction to $X\setminus \fS_-,$ is isomorphic to $V$. This completes the proof of Theorem \ref{GU-comp} in the introduction.

 \section{Fredholm property}\label{Fredholm-section}

Our goal in this section is to address Theorem \ref{Fred-mixed-op}. In fact, it would be more convenient if we work on a slightly more generalized version of the theorem. First we write an explicit formula for the mixed operator $\mathcal D_{(A,u)}$ for a smooth mixed pair $(A,u)$ associated to the cylinder quintuple $\fc_I$ in \eqref{mixed-cylinder}. For any $(\zeta,\nu)$, we have 
\begin{equation}\label{mixed-op}
	\mathcal D_{(A,u)}(\zeta,\nu):=(d_A^+\zeta,-d_A^*\zeta,\nabla_\theta \nu-J_{s,\theta}(u)\nabla_{s}\nu-(\nabla_\nu J_{s,\theta})\frac{du}{ds}).
\end{equation}
The first component is the linearization of the ASD equation $F^+(A)=0$ at the connection $A$ and the third term is the linearization of the Cauchy-Riemann equation
\[
  \frac{du}{d\theta}-J_{s,\theta}(u)\frac{du}{ds}=0.
\]
Here we use the Levi-Civita connection on $M$ defined using a Riemannian metric on $M$ (possibly the metric induced by a compatible almost complex structure) to define the covariant derivatives 
$\nabla_s$ and $\nabla_\theta$ in the $s$ and $\theta$ directions. The middle term in $\mathcal D_{(A,u)}(\zeta,\nu)$ is given by the Coulomb gauge fixing condition.

It is helpful to rewrite $\mathcal D_{(A,u)}$ in a from which makes use of the product structure of $\fc_I$. By applying a gauge transformation, we may assume that the connection $A$ on $Y\times I$ is in temporal gauge and hence is determined by its restrictions $B_\theta$ to $Y\times \{\theta\}$ for $\theta\in I$. The restriction of $B_\theta$ to $\Sigma=\partial Y$, which is a flat connection, is denoted by $\alpha_\theta$. Any element $\zeta$ of $\Omega^1(X,E)$ can be written as $\zeta=b+\varphi d\theta$. Thus we may identify $\Omega^1(X,E)$ with maps from $I$ to $\Omega^1(Y,E))\oplus \Omega^0(Y,E)$. Using this presentation of $\zeta$, we have
\begin{align*}
	d_A^*\zeta&=d_B^*b-\frac{d\varphi}{d\theta},\\
	d_A^+\zeta&=\frac{1}{2}\left[d\theta\wedge(-*_3d_Bb+\frac{db}{d\theta}-d_B\varphi)-*_3(-*_3d_Bb+\frac{db}{d\theta}-d_B\varphi)\right].
\end{align*}
An element of  $\Omega^+(X,E)$ can be also identified with a map from $I$ to $\Omega^1(Y,E)$ by sending a self-dual 2-form $\frac{1}{2}(d\theta\wedge b-*_3b)$ to $b$. Similarly, any element of $\Omega^0(X,E)$ can be identified with ${\rm Map}(I, \Omega^0(Y,E))$ in the obvious way.

To study the last component of $\mathcal D_{(A,u)}$, fix a Hermitian isomorphism of $u^*T\mathcal M(\Sigma,F)$ with the trivial bundle with fiber $(\R^{2n},J_0,\omega_0)$. Here $J_0$ and $\omega_0$ are standard complex and symplectic structures on the Euclidean space $\R^{2n}$ with $2n$ being the dimension of $M$. Note that the almost complex structure on the fiber of $u^*T\mathcal M(\Sigma,F)$ over $(s,\theta)$ is given by $J_{s,\theta}$. This reparametrization of $u^*TM$ allows us to regard $\nu$ as a map  $[0,1]\times I \to \R^{2n}$, and then the third component of \eqref{mixed-op} can be written as 
\begin{equation}\label{linear-CR}
	\frac {d\nu}{d\theta}-J_0\frac{d\nu}{ds}-S(\nu),
\end{equation}
where $S$ is a map from $[0,1]\times I$ to the space of endomorphisms of $\R^{2n}$. Moreover, for any $\theta\in I$, there is a family of Lagrangian subspaces $L_\theta\subset \R^{2n}$ and subspaces $\mathcal L_\theta$ of $L^2(\Sigma,\Lambda^1\otimes F)\oplus \R^{2n}$ such that 
\[
  *\zeta|_{\Sigma\times I}=0,\hspace{1cm}(\zeta\vert_{\Sigma\times \theta},\nu(0,\theta))\in \mathcal L_\theta,\hspace{1cm}\nu(1,\theta)\in L_\theta.
\]
Here $\mathcal L_\theta$ is a {\it canonical linearized Lagrangian correspondence from $\Omega^1(\Sigma,F)$ to $\R^{2n}$ compatible with $\alpha_\theta$}, whose definition is given below.

\begin{definition}
	Suppose $\alpha$ is a flat connection on $F$. A canonical linearized Lagrangian correspondence $\mathcal L$ from $\Omega^1(\Sigma,F)$ to $\R^{2n}$ compatible with $\alpha$ is determined by a Lagrangian subspace $V$ of $\mathcal H^1(\Sigma;\alpha)\times \R^{2n}$ with 
	respect to the symplectic form $-\omega_{\rm fl}\oplus \omega$. The space $\mathcal L$ is the subspace of $L^2(\Sigma,\Lambda^1\otimes F)\oplus \R^{2n}$ given by elements of the form
	\[
	  v+(d_{\alpha}\zeta,0)
	\]
	where $\zeta\in L^2_1(\Sigma,F)$. We write $\mathcal L_{\alpha,V}$ for this Lagrangian correspondence if we want to clarify the choices of $\alpha$ and $V$.
\end{definition}

Given $\mathcal L$ as in the above definition, the pairing of any two elements of $\mathcal L$ with respect to the symplectic form $-\Omega\oplus \omega_0$ vanishes and as a consequence of Lemma \ref{Hodge-decom}, an element of $L^2(\Sigma,\Lambda^1\otimes F)\oplus \R^{2n}$ belongs to $\mathcal L$ if its pairing with all elements of $\mathcal L$ vanish. The $L^2$ closures of the tangent spaces of a canonical Lagrangian correspondence from $\mathcal A^p(\Sigma,F)$ to $\R^{2n}$ give rise to instances of $\mathcal L$.

Motivated by the above discussion, we may slightly relax the definition of the mixed operator. Suppose $A$ is a smooth connection on $E\times I$ over $Y\times I$ such that the restriction $\alpha_\theta$ of $A$ to $\Sigma\times \{\theta\}$ is flat for any $\theta\in I$. Suppose $S$ is a smooth map from $[0,1]\times I$ to the space of endomorphisms of $\R^{2n}$. Following the same convention as above, $A$ is in temporal gauge and its restriction to $Y\times \{\theta\}$ is denoted by $B_\theta$. Similarly, we write $S_\theta$ for the restriction of $S$ to $[0,1]\times \{\theta\}$. For any $\theta\in I$, suppose $V_\theta$ is a Lagrangian subspace of $\mathcal H^1(\Sigma;\alpha_\theta)\times \R^{2n}$ with respect to the symplectic form $-\omega_{\rm fl}\oplus \omega$ and $L_\theta$ is a Lagrangian subspace of $\R^{2n}$ with respect to the symplectic form $\omega_0$. We assume that both of $V_\theta$ and $L_\theta$ depend smoothly on $\theta$. Let $\mathcal L_\theta=\mathcal L_{\alpha_\theta,V_\theta}$.

We define a differential operator $\mathcal D_{(A,S)}$ associated to $A$, $S$, $\fL=\{\mathcal L_\theta,L_\theta\}_\theta$. Fix a positive integer $k$. Similar to $E^k_{(A,u)}(I)$, let $E^k_{\fL}(I)$ be the space of pairs 
\begin{equation}\label{zeta-nu-dom-A-S}
  \zeta\in L^2_{k}(Y\times I,\Lambda^1\otimes E),\hspace{1cm} \nu \in L^2_{k}([0,1]\times I,\R^{2n}),
\end{equation}
satisfying the boundary conditions
\begin{equation}\label{matching-bdry-relaxed}
  *\zeta|_{\Sigma\times I}=0,\hspace{1cm}(\zeta\vert_{\Sigma\times \{\theta\}},\nu(0,\theta))\in \mathcal L_\theta,\hspace{1cm}\nu(1,\theta)\in L_\theta.
\end{equation}
Let $\mathcal D_{(A,S)}$ be the linear map with the domain $E^k_{\fL}(I)$ defined as
\begin{equation}\label{D-A-S}
	\mathcal D_{(A,S)}:=\frac{d}{d\theta}-\fD_{(B_\theta,S_\theta)},
\end{equation}
where 
\begin{equation}\label{DBp-theta}
  \fD_{(B_\theta,S_\theta)}(\varphi,b,\nu)=(d_{B_\theta}^*b,*d_{B_\theta}b+d_{B_\theta}\varphi,J_0\frac{d\nu}{ds}+S_\theta(\nu)).
\end{equation}
Here we use the presentation of $\zeta$ in \eqref{zeta-nu-dom-A-S} as $b+\varphi d\theta$ with $\varphi$ being a section of $E$ and $b$ being a 1-form on $Y$ with values in $E$. The target of the operator $\mathcal D_{(A,S)}$ is given as
\begin{equation}\label{target-flex-lin}
  L^2_{k-1}(Y\times I,\Lambda^1\otimes E)\oplus L^2_{k-1}([0,1]\times I,\R^{2n}).
\end{equation}
A straightforward integration by parts shows that the formal adjoint of $\mathcal D_{(A,S)}$, characterized by the analogue of \eqref{adjoint}, is equal to
\[
  \frac{d}{d\theta}+\fD_{(B_\theta,S_\theta^*)}.
\]

\begin{theorem}\label{Fredholm-flex}
	For any open interval $J$ that its closure is a compact subset of $I$, the following holds. 
	\vspace{-5pt}
	\begin{itemize}
	\item[(i)] For $k\geq 1$, suppose $(\zeta,\nu)\in E^1_{\fL}(I)$ and $\mathcal D_{(A,S)}(\zeta,\nu)$ is in $L^2_{k-1}$. Then $(\zeta,\nu)\in E^k_{\fL}(J)$. Moreover, there is a constant $C$, depending only on $(A,S)$ and $k$, such that
	\begin{equation}\label{Fed-prop-k-I-J}
		|\!|(\zeta,\nu)|\!|_{L^2_{k}(J)}\leq C\left(|\!|\mathcal D_{(A,S)}(\zeta,\nu)|\!|_{L^2_{k-1}(I)}+|\!|(\zeta,\nu)|\!|_{L^2(I)}\right).
	\end{equation}
	\item[(ii)] Suppose $(\zeta,\nu)$ is as in \eqref{target-flex-lin} for $k=1$, and there is a constant $\kappa$ such that
	\[
	   \langle (\zeta,\nu),\mathcal D_{(A,S)}(\xi,\eta)\rangle_{L^2} \leq \kappa |\!|(\xi,\eta)|\!|_{L^2(I)}
	\]
	for any smooth and compactly supported $(\xi,\eta)$ as in \eqref{zeta-nu-dom-A-S} satisfying \eqref{matching-bdry-relaxed}. Then $(\zeta,\nu)$ is in $E^1_{\fL}(J)$. Moreover, there is a constant $C$, depending only on $(A,S)$, such that
	\begin{equation}
		|\!|(\zeta,\nu)|\!|_{L^2_{1}(J)}\leq C\left(|\!|\mathcal D_{(A,S)}(\zeta,\nu)|\!|_{L^2(I)}+|\!|(\zeta,\nu)|\!|_{L^2(I)}\right).
	\end{equation}	\end{itemize}
\end{theorem}
Verifying Theorem \ref{Fredholm-flex} in the case that $S$ is replaced with $S+S^*$ is sufficient for proving the theorem in the general case. In particular, we can assume that $S$ takes values in self-adjoint transformations of $\R^{2n}$, and this is the assumption that we make for the rest of the section. As we shall see in Subsection \ref{unbdd-self-adj-3D}, this assumption allows us to show that $\fD_{(B_\theta,S_\theta)}$ is an unbounded self-adjoint operator. Then we use general Fredholm property results about operators which have the form 
\begin{equation}\label{op-4D}
	\frac{d}{d\theta}-D_\theta
\end{equation}
for a family of unbounded self-adjoint operators $D_\theta$. The additional layer of difficulty here is that the domain of the operators $\fD_{(B_\theta,S_\theta)}$ depends on $\theta$. To resolve this issue, we will have a closer look at the domain of these operators in Subsection \ref{domain-lemmas} and show that the variation in the domains of these operators can be controlled in a nice way. Then we conclude Theorem \ref{Fredholm-flex} from the results of \cite{SW:I-bdry} about Fredholm property of operators of the form \eqref{op-4D} where the domains of $D_\theta$ are not constant. Our discussion above shows that Theorem \ref{Fred-mixed-op} follows from Theorem \ref{Fredholm-flex}.

\subsection{The Hilbert space $\mathcal W$}\label{domain-lemmas}

Let $\mathcal H$ be the Hilbert space given as the completion of smooth triples 
\begin{equation}\label{triple-Hilbert}
	(\varphi,b,\nu)\in \Omega^0(Y,E)\oplus \Omega^1(Y,E)\oplus \Omega^0([0,1],\R^{2n}),
\end{equation}
equipped with the $L^2$ inner product
\begin{equation}
	 \langle (\varphi_0,b_0,\nu_0),(\varphi_1,b_1,\nu_1)\rangle_{L^2}:=\int_Y\tr\left(\varphi_0\wedge *\varphi_1+b_0\wedge *b_1\right)+\int_{0}^{1}\omega_0(\nu_0(s),J_0\nu_1(s) )ds.
\end{equation}
In this subsection and the next one, we write $*$ for the 3-dimensional Hodge operator, and the Hodge operator on $\Sigma$ is denoted by $*_2$ as before. Suppose $B$ is a connection whose restriction $\alpha$ to $\Sigma$ is flat and $S$ is a map from $[0,1]$ to the space of self-adjoint linear transformations of $\R^{2n}$. In the last subsection, we introduced
\begin{equation}\label{DBp}
  \fD_{(B,S)}(\varphi,b,\nu)=(d_B^*b,*d_Bb+d_B\varphi,J_0\frac{d\nu}{ds}+S(\nu)),
\end{equation}
which can be regarded as an unbounded operator on $\mathcal H$.

We fix a domain for $\fD_{(B,S)}$ using a Lagrangian $L\subset \R^{2n}$ and a canonical linearized Lagrangian correspondence $\mathcal L$ from $\Omega^1(\Sigma,F)$ to $\R^{2n}$ compatible with $\alpha$. Thus, $\mathcal L=\mathcal L_{\alpha,V}$ for a Lagrangian subspace $V$ of $\mathcal H^1(\Sigma;\alpha)\oplus \R^{2n}$. Let $\mathcal W$ denote the $L^2_1$ completion of the space of all triples $(\varphi,b,\nu)$ as in \eqref{triple-Hilbert} such that
\begin{equation}\label{match-bdry-cond}
	*b|_{\Sigma}=0,\hspace{1cm}(b|_{\Sigma},\nu(0))\in \mathcal L,\hspace{1cm} \nu(1)\in L.
\end{equation}
Clearly, $\mathcal W$ is a dense subspace of $\mathcal H$ because any element of $\mathcal H$ can be obtained as the $L^2$ limit of a sequence of triples $(\varphi_i,b_i,\nu_i)$ as in \eqref{triple-Hilbert} such that such that $b_i$ vanishes in a neighborhood of the boundary $\Sigma$ of $Y$ and $\nu_i$ vanishes in a neighborhood of the boundary of $[0,1]$. We fix the $L^2_1$ inner product on $\mathcal W$ where the $L^2_1$ inner product on sections of $E$ and $\Lambda^1\otimes E$ are defined using the connection $B$ and the Levi-Civita connection associated to the metric on $Y$. Sobolev embedding implies that the inclusion of $\mathcal W$ into $\mathcal H$ is compact.

The Hodge decomposition in Lemma \ref{Hodge-decom} allows us to give a useful description of the Hilbert spaces $\mathcal H$ and $\mathcal W$. Fix a tubular neighborhood $(-\epsilon,0]\times \Sigma$ of the boundary of $\Sigma$. Using \eqref{fun-space-Ban}, we have the identification 
\begin{equation}\label{splitting-Sobolev}
  L^2_k((-\epsilon,0]\times \Sigma,\Lambda^1_\Sigma \otimes E)=L^2_k((-\epsilon,0],L^2(\Sigma,\Lambda^1_\Sigma \otimes F))\cap L^2((-\epsilon,0],L^2_k(\Sigma,\Lambda^1_\Sigma \otimes F)).
\end{equation}
Lemma \ref{Hodge-decom} asserts that we have the continuous splitting
\[
  L^2_{k}(\Sigma,\Lambda^1_\Sigma \otimes F)\cong \mathcal H^1(\Sigma;\alpha)\oplus L^2_{k+1}(\Sigma,F)\oplus L^2_{k+1}(\Sigma,F),
\]
for any non-negative integer $k$. Therefore, we have
\begin{equation}
  L^2_k((-\epsilon,0]\times \Sigma,\Lambda^1_\Sigma \otimes E)=L^2_k((-\epsilon,0],\mathcal H^1(\Sigma;\alpha))\oplus Z_k\oplus Z_k,
\end{equation}
where 
\begin{equation}\label{space-Z}
	Z_k=L^2_k((-\epsilon,0],L^2_{1}(\Sigma,F))\cap L^2((-\epsilon,0],L^2_{k+1}(\Sigma,F)).
\end{equation}
Given the canonical Lagrangian correspondence $\mathcal L_{\alpha,V}$ from $\Omega^1(\Sigma,F)$ to $\R^{2n}$, an element of 
\begin{equation}\label{domain-H-W}
	L^2_k((-\epsilon,0]\times \Sigma,\Lambda^1_\Sigma \otimes E)\oplus L^2_k([0,\epsilon),\R^{2n})
\end{equation}
can be written as 
\[
  v+\bJ v'+(d_\alpha \zeta+*_2 d_\alpha \zeta',0),
\]
where $v,v' \in L^2_k((-\epsilon,0],V)$, $\zeta,\zeta'\in Z_k$ and $\bJ$ is defined as in \eqref{bJ} using $J_0$. For each $s\in (-\epsilon,0]$, $v(s),v'(s)\in V$ has a component in $\mathcal H^1(\Sigma;\alpha)$, and this component for different values of $s$ gives rise to an element of the first summand of \eqref{domain-H-W}. The components of $v(s),v'(s)$ in $\R^{2n}$ define an element of $L^2_k((-\epsilon,0],\R^{2n})$, which we identify with an element of the  second summand of \eqref{domain-H-W} by precomposing with the map $s\to -s$ from $[0,\epsilon)$ to $(-\epsilon,0]$. Moreover, the exterior derivative $d_\alpha$ in the above expression is only taken in the $\Sigma$ direction and hence $d_\alpha$ maps an element of $Z_k$ to an element of \eqref{splitting-Sobolev}.

Suppose $(\varphi,b,\nu)\in \mathcal H$. We focus on the restriction of $(\varphi,b)$ to the subspace $(-\epsilon,0]\times \Sigma$ of $Y$ and the restriction of $\nu$ to the interval $[0,\epsilon)$ of $[0,1]$, and by a slight abuse of notation use the same notations to denote these restrictions. Then the $1$-form $b$ has the form
\begin{equation}\label{decomp-b-bdry}
	q+\tau ds,
\end{equation}
where $s$ denotes the coordinate on $(-\epsilon,0]$, $q\in L^2((-\epsilon,0]\times \Sigma,\Lambda^1_\Sigma \otimes E)$ and $\tau \in L^2((-\epsilon,0]\times \Sigma,E)$. Using the discussion of the previous paragraph, the pair $q$ and $b$ can be reparametrized by 
\[
  v,v'\in L^2((-\epsilon,0],V),\hspace{1cm}\zeta,\zeta'\in Z_0.
\]
Then $(\varphi,b,\nu)\in \mathcal W$ is equivalent to require that $v$, $v'$ and $\tau$ are in $L^2_1$, $\zeta$, $\zeta'$ are in $Z_1$, $\tau|_{\{0\}\times \Sigma}=0$, $v'(0)=0$, $\zeta'(0)=0$.

We shall use this discussion to construct isomorphisms between the Hilbert spaces $\mathcal W$ associated to two different choices of $(\alpha,V,L)$ that are close to each other. First we need the following lemma which allows us to identify the vector spaces $\mathcal H^1(\Sigma;\alpha)$ associated to choices of $\alpha$ that are close to each other.

\begin{lemma}\label{cons-Phi-alpha}
	Fix a flat connection $\alpha_0$ on $F$. There is a positive constant $\epsilon$ such that if $\alpha$ is another flat connection on $F$ with $|\!|\alpha-\alpha_0|\!|_{L^2_{1}}<\epsilon$, then there is an isomorphism $\Phi_\alpha:\mathcal H^1(\Sigma;\alpha_0) \to \mathcal H^1(\Sigma;\alpha)$ and a constant $C$
	such that
	\begin{equation}
		|\!|\Phi_\alpha(v)-v|\!|_{L^2_1}\leq C|\!|\alpha-\alpha_0|\!|_{L^2}|\!|v|\!|_{L^2}.
	\end{equation}
	More generally, there are positive constants $\epsilon_k$ and $C_k$ such that if $|\!|\alpha-\alpha_0|\!|_{L^2_{k}}<\epsilon_k$, then
	\begin{equation}
		|\!|\Phi_\alpha(v)-v|\!|_{L^2_k}\leq C|\!|\alpha-\alpha_0|\!|_{L^2_{k-1}}|\!|v|\!|_{L^2}.
	\end{equation}	
\end{lemma}
\begin{proof}
	For any positive integer $k$ and any flat connection $\alpha$ on $F$, the twisted Laplace operator \[\Delta_\alpha=d_\alpha^*d_\alpha:L^2_{k+1}(\Sigma,F)\to L^2_{k-1}(\Sigma,F)\] is invertible, and we denote the inverse by $G_\alpha$. 
	It is straightforward to see that there are 
	positive constants $\epsilon_k$ and $C_k$ such that if $|\!|\alpha-\alpha_0|\!|_{L^2_{k}}<\epsilon_k$, then we have
	\begin{equation}\label{ineq-Green}
	  |\!|d_\alpha \zeta |\!|_{L^2_k}\leq C_k|\!|\zeta |\!|_{L^2_{k+1}}, \hspace{1cm}|\!|G_\alpha(\zeta)|\!|_{L^2_{k+1}}\leq C_k|\!|\zeta |\!|_{L^2_{k-1}}.
	\end{equation}

	For any $v\in \mathcal H^1(\Sigma;\alpha_0)$, define
	\begin{equation}\label{Phi-alpha-def}
	  \Phi_\alpha(v):=v+d_\alpha G_\alpha(*_2[\alpha-\alpha_0,*_2v])+*_2d_\alpha G_\alpha(*_2[\alpha-\alpha_0,v]).
	\end{equation}
	Then we have
	\begin{align*}
		d_\alpha\Phi_\alpha(v)&=d_\alpha v+d_\alpha*_2d_\alpha G_\alpha(*_2[\alpha-\alpha_0,v])\\
		&=(d_{\alpha_0} v+[\alpha-\alpha_0,v])-*_2\Delta_\alpha  G_\alpha(*_2[\alpha-\alpha_0,v])\\
		&=0,
	\end{align*}
	where in the last identity we use the assumption that $d_{\alpha_0} v=0$. Using a similar argument we have
	\begin{align*}
		d_\alpha^*\Phi_\alpha(v)&=d_\alpha^* v+\Delta_\alpha G_\alpha(*_2[\alpha-\alpha_0,*_2v])\\
		&=(d_{\alpha_0}^* v-*_2[\alpha-\alpha_0,*_2v])+(*_2[\alpha-\alpha_0,*_2v])\\
		&=0.
	\end{align*}
	In particular, we have $\Phi_\alpha(v)\in \mathcal H^1(\Sigma;\alpha)$. 
	Using \eqref{ineq-Green},  we can conclude the following inequalities where in each line we might need to increase the value of $C_k$ in compare to the previous one:
	\begin{align*}
		|\!|\Phi_\alpha(v)-v|\!|_{L^2_k}&\leq |\!|d_\alpha G_\alpha(*_2[\alpha-\alpha_0,*_2v])|\!|_{L^2_k}+|\!|*_2d_\alpha G_\alpha(*_2[\alpha-\alpha_0,v])|\!|_{L^2_k}\\
		&\leq C_k\(|\!|G_\alpha(*_2[\alpha-\alpha_0,*_2v])|\!|_{L^2_{k+1}}+|\!|G_\alpha(*_2[\alpha-\alpha_0,v])|\!|_{L^2_{k+1}}\)\\
		&\leq C_k\(|\!|*_2[\alpha-\alpha_0,*_2v]|\!|_{L^2_{k-1}}+|\!|(*_2[\alpha-\alpha_0,v]|\!|_{L^2_{k-1}}\)\\
		&\leq C_k |\!|\alpha-\alpha_0|\!|_{L^2_{k-1}}|\!|v|\!|_{L^2}.
	\end{align*}
	This, in particular, shows that after possibly decreasing the value of $\epsilon_k$, $\Phi_\alpha:\mathcal H^1(\Sigma;\alpha_0)\to \mathcal H^1(\Sigma;\alpha)$ is an isomorphism of vector spaces.
\end{proof}

Fix a triple $(\alpha_0,V_0,L_0)$, and let $\mathcal U$ be the space of all triples $(\alpha,V,L)$ such that:
\begin{enumerate}
	\item[(i)] $|\!|\alpha-\alpha_0|\!|_{L^2_{1}}$ is less than the constant $\epsilon$ provided by Lemma \ref{cons-Phi-alpha};
	\item[(ii)] $V$ has a trivial intersection with $\Phi_\alpha(\bJ_0(V_0))$ and $\bJ V$ has trivial intersection with $\Phi_\alpha(V_0)$;
	\item[(iii)] $L \cap J_0L_0=0$.
\end{enumerate}
In (ii), $\bJ_0$ and $\bJ$ are respectively the almost complex structures on $\mathcal H^1(\Sigma;\alpha_0)\oplus \R^{2n}$ and $\mathcal H^1(\Sigma;\alpha)\oplus \R^{2n}$ given as $(-J_*,J_0)$. For any $(\alpha,V,L)\in \mathcal U$, there is a Linear map $R: L_0\to J_0L_0$ such that $L$ is given by the subspace of $\R^{2n}$ consisting of $x+R(x)$ with $x\in L_0$. We define the distance between $L$ and $L_0$, denoted by $d(L,L_0)$, to be the norm of the linear map $R$. Similarly, there is a linear map $\fR: V_0 \to \bJ_0 V_0$ (resp. $\fR': \bJ_0 V_0 \to  V_0$) such that $\Phi_\alpha^{-1}(V)$ (resp. $\Phi_\alpha^{-1}(\bJ V)$) is given by the subspace of $\mathcal H^1(\Sigma;\alpha_0)\oplus \R^{2n}$ consisting of $v+\fR(v)$ (resp. $v+\fR'(v)$) with $v\in V_0$ (resp. $v\in \bJ_0V_0$). We define the distance between $V$ and $V_0$, denoted by $d(V,V_0)$, to be the sum of the norms of the linear maps $\fR$ and $\fR'$. 

\begin{prop}\label{W-change-domain}
	Suppose $\alpha_0$, $V_0$ and $L_0$ are given as above, and $\mathcal W_0$ is the Hilbert space defined using $(\alpha_0,V_0,L_0)$.
	There is a positive constant $c_1$ such that for any $(\alpha,V,L)\in \mathcal U$ with 
	\begin{equation}\label{norm-c-1}
	  |\!|\alpha-\alpha_0|\!|_{L^2_{1}}+d(V,V_0)+d(L,L_0)<c_1
	\end{equation}
	the following holds. There is an invertible bounded linear map $Q:\mathcal H\to \mathcal H$
	which maps $\mathcal W_0$ to $\mathcal W$, defined using $(\alpha,V,L)$. There is a constant $C$, independent of $(\alpha,V,L)$, such that
	\[
	  |\!|Q-{\rm Id}|\!|_{L^2}\leq C(|\!|\alpha-\alpha_0|\!|_{L^2_{1}}+d(V,V_0)+d(L,L_0))
	\]
	For any $k$, $Q$ induces an isomorphism on the space of $L^2_k$ triples in $L^2_k$. There are positive constants $c_k$ and $C_k$ such that for any $(\alpha,V,L)\in \mathcal U$ with 
	\begin{equation}\label{norm-c-1}
	  |\!|\alpha-\alpha_0|\!|_{L^2_{k}}+d(V,V_0)+d(L,L_0)<c_k
	\end{equation}
	then the operator norm of $Q-{\rm Id}$ as an operator, acting on the subspace of $\mathcal H$ given by $L^2_k$ triples, satisfies
	\begin{equation}\label{norm-c-k}
	  |\!|Q-{\rm Id}|\!|_{L^2_k}\leq C(|\!|\alpha-\alpha_0|\!|_{L^2_{k}}+d(V,V_0)+d(L,L_0)),
	\end{equation}	
	and for any $(\varphi,b,\nu)$ in the subspace of $L^2_k$ triples of $L^2_k$, we have
	\begin{equation}\label{ineq-Q-k}
		C_k^{-1}|\!|(\varphi,b,\nu)|\!|_{L^2_k} \leq |\!|Q(\varphi,b,\nu)|\!|_{L^2_k}\leq C_k|\!|(\varphi,b,\nu)|\!|_{L^2_k}.
	\end{equation}
\end{prop}

\begin{proof}
	Let $T_L:\R^{2n}\to \R^{2n}$ be the isomorphism that its restriction to $L_0$ is the orthogonal projection to $L$ and its restriction to $J_0L_0$ (the orthogonal complement of $L_0$) is orthogonal projection to $J_0L$. 
	Similarly, let $T_{\alpha,V}$
	be the isomorphism $\mathcal H^1(\Sigma;\alpha_0)\oplus \R^{2n} \to \mathcal H^1(\Sigma;\alpha)\oplus \R^{2n}$ that maps $V_0$ to $V$ by the composition of $\Phi_\alpha$ and the orthogonal projection to $V$ and maps 
	the orthogonal complement of $V_0$ in $\mathcal H^1(\Sigma;\alpha_0)\oplus \R^{2n}$ to $\bJ_0 V$ by the composition of $\Phi_\alpha$ and the orthogonal projection to $\bJ_0 V$. 
	We extend $T_{\alpha,V}$ into an isomorphism
	\begin{equation*}
		L^2(\Sigma;\Lambda^1\otimes F)\oplus \R^{2n}\to L^2(\Sigma;\Lambda^1\otimes F)\oplus \R^{2n}
	\end{equation*}
	which maps an element $z\in L^2(\Sigma;\Lambda^1\otimes F)\oplus \R^{2n}$ presented as 
	\begin{equation}\label{z-element}
	  \hspace{2cm}z=v+(d_{\alpha_0}\zeta+*_2d_{\alpha_0}\zeta',0) \hspace{1cm}v\in \mathcal H^1(\Sigma;\alpha_0)\oplus \R^{2n},
	\end{equation}
	into 
	\begin{equation}\label{def-Tav}
		T_{\alpha,L}(v)+(d_{\alpha}\zeta+*_2d_{\alpha}\zeta',0).
	\end{equation}	
	By a slight abuse of notation, we denote this map with the same notation $T_{\alpha,L}$. The key property of $T_{\alpha,V}$ is that it maps $\mathcal L_{\alpha_0,V_0}$ isomorphically onto $\mathcal L_{\alpha,V}$. 
	The operators $T_{\alpha,V}$ and $T_L$ satisfy the following operator norms with respect to the standard norms on $L^2(\Sigma;\Lambda^1\otimes F)\oplus \R^{2n}$
	and $\R^{2n}$:
	\begin{equation}\label{diff-T-I}
	  |\!|T_{\alpha,V}-{\rm Id}|\!|<C\(|\!|\alpha-\alpha_0|\!|_{L^2_{1}}+d(V,V_0)\),\hspace{1cm}|\!|T_L-{\rm Id}|\!|\leq Cd(L,L_0).
	\end{equation}
	In fact, $T_{\alpha,V}$ sends $L^2_k(\Sigma;\Lambda^1\otimes F)\oplus \R^{2n}$ into itself, and the operator norm of $T_{\alpha,V}$, as an operator acting on $L^2_k(\Sigma;\Lambda^1\otimes F)\oplus \R^{2n}$ with respect to its standard
	norm satisfies
	\begin{equation}\label{diff-T-k}
	  |\!|T_{\alpha,V}-{\rm Id}|\!|_{L^2_k}<C\(|\!|\alpha-\alpha_0|\!|_{L^2_{k}}+d(V,V_0)\).
	\end{equation}
	If $(\alpha,V,L)$ satisfies \eqref{norm-c-1} for a small enough $c_1$, then \eqref{diff-T-I} implies that for any $s\in [0,1]$, the operator $sT_{\alpha,V}+(1-s)I$ is invertible. Let
	\[
	  T_{\alpha,V}^G:L^2(\Sigma;\Lambda^1\otimes F)\oplus \R^{2n}\to L^2(\Sigma;\Lambda^1\otimes F),\hspace{1cm}T_{\alpha,V}^S:L^2(\Sigma;\Lambda^1\otimes F)\oplus \R^{2n}\to \R^{2n},
	\]
	denote the composition of $T_{\alpha,V}$ with projection maps to $L^2(\Sigma;\Lambda^1\otimes F)$ and $\R^{2n}$.

	We use the maps $T_{\alpha,V}$ and $T_{L}$ to define the desired $Q:\mathcal H\to \mathcal H$. Fix a cutoff function $\rho:[0,1]\to [0,1]$ that is equal to $1$ on $[0,\epsilon/3)$ and vanishes on $(\epsilon/2,1]$. 
	Let $(\varphi,b,\nu)\in \mathcal H$, and the restriction of $b$ to $(-\epsilon,0]\times \Sigma\subset Y$ is given as in \eqref{decomp-b-bdry}.
	Then $Q(\varphi,b,\nu)=(\varphi,c,\eta)$, where $c$ is equal to $b$ on the complement of $(-\epsilon,0]\times \Sigma$, the restriction of $c$ on $(-\epsilon,0]\times \Sigma$ is given as
	\[
	  \rho(-s) T_{\alpha,V}^G(q(s),\nu(-s))+ (1-\rho(-s))q(s)+\tau ds,
	\]
	$\eta=\nu$ on the complement of $[0,\epsilon)\cup (1-\epsilon,1]$, the restriction of $\eta$ on $[0,\epsilon)$ is given as
	\[
	  \rho(s) T_{\alpha,V}^S(q(s),\nu(-s))+(1-\rho(s))\nu(s), 
	\]
	and the restriction of $\eta$ to $(1-\epsilon,1]$ is given as
	\[
	  \rho(1-s) T_L(\nu(s))+(1-\rho(1-s))\nu(s). 
	\]
	The inequalities \eqref{diff-T-I} and \eqref{diff-T-k} can be used to verify \eqref{norm-c-1} and \eqref{norm-c-k}. The inequalities in \eqref{ineq-Q-k} is a consequence of \eqref{norm-c-k}.
\end{proof}

\subsection{The operator $\fD_{(B,S)}$}\label{unbdd-self-adj-3D}

In this subsection, we fix $B$, $S$, $\mathcal L$ and $L$ as in the previous subsection, and form the Hilbert spaces $\mathcal H$ and $\mathcal W$ and the operator $\fD_{(B,S)}$. Here we focus on the operator $ \fD_{(B,S)}$, and our goal is to show that it is self-adjoint and satisfies some regularity properties.

\begin{lemma} \label{DBp-syymetric}
	The operator $ \fD_{(B,S)}$ is symmetric. That is to say, for any $(\varphi,b,\nu), (\psi,c,\eta)\in \mathcal W$, we have
	\[
	  \langle (\varphi,b,\nu),\fD_{(B,S)}(\psi,c,\eta)\rangle_{L^2}=
	  \langle \fD_{(B,S)} (\varphi,b,\nu),(\psi,c,\eta)\rangle_{L^2}.
	\]
\end{lemma}
\begin{proof}
	Using Stokes theorem we have
	\begin{equation}\label{3D-stokes-2}
	  \int_{Y}\tr(d_B^*b\wedge *\psi)-\int_{Y}\tr(b\wedge *d_B \psi)=
	  -\int_{\Sigma}\tr(*b\wedge \psi),
	\end{equation}
	and
	\begin{equation}\label{3D-stokes-3}
		\int_{Y}\tr(d_B^*c\wedge *\varphi)-\int_{Y}\tr(c\wedge *d_B \varphi)=  
		-\int_{\Sigma}\tr(*c\wedge \varphi).
	\end{equation}
	Since $*b|_{\Sigma}$ and $*c|_{\Sigma}$ vanish, the above expressions are equal to 
	zero.
	Another application of Stokes theorem implies that
	\begin{equation}\label{3D-Stokes-1}
		\int_{Y}\tr(d_Bb\wedge c)-\int_{Y}\tr(b\wedge d_B c)=
		\int_\Sigma\tr(b\wedge c),
	\end{equation}
	and 
	\begin{equation}\label{3D-Stokes-2}
	  \int_0^1 \omega_0(\frac{d\nu}{ds},\eta)ds-\int_0^1 \omega_0(\frac{d\eta}{ds},\nu)ds=
	  \omega_0(\nu(1),\eta(1))-\omega_0(\nu(0),\eta(0)).
	\end{equation}
	The first term on the right hand side of the above expression vanishes because $\nu(1),\eta(1) \in L$ and the second term is equal to the negative of the right
	hand side of \eqref{3D-Stokes-1}, because $(b|_{\Sigma},\nu(0))$, $(c|_{\Sigma},\eta(0))\in \mathcal L$.
	These observations immediately yield the claim that $\fD_{(B,S)}$ is symmetric.
\end{proof}

\begin{lemma}\label{3D-Fred-k=1}
	There is a constant $C$ such that the following holds.
	Suppose $(\varphi,b,\nu)\in \mathcal H$ has the property that there exists a constant 
	$\kappa$ with
	\begin{equation}\label{weak-conv-assumption}
		\langle (\varphi,b,\nu),\fD_{(B,S)}(\psi,c,\eta)\rangle_{L^2}
		\leq \kappa\vert \!\vert(\psi,c,\eta)\vert \!\vert_{L^2}  
	\end{equation}
	for any $(\psi,c,\eta)\in \mathcal W$. Then $(\varphi,b,\nu)\in \mathcal W$ and
	\begin{equation}\label{Fredholm-3D-k=1}
		\vert \!\vert(\varphi,b,\nu)\vert \!\vert_{L^2_1}\leq C(\kappa+
		\vert \!\vert(\varphi,b,\nu)\vert \!\vert_{L^2}).
	\end{equation}
\end{lemma}
In other words, this lemma asserts that the domain of the adjoint of the symmetric unbounded operator $\fD_{(B,S)}$ is $\mathcal W$. Therefore, $\fD_{(B,S)}$ is a self-adjoint operator. Another immediate consequence of the above two lemmas is that for any $(\varphi,b,\nu)\in \mathcal W$, we have
\begin{equation}\label{Fredholm-3D-k=1}
	\vert \!\vert(\varphi,b,\nu)\vert \!\vert_{L^2_1}\leq C(\vert \!\vert\fD_{(B,S)}(\varphi,b,\nu)\vert \!\vert_{L^2}+\vert \!\vert(\varphi,b,\nu)\vert \!\vert_{L^2}).
\end{equation}

\begin{proof}
	Suppose $\rho_1:Y\to \R$ and $\rho_2:[0,1]\to \R$ are smooth functions 
	such that the restriction of $\rho_1$ to $\partial Y=\Sigma$ is the constant function with value $\rho_2(0)$. 	
	If \eqref{weak-conv-assumption} holds for $(\varphi,b,\nu)$, then it is also satisfied for 
	$(\rho_1\varphi,\rho_1b,\rho_2\nu)$. To see this, note that if $(\psi,c,\eta)\in \mathcal W$, then $(\rho_1\psi,\rho_1c,\rho_2\eta)\in \mathcal W$, and the difference
	\[
	  \left|\langle (\rho_1\varphi,\rho_1b,\rho_2\nu),\fD_{(B,S)}(\psi,c,\eta)\rangle_{L^2}-\langle (\varphi,b,\nu),\fD_{(B,S)}(\rho_1\psi,\rho_1c,\rho_2\eta)\rangle_{L^2}\right|
	\]
	is bounded by $C\vert \!\vert(\varphi,b,\nu)\vert \!\vert_{L^2}\vert \!\vert(\psi,c,\eta)\vert \!\vert_{L^2}$ for a suitable constant $C$, which depends only on $\rho_1$ and $\rho_2$.
	Thus a partition of unity argument allows us to divide the proposition into three cases:
	\begin{itemize}
		\item[(i)] $\nu=0$ and $(\varphi,b)$ is compactly supported in the interior of $Y$;
		\item[(ii)] $(\varphi,b)=0$ and $\nu$ is compactly supported in $(0,1]$
		\item[(iii)] $(\varphi,b)$ is compactly supported in a collar neighborhood 
		$(-\epsilon,0]\times \Sigma$ of the boundary of $Y$ and $\nu$ is compactly supported in 
		$[0,\epsilon)$. We use $s$ to denote the standard coordinate for the first factor of the collar 
		neighborhood $(-\epsilon,0]\times \Sigma$. The metric in this neighborhood has the form 
		$ds^2+g_\Sigma$.
	\end{itemize}

	The first two cases are standard and and we only need to 	address the third case. Let $b=q+\tau ds$ as in \eqref{decomp-b-bdry}.
	We assume that the connection $B$ on $(-\epsilon,0]\times \Sigma$ is in temporal gauge with respect to the coordinate $s$, and for each $s \in (-\epsilon,0]$, we write $\beta_s$ (or simply $\beta$) for the restriction of 
	the connection $B$ to $ \{s\}\times \Sigma\subset (-\epsilon,0]\times \Sigma$. 
	We prove the claim of \eqref{Fredholm-3D-k=1} in four steps. In the following, $C$ is the desired 
	constant in \eqref{Fredholm-3D-k=1}. Throughout the proof we might need to increase this constant from each line to the
	next one.

	{\bf Step 1:} {\it The term $\varphi$ is in $L^2_1$ and the constant $C$ can be chosen such that
	\begin{equation}\label{ineq-alpha0}
	  |\!|\varphi|\!|_{L^{2}_1}\leq C(\kappa+|\!|(\varphi,b)|\!|_{L^2}).
	\end{equation}}	
	Suppose $\xi$ is a smooth section of $E$ such that the normal derivative $\partial_s \xi$ restricted to the boundary $\Sigma$
	vanishes. This means that $\xi\in \Gamma_\nu(Y,E)$ in the notation of Appendix \ref{elliptic-reg-sec}.
	Then $(0,d_B\xi,0)$ belongs to $\mathcal W$ with respect to the connection $B$. 
	Applying \eqref{weak-conv-assumption} implies that the expression 
	\[
	  \langle (\varphi,b,\nu),\fD_{(B,S)}(0,d_B\xi,0)\rangle_{L^2}
	  =\langle \varphi,d_B^*d_B\xi\rangle_{L^2}+\langle b,*[F_B,\xi]\rangle_{L^2}
	\]
	is bounded by $\kappa\vert \!\vert d_B\xi \vert \!\vert_{L^2}$. In particular, we may pick $C$ such 
	that
	\begin{equation}
		\langle \varphi,d_B^*d_B\xi\rangle_{L^2}\leq 
		C(\kappa+\vert \!\vert b \vert \!\vert_{L^2})\vert \!\vert \xi \vert \!\vert_{L^2_1}
	\end{equation}
	By working in charts on $(-\epsilon,0]\times \Sigma$ and trivializing $F$ on each chart, we can bound $\langle \varphi,\Delta \xi\rangle_{L^2}$ by $C(\kappa+\vert \!\vert (\varphi,b) \vert \!\vert_{L^2})\vert \!\vert \xi \vert \!\vert_{L^2_1}$.
	Therefore, we may apply part (ii) of Lemma \ref{elliptic-reg-laplace} to obtain \eqref{ineq-alpha0}.

	{\bf Step 2:} {\it The component $\tau$ of  $b=q+\tau ds$ is in $L^2_1$ and the constant $C$ can be chosen such that
	\begin{equation}\label{ineq-phi-3D}
	  |\!|\tau|\!|_{L^{2}_1}\leq C(
	  \kappa+|\!|(\varphi,b)|\!|_{L^2}).
	\end{equation}}
	If $\xi$ is a smooth section of $E$, then $(\xi,0,0)$ defines an element of $\mathcal W$. 
	Thus \eqref{weak-conv-assumption} implies that
	\begin{equation}\label{one-dire-crit-3D}
		\langle b,d_B\xi\rangle_{L^2}\leq \kappa\vert \!\vert \xi \vert \!\vert_{L^2}.
	\end{equation}	
	Next, let $\gamma$ be a smooth section of $E$ that vanishes on $\Sigma$, that is to say 
	$\gamma \in \Gamma_\tau(Y,E)$. We also assume that the support of $\gamma$ is contained in 
	$(-\epsilon,0]\times \Sigma$. Therefore, $*d_B(\gamma ds)$ can be regarded as a 1-form on $Y$.
	Moreover, $(0,*d_B(\gamma ds),0)$ belongs to $\mathcal W$. Therefore, another application of 
	\eqref{weak-conv-assumption} gives
	\[
	  \langle (\varphi,b,\nu),\fD_{(B,S)}(0,*d_B (\gamma ds),0)\rangle_{L^2}
	  =\langle b,*d_B*d_B(\gamma ds)\rangle_{L^2}-
	  \langle \varphi,*([F_B,\gamma]\wedge ds) \rangle_{L^2}
	\]
	is bounded by $\kappa\vert \!\vert d_B\gamma \vert \!\vert_{L^2}$. Thus we can conclude that:
	\begin{equation}\label{another-dire-crit-3D}
		\langle b,d_B^*d_B(\gamma ds)\rangle_{L^2}\leq 
		C(\kappa+\vert \!\vert \varphi \vert \!\vert_{L^2})\vert \!\vert \gamma \vert \!\vert_{L^2_1}
	\end{equation}	
	after possibly enlarging the value of $C$. Inequalities \eqref{one-dire-crit-3D} and 
	\eqref{another-dire-crit-3D} are the necessary inputs to apply Lemma \ref{elliptic-reg-k=0}, where $\alpha$, $r$, $\sigma$ and $A_0$ in the statement of this lemma are $b$, $2$, $\partial_s$ and the smooth connection $B$. In particular, this 
	shows that $\tau$ is in $L^2_1$ and 
	\[
	  |\!|\tau|\!|_{L^{2}_1}=|\!|b(\partial_s)|\!|_{L^{2}_1}\leq 
	  C(\kappa+|\!|(\varphi,b)|\!|_{L^2}).
	\]
	This gives us the inequality in \eqref{ineq-phi-3D}.

	{\bf Step 3:} {\it The section $\nabla_\Sigma (q)$ of $T^*\Sigma\otimes T^*\Sigma \otimes E$ on 
	$(-\epsilon ,0] \times \Sigma$ given by the covariant derivatives of $q$ along  $\Sigma$ with respect to the connection $B$ is in $L^2$ . 
	Moreover, the constant $C$ can be chosen such that
	\begin{equation}\label{ineq-nab-sig-a-3D}
	  |\!|\nabla_\Sigma (q) |\!|_{L^2}\leq C(\kappa+|\!|(\varphi,b)|\!|_{L^2}).
	\end{equation}}
	Let $\xi$ be a smooth section of $E$ that vanishes on the boundary of $Y$ and is 
	supported in the collar neighborhood $(-\epsilon,0]\times \Sigma$. Since $(0,\xi ds,0)$ is an element 
	of $\mathcal W$, the expression
	\[
	  \langle (\varphi,b,\nu),\fD_{(B,S)}(0,\xi ds,0)\rangle_{L^2}
	  =\langle \varphi,d_B^*(\xi ds)\rangle_{L^2}+\langle b,*d_B(\xi ds) \rangle_{L^2}
	\]
	is bounded by $\kappa\vert \!\vert \xi \vert \!\vert_{L^2}$. Using Stokes' theorem and Step 1, the first term on the right hand side of the above identity is bounded by $C(\kappa+|\!|(\varphi,b)|\!|_{L^2})\vert \!\vert \xi \vert \!\vert_{L^2}$. 
	Note that the assumption on the 
	vanishing of $\xi$ on the boundary implies that there is no boundary term in the application of Stokes' theorem. In summary, we have
	\begin{equation}\label{qs-Sigma}
	  \int_{-\epsilon}^0 \langle q_s,*_2 d_{\beta_s} \xi_s \rangle_{L^2(\Sigma)} ds \leq 
	  C(\kappa+|\!|(\varphi,b)|\!|_{L^2})\vert \!\vert \xi \vert \!\vert_{L^2}.
	\end{equation}
	where $q_s$ and $\xi_s$ are restrictions of $q$ and $\xi$ to $\Sigma\times \{s\}$. In fact, the same inequality holds if we drop the assumption of the vanishing of $\xi$ on the boundary. Let $\rho:(-\infty,0]\times \R$ be a smooth function that 
	vanishes on $(-\epsilon/3,0]$ and is equal to $1$ on $(-\infty,-\epsilon/2]$. For any smooth section $\xi$ of $E$, we map apply \eqref{qs-Sigma} to $\xi_i:=\rho(is)\xi$, and by taking the limit $i \to \infty$, we obtain a similar inequality for $\xi$. 
	
	Suppose again $\xi$ is a smooth section of $E$ and follow Step 2 to show that 
	\[
	  \langle b,d_B\xi\rangle_{L^2}= \int_{-\epsilon}^0 \langle q_s,d_{\beta_s} \xi_s \rangle_{L^2(\Sigma)}+ \int_{-\epsilon}^0 \langle \tau_s,\partial_s \xi_s \rangle_{L^2(\Sigma)}
	\]
	is bounded by $\kappa\vert \!\vert \xi \vert \!\vert_{L^2}$. Integration by parts and Step 2 imply that the second term on the right hand side of the above identity is bounded by $C(\kappa+|\!|(\varphi,b)|\!|_{L^2})\vert \!\vert \xi \vert \!\vert_{L^2}$. 
	We again use the vanishing of $\xi$ on the boundary to show that there is no boundary term. Thus we obtain 
	\[
	  \int_{-\epsilon}^0 \langle q_s,d_{\beta_s} \xi_s \rangle_{L^2(\Sigma)}ds \leq 
	  C(\kappa+|\!|(\varphi,b)|\!|_{L^2})\vert \!\vert \xi \vert \!\vert_{L^2}.
	\]
	We can again drop the assumption on the vanishing of $\xi$ on the boundary as in the previous paragraph. Therefore, we can deduce from Lemma \ref{weak-Dirac-lemma} that $\nabla_\Sigma (q)$ is in $L^2$ and 
	the constant $C$ can be chosen such that \eqref{ineq-nab-sig-a-3D} holds.
	
	{\bf Step 4:} {\it 	The derivatives of $q$ and $\nu$ with respect to $s$ are in $L^2$, and the constant $C$ can be chosen such that
	\begin{equation}\label{ineq-s-theta-der}
	  |\!|\partial_s q|\!|_{L^2}+|\!|\frac{d\nu}{ds}|\!|_{L^2}\leq 
	  C(\kappa+|\!|(\varphi,b,\nu)|\!|_{L^2}).
	\end{equation}}
	Suppose $c$ is a 1-form with values in $E$ supported in $(-\epsilon,0]\times \Sigma$ which has 
	a vanishing $ds$ component and $c|_{\Sigma}=0$. We write $*_2c$ for the $1$-form on $(-\epsilon,0]\times \Sigma$ given by the Hodge star of $ds\wedge c$. Suppose also $\eta:[0,1]\to \R^{2n}$ is a smooth map supported in $[0,\epsilon)$
	such that $\eta(0)=0$. Then $(0,*_2 c,J_0\eta)\in \mathcal W$ 
	and \eqref{weak-conv-assumption} implies that
	\begin{align*}
	  \langle (\varphi,b,\nu),\fD_{(B,S)}(0,*_2 c,J_0\eta)\rangle_{L^2}
	  =&\langle \varphi,*_2d_{\beta}c \rangle_{L^2}-\langle q,\partial_sc\rangle_{L^2}-\langle \tau,d_\beta^*c\rangle_{L^2}\\
	  &+\int_0^1\langle \nu ,-\frac{d\eta}{ds}+S(J_0\eta)\rangle ds
	\end{align*}
	is bounded by $\kappa\vert \!\vert(c,\eta)\vert \!\vert_{L^2}$. Stokes' theorem, Step 1 and Step 2 imply that the first and the third term on the right hand side of the above 
	identity is bounded by $C(\kappa+|\!|(\varphi,b)|\!|_{L^2})\vert \!\vert c \vert \!\vert_{L^2}$. Therefore, we have 
	\[
	  \vert \langle q,\partial_sc\rangle_{L^2}+\int_0^1\langle \nu ,\frac{d\eta}{ds}\rangle ds\vert \leq 
	  C(\kappa+|\!|(\varphi,b,\nu)|\!|_{L^2})\vert \!\vert (c,\eta)\vert \!\vert_{L^2}.
	\]
	This shows that the derivative of $q$ and $\nu$ with respect to $s$ exist in the weak sense and the claimed inequality in \eqref{ineq-s-theta-der} holds.

	{\bf Step 5:} {\it 	$(\varphi,b,\nu)\in \mathcal W$.}\\
	Previous steps give us a control over the $L^2_1$ norm of $(\varphi,b,\nu)$. Thus we just need
	to check the boundary conditions. This is a straightforward consequence of the identities produced by
	the Stokes theorem in \eqref{3D-stokes-2}, \eqref{3D-stokes-3} and \eqref{3D-Stokes-1}. In fact, these 
	identities show that if $(\psi,c,\eta)\in \mathcal W$, then
	\[
	  \langle (\varphi,b,\nu),\fD_{(B,S)}(\psi,c,\eta)\rangle_{L^2}
	\]
	is equal to the sum of
	\begin{equation} \label{int-by-parts-1}
	  \langle \fD_{(B,S)}(\varphi,b,\nu),(\psi,c,\eta)\rangle_{L^2}
	\end{equation}
	and the boundary terms
	\begin{equation}\label{int-by-parts-2}
	   \int_{\Sigma}\tr(*c\wedge \varphi)
	   - \int_\Sigma\tr(*b\wedge \psi)+\int_\Sigma\tr(b\wedge c)
	   +\omega_0(\nu(0),\eta(0))-\omega_0(\nu(1),\eta(1)).
	\end{equation}
	The first and the last terms in \eqref{int-by-parts-2} vanish because $(\psi,c,\eta)\in \mathcal W$ and $\nu(1)=0$. Previous steps show that \eqref{int-by-parts-1} is bounded by 
	$C(\kappa+|\!|(\varphi,b,\nu)|\!|_{L^2})\vert \!\vert (\psi,c,\eta)\vert \!\vert_{L^2}$. Therefore, 
	\eqref{weak-conv-assumption} asserts that the same is true for \eqref{int-by-parts-2}. This implies that
	\begin{equation}
		*b\vert_\Sigma=0,\hspace{3cm}
	   -\int_\Sigma\tr(b\wedge \beta)=\omega_0(\nu(0),\eta), \hspace{.5cm}\forall\,\, (\beta,\eta)\in \mathcal L.
	\end{equation}	  
	These identities show that $(\varphi,b,\nu)$ satisfy the conditions in \eqref{match-bdry-cond}.
\end{proof}

\begin{remark}\label{constant-dep-k=0}
	One might ask how the constant $C$ in Lemma \ref{3D-Fred-k=1}
	depends on $(B,S)$. An examination of the proof shows that for an open neighborhood of $(B,S)$, defined using the $L^2_l$ norm for some value of $l$, we may find a constant $C$ 
	which works for all elements in this neighborhood. (In fact, we can work with $l=2$. But the precise value of $l$ shall not be important for us.)
\end{remark}

\begin{remark} \label{subsapce-W}
	It is worthwhile to observe that for the most part in the proof of Lemma \ref{3D-Fred-k=1} 
	we can work with $(\psi,c,\eta)$ inside a smaller subspace of $\mathcal W$
	(compare \cite[Lemma  3.5]{SW:I-bdry}.) 
	In Step 1, triples $(\psi,c,\eta)=(0,d_B\xi,0)$ with 
	$\xi\in \Gamma_\nu(Y,E)$ suffices for our purposes and through Steps 2-4 of the proof, we 
	need the inequality in \eqref{weak-conv-assumption} only for smooth $(\psi,c,\eta)$ such that 
	$\eta(0)=\eta(1)=0$, $*c|_\Sigma=0$ and $c\vert_\Sigma=0$. 
	It is only in the last step of the proof 
	that we use the full strength of  \eqref{weak-conv-assumption} to show that $(\varphi,b,\nu)\in \mathcal W$.
\end{remark}

The following lemma concerns the generalization of \eqref{Fredholm-3D-k=1} for higher Sobolev norms. This lemma is the counterpart of \cite[Proposition 3.1]{SW:I-bdry}.

\begin{lemma}\label{3D-Fred-k=higher}
	For any non-negative integer $k$, there is a constant $C_k$ such that if $(\varphi,b,\nu)\in \mathcal W$ and $\fD_{(B,S)} (\varphi,b,\nu)$
	has finite $L^2_k$ norm, then $(\varphi,b,\nu)$ is in $L^2_{k+1}$ and
	\[
	  \vert\!\vert (\varphi,b,\nu)\vert\!\vert_{L^2_{k+1}}\leq C_k (\vert\!\vert  \fD_{(B,S)} (\varphi,b,\nu)\vert\!\vert_{L^2_k}
	  +\vert\!\vert (\varphi,b,\nu)\vert\!\vert_{L^2}).
	\]
\end{lemma}

\begin{proof}
	It is obvious from the definition that 
	\[
	  \vert\!\vert\nu\vert\!\vert_{L^2_{k+1}}\leq C_k(\vert\!\vert  \fD_{(B,S)} (\varphi,b,\nu)\vert\!\vert_{L^2_k}+\vert\!\vert\nu\vert\!\vert_{L^2}).
	\]
	We prove the corresponding claim for $(\varphi,b)$ by induction on $k$. 
	We already verified the case that $k=0$. Suppose $X_1$, $\dots$, $X_k$ are smooth vector fields on $Y$. 
	We assume that the restriction of $X_i$ to the boundary of $Y$ 
	is either tangential or is $\partial_s$. 
	To obtain the desired result, it suffices to show that for any such combination of vector fields, the inequality in \eqref{weak-conv-assumption} holds for any $(\psi,c,\eta)$
	if we replace $(\varphi,b,\nu)$ with
	\begin{equation}\label{sL-def}
	  \sL(\varphi,b,\nu):= 
	  (\sL^k_{\bX}(\varphi),\sL^k_{\bX}(b),0).
	\end{equation}
	Here $\sL^k_{\bX}$ is the composition $\sL_{X_1}\dots\sL_{X_k}$ of Lie derivatives.
	This would be a straightforward application of integration by parts if there were no boundary terms. However, the boundary terms on $Y$ and the interval $[0,1]$
	require a more careful analysis.
	
	First we consider the case that all $X_i$ have tangential restriction to the boundary of $Y$. Let $(\psi, c ,\eta)\in \mathcal W$ 
	be chosen such that $ c \vert_{\Sigma}=0$. Since $*b|_\Sigma=0$ and 
	the vector fields $X_i$ are tangential, we have $*\sL^k_\bX b\vert_\Sigma=0$. For now, we also assume that $(\varphi,b,\nu)$ is a smooth triple. By replacing $\varphi$ and $b$
	 in \eqref{3D-stokes-2}, \eqref{3D-stokes-3} and \eqref{3D-Stokes-1} with $\sL^k_\bX\varphi$ and $\sL_\bX^kb$, we have
	\begin{align*}
	  \langle\sL( \varphi,b,\nu),\fD_{(B,S)}(\psi, c,\eta)\rangle_{L^2}
	 =&\langle\fD_{(B,S)}\sL( \varphi, b,\nu),(\psi, c,\eta)\rangle_{L^2}\\
	 \leq &\vert \langle \sL\fD_{(B,S)} (\varphi,b, \nu),(\psi, c,\eta)\rangle_{L^2}\vert+ C\vert\!\vert (\varphi,b,\nu)\vert\!\vert_{L^2_{k}}\cdot
	 \vert\!\vert (\psi, c,\eta)\vert\!\vert_{L^2} \\
	 \leq &C\(\vert\!\vert\fD_{(B,S)} (\varphi,b, \nu)\vert\!\vert_{L^2_{k}}+\vert\!\vert (\varphi,b,\nu)\vert\!\vert_{L^2_{k}} \)
	 \vert\!\vert (\psi, c,\eta)\vert\!\vert_{L^2} \\
	 \leq &CC_{k-1}\(\vert\!\vert\fD_{(B,S)} (\varphi,b, \nu)\vert\!\vert_{L^2_{k}}+\vert\!\vert (\varphi,b,\nu)\vert\!\vert_{L^2} \)
	 \vert\!\vert (\psi, c,\eta)\vert\!\vert_{L^2}.
	\end{align*}
	The first inequality above is a consequence of the fact that $\sL$ and $\fD_{(B,S)}$ 
	commute up to differential operators of degree at most $k$. We can drop the smoothness assumption on
	$(\varphi,b,\nu)$ by taking a sequence of smooth triples $\{(\varphi^j,b^j,\nu^j)\}$ which are $L^2_k$ convergent to $(\varphi,b,\nu)$. Repeating the above argument gives the inequality
	\begin{align}
	  \langle\sL( \varphi^j,b^j,\nu^j),\fD_{(B,S)}(\psi, c,\eta)\rangle_{L^2}
	 \leq &\vert \langle \sL\fD_{(B,S)} (\varphi^j,b^j, \nu^j),
	 (\psi, c,\eta^j)\rangle_{L^2}\vert\nonumber\\
	 &+ C\vert\!\vert (\varphi^j,b^j,\nu^j)
	 \vert\!\vert_{L^2_{k}}\cdot\vert\!\vert (\psi, c,\eta)\vert\!\vert_{L^2} \label{inequality-j}
	\end{align}
	Since all the vector fields involved in the definition of $\sL$ are tangential, we can use integration
	by parts to move the operator of degree $k$ to the other side of the pairing without adding any boundary term:
	\[
	  \langle \sL\fD_{(B,S)} (\varphi^j,b^j,\nu^j),(\psi, c,\eta^j)\rangle_{L^2}=
	  \langle \fD_{(B,S)} (\varphi^j,b^j, \nu^j),\sL^*(\psi, c,\eta^j)\rangle_{L^2}.
	\]
	This can be used to show that we can take the limit of the inequality in \eqref{inequality-j} as $j$ goes to
	infinity to obtain the desired inequality for $(\varphi,b,\nu)$.
	
	According to Remark \ref{subsapce-W}, in order to obtain \eqref{Fredholm-3D-k=1} with $(\varphi,b,\nu)$ being replaced by 
	$\sL(\varphi,b,\nu)$, we need to control the following $L^2$-pairing where 
	$\xi\in \Gamma_\nu(Y,E)$:
	\begin{equation}\label{special-pairing}
	  \langle\sL( \varphi,b,\nu),\fD_{(B,S)}(0,d_B\xi,0)\rangle_{L^2}
	 =\langle \sL^k_\bX\varphi,d_B^*d_B\xi\rangle_{L^2}+
	 \langle \sL^k_\bX b,*[F_B,\xi]\rangle_{L^2}
	\end{equation}
	To estimate \eqref{special-pairing}, we assume that $(\varphi,b,\nu)$ is smooth. Then a similar argument as in the previous case shows that the same estimate holds for the general case.
	First consider the first term on the right hand side, which is equal to 
	$\langle d_B \sL^k_\bX \varphi,d_B\xi\rangle$ as a consequence of the Stokes theorem and 
	$\xi\in \Gamma_\nu(Y,E)$:
	\begin{align}
	  \langle d_B \sL^k_\bX \varphi,d_B\xi\rangle
	 \leq &\vert \langle \sL^k_\bX d_B\varphi,d_B\xi\rangle\vert +
	 C\vert\!\vert (\varphi,b,\nu)\vert\!\vert_{L^2_{k}}\cdot
	 \vert\!\vert d_B\xi\vert\!\vert_{L^2}\nonumber\\
	 \leq &\vert \langle d_B\varphi+*d_Bb, (\sL^k_\bX)^* d_B\xi\rangle\vert
	 +\vert \langle *d_Bb, (\sL^k_\bX )^*d_B\xi\rangle\vert
	 +C\vert\!\vert (\varphi,b,\nu)\vert\!\vert_{L^2_{k}}\cdot
	 \vert\!\vert d_B\xi\vert\!\vert_{L^2}\nonumber\\
	 \leq &\vert \langle *d_Bb, (\sL^k_\bX )^*d_B\xi\rangle\vert
	 +CC_{k-1}(\vert\!\vert\fD_{(B,S)} (\varphi,b, \nu)\vert\!\vert_{L^2_{k}}+
	 \vert\!\vert (\varphi,b,\nu)\vert\!\vert_{L^2})\cdot
	 \vert\!\vert d_B\xi\vert\!\vert_{L^2}\nonumber\\
	 \leq &\vert \langle *d_Bb, d_B(\sL^k_\bX )^*\xi\rangle\vert
	 +CC_{k-1}(\vert\!\vert\fD_{(B,S)} (\varphi,b, \nu)\vert\!\vert_{L^2_{k}}+
	 \vert\!\vert (\varphi,b,\nu)\vert\!\vert_{L^2})\cdot
	 \vert\!\vert d_B\xi\vert\!\vert_{L^2}\nonumber.
	 \end{align}
	 The first inequality is a consequence of the fact that $d_B \sL^k_\bX-\sL^k_\bX d_B $ is a differential operator of degree at most $k$. Similarly, to obtain the last inequality we observe that 
	 $(\sL^k_\bX )^*d_B-d_B(\sL^k_\bX )^*$ is a differential operator of degree at most $k$ such that each term has at most one derivative in the 
	 normal direction. Therefore, we can use integration by parts to obtain
	 \[
	   \left\vert\langle *d_Bb, \((\sL^k_\bX )^*d_B-d_B(\sL^k_\bX )^*\)\xi\rangle\right\vert
	   \leq C\vert\!\vert (\varphi,b,\nu)\vert\!\vert_{L^2_{k}}\cdot
	 \vert\!\vert d_B\xi\vert\!\vert_{L^2},
	 \]
	
	To bound the term $\langle *d_Bb, d_B(\sL^k_\bX )^*\xi\rangle$, note that we have
	\begin{align}
	  \left |\langle *d_Bb, d_B(\sL^k_\bX )^*\xi\rangle \right |
	  &= 
	  \left |\langle b, *[F_B,(\sL^k_\bX )^*\xi]\rangle-\int_{\Sigma}\tr( b\wedge d_B(\sL^k_\bX )^*\xi)\right|\label{Stokes-apply-sim}\\
	 	  &\leq C \vert\!\vert (\varphi,b,\nu)\vert\!\vert_{L^2_{k-1}}\vert\!\vert d_B\xi\vert\!\vert_{L^2}\label{integ-parts-sim}.
	 \end{align}
	 The identity in \eqref{Stokes-apply-sim} is a consequence of \eqref{3D-Stokes-1}. To obtain \eqref{integ-parts-sim}, we use integration by parts and the fact that the integral over $\Sigma$ in \eqref{Stokes-apply-sim} vanishes because 
	 $(b\vert_{\Sigma},\nu(0))$ and $(d_B(\sL^k_\bX )^*\xi,0)$ belong to $\mathcal L$. In summary, we have
	 \[
	   \langle \sL^k_\bX\varphi,d_B^*d_B\xi\rangle_{L^2}\leq 
	   CC_{k-1}(\vert\!\vert\fD_{(B,S)} (\varphi,b, \nu)\vert\!\vert_{L^2_{k}}+
	 \vert\!\vert (\varphi,b,\nu)\vert\!\vert_{L^2})\cdot
	 \vert\!\vert d_B\xi\vert\!\vert_{L^2}.
	 \]
	 It is clear that the second term on the left hand side of \eqref{special-pairing} can be also bounded by a similar term as in the right hand side of the above inequality. 
	 Therefore, we can use Lemma \ref{3D-Fred-k=1} and Remark \ref{subsapce-W} to conclude that
	 $\sL( \varphi,b,\nu)$ is in $L^2_{1}$ norm and we have
	 \begin{equation} \label{result-ineq}
	   \vert\!\vert\sL( \varphi,b,\nu)\vert\!\vert_{L^2_{1}}\leq 
	   CC_{k-1}(\vert\!\vert\fD_{(B,S)} (\varphi,b, \nu)\vert\!\vert_{L^2_{k}}+
	 \vert\!\vert (\varphi,b,\nu)\vert\!\vert_{L^2})\cdot
	 \vert\!\vert d_B\xi\vert\!\vert_{L^2}
	 \end{equation}
	 for an appropriate choice of $C$. This completes the proof in the case that all vector fields $X_i$
	 are tangential. 
	 
	Let $b$ have the form $q+\tau ds$ in a collar neighborhood $(-\epsilon,0]\times \Sigma$ of 
	the boundary of $Y$ and $\beta_s$ denotes the restrictions of $B$ to $\Sigma\times\{s\}$. Then we have
	 \begin{align*}
		*d_Bb+d_B\varphi&=\left(*_2\partial_s q-*_2d_{\beta_s}\tau+d_{\beta_s}\varphi\right)+\left(\partial_s\varphi+*_2d_{\beta_s} q\right)ds,\\
		d_B^*b&=d_{\beta_s}^*q-\partial_s\tau.
	 \end{align*}
	 Thus, these identities can be used to replace each normal derivative with components
	 of $\fD_{(B,S)} (\varphi,b, \nu)$ and tangential derivatives. Our analysis in the 
	 tangential case allows us to conclude that \eqref{result-ineq} holds in the case that some of 
	 the vector fields $X_i$ are equal to $ds$ in a neighborhood of the boundary. This completes the proof of 
	 the lemma.
\end{proof}

\begin{remark}\label{constant-dep-k=higher}
	An analogue of Remark \ref{constant-dep-k=0} applies to Lemma \ref{3D-Fred-k=higher}. 
	We may find a neighborhood of $(B,S)$, defined using an appropriate $L^2_l$ norm, such that Lemma \ref{3D-Fred-k=higher} holds for all 
	elements of this neighborhood using a universal constant $C$.	
\end{remark}

\begin{remark}\label{invertible-perturbation}
	Lemma \ref{3D-Fred-k=1} implies that $\fD_{(B,S)}:\mathcal H\to \mathcal W$ is a self-adjoint Fredholm operator because the inclusion of $\mathcal W$ in $\mathcal H$ is compact.  
	In particular, for $\lambda\in \R$, the operator $\fD_{(B,S)}-\lambda\cdot  {\rm Id}$ is invertible if and only if $\fD_{(B,S)}-\lambda\cdot  {\rm Id}$ is injective. 
	Moreover, spectral theory of self-adjoint compact operators implies that eigenvectors of $\fD_{(B,S)}$ provide a basis for $\mathcal H$, and the intersection of any finite interval with the eigenvalues of $\fD_{(B,S)}$ is finite. 
	In particular, if $\delta$ is small enough, then the operator $\fD_{(B,S)}-\delta \cdot {\rm Id}$ is invertible.
\end{remark}

\subsection{Fredholm theory on mixed cylinders}\label{Fr-mixed-cly}

Our next goal is to use the results of the previous two subsections to prove Theorem \ref{Fredholm-flex}. Another key input is given by the results of \cite{SW:I-bdry} about spectral flows of self-adjoint operators with varying domains. In fact, our proof here is inspired by the proof of Fredholm theory results in \cite{SW:I-bdry}. We assume that $I$, $J$, $A$, $S$, $\{\mathcal L_\theta,L_\theta\}_{\theta\in I}$ are given as in Theorem \ref{Fredholm-flex}. As before we denote the restriction of $A$ and $S$ to $Y\times \{\theta\}$ and $[0,1]\times \{\theta\}$ by $B_\theta$ and $S_\theta$. Let also $\alpha_\theta$ denote the restriction of $B_\theta$ to $\Sigma\times\{\theta\}$. Associated to $\alpha_\theta$, $\mathcal L_\theta$ and $\theta$, we have the Hilbert subspace $\mathcal W_\theta$ of $\mathcal H$, defined in Subsection \ref{domain-lemmas}. Then any element $(\zeta,\nu)$ in $E^k_{\fL}(I)$, the domain of $\mathcal D_{(A,u)}$, determines
\[
  \hspace{2cm} (\varphi_\theta,b_\theta,\nu_\theta)\in \mathcal W_\theta,\hspace{1cm}\theta\in I
\]
by restriction to $Y\times \{\theta\}$ and $[0,1]\times \{\theta\}$. (In the case that $k=1$, this holds for almost every value of $\theta$.) Moreover, the operator $\mathcal D_{(A,S)}$ has the form $\frac{d}{d\theta}-\fD_{(B_\theta,S_\theta)}$ as it is pointed out in \eqref{D-A-S}.

Proposition \ref{W-change-domain} implies that for $\theta_0\in I$, there is an open neighborhood of $\theta_0$ such that for any point $\theta$ in this neighborhood, there is an isomorphism $Q_\theta :\mathcal H\to \mathcal H$ mapping $\mathcal W_{\theta_0}$ to $\mathcal W_{\theta}$. To prove Theorem \ref{Fredholm-flex}, it suffices to consider the case that $I$ equals this neighborhood of $\theta_0$, and then use compactness of the closure of $J$ in $I$ to extend the result to the general case. We also assume that $\theta_0=0$, and denote $\mathcal W_{\theta_0}$ by $\mathcal W_0$. The following lemma is a consequence of Proposition \ref{W-change-domain} and the definition of the operators $Q_\theta$ given there.
\begin{lemma}\label{Q-family-C1}
	The map $\bQ:I\to B(\mathcal H)$ given by $\{Q_\theta\}_{\theta\in I}$ is smooth. 
	Furthermore, for any $k$ and any $(\zeta,\nu)\in L^2(I,\mathcal H)$, we have
	\[
	  (\zeta,\nu)\in L^2_{k}(Y\times I,\Lambda^1\otimes E)\oplus L^2_{k}([0,1]\times I,\R^{2n}) \iff \bQ(\zeta,\nu) \in L^2_{k}(Y\times I,\Lambda^1\otimes E)\oplus L^2_{k}([0,1]\times I,\R^{2n}),
	\]
	where $\bQ(\zeta,\nu) $ is defined as follows. The restriction of $\bQ(\zeta,\nu) $ to $Y\times \{\theta\}$ and $I\times \{\theta\}$ is given by the triple $Q_\theta(\varphi_\theta,b_\theta,\nu_\theta)$
	where $(\varphi_\theta,b_\theta,\nu_\theta)$ is given by the restriction of $(\zeta,\nu) $ to $Y\times \{\theta\}$ and $I\times \{\theta\}$.
	There is also a constant $C_k$ such that for any $(\zeta,\nu)$ as above, we have
	\[
	  C_k^{-1}|\!|(\zeta,\nu)|\!|_{L^2_k}\leq |\!|\bQ(\zeta,\nu)|\!|_{L^2_k}\leq C_k|\!|(\zeta,\nu)|\!|_{L^2_k}.
	\]
\end{lemma}

The following lemma follows easily from Proposition \ref{W-change-domain} and Lemma \ref{Q-family-C1}.

\begin{lemma}\label{D-change-basis-control}
	There is a constant $C$ such that for any $\theta \in I$, the operator \[\rD_\theta:=Q_\theta^{-1} \fD_{(B_\theta,S_\theta)}Q_\theta:\mathcal W_0\to \mathcal H\]
	satisfies
	\begin{equation}
		|\!|\rD_\theta(\varphi,b,\nu)|\!|_{L^2}+|\!|\frac{d\rD_\theta}{d\theta} (\varphi,b,\nu)|\!|_{L^2}\leq C|\!|(\varphi,b,\nu)|\!|_{L^2_1}
	\end{equation}
	for $(\varphi,b,\nu)\in \mathcal W_0$.
\end{lemma}

In summary, we verify the following properties for the Hilbert spaces $\mathcal W_\theta$, and the operators $Q_\theta$.
\begin{itemize}
	\item[(W1)] (Proposition \ref{W-change-domain}) The inclusion map from the Hilbert space $\mathcal W_\theta$ to the Hilbert space $\mathcal H$ is compact and has a dense image.
	The bounded maps  
	$Q_\theta:\mathcal H \to \mathcal H$ defines a family of isomorphisms such that $Q_\theta(\mathcal W_0)=\mathcal W_\theta$.
	\item[(W2)] (Proposition \ref{W-change-domain} and Lemma \ref{Q-family-C1}) $\bQ:J\to B(\mathcal H)$ is $C^1$ and for any $k\geq 0$, there is a constant $C_k$ such that for any $(\varphi,b,\nu)\in \mathcal W_0$ we have
	\begin{align*}
		C_k^{-1}|\!|(\varphi,b,\nu)|\!|_{L^2_k}\leq |\!|Q_\theta(\varphi,b,\nu)|\!|_{L^2_k}\leq C_k|\!|(\varphi&,b,\nu)|\!|_{L^2_k},\\
		|\!|\frac{dQ_\theta}{d\theta}(\varphi,b,\nu)|\!|_{L^2}\leq C_0|\!|(\varphi,b,\nu)|\!|_{L^2}.
	\end{align*}
\end{itemize}
The following properties are also established for the operators  $\fD_{(B_\theta,S_\theta)}$.
\begin{itemize}
	\item[(A1)] (Lemma \ref{3D-Fred-k=1} and Remark \ref{constant-dep-k=0}) The operators $\fD_{(B_\theta,S_\theta)}:\mathcal H\to \mathcal H$ is an (unbounded) self-adjoint operator with domain $\mathcal W_\theta$, and they satisfy
			\[\vert \!\vert(\varphi,b,\nu)\vert \!\vert_{L^2_1}\leq C'_1(\vert \!\vert\fD_{(B_\theta,S_\theta)}(\varphi,b,\nu)\vert \!\vert_{L^2}+\vert \!\vert(\varphi,b,\nu)\vert \!\vert_{L^2})\]
			where the constant $C'_1$ is independent of $\theta$.
	\item[(A2)] (Lemma \ref{D-change-basis-control}) For any $(\varphi,b,\nu)\in \mathcal W_0$, we have
			\[
			  |\!|\rD_\theta(\varphi,b,\nu)|\!|_{L^2}+|\!|\frac{d\rD_\theta}{d\theta} (\varphi,b,\nu)|\!|_{L^2}\leq C'_2|\!|(\varphi,b,\nu)|\!|_{L^2_1},
			\]
			where the constant $C'_2$ is independent of $\theta$.			
\end{itemize}

Now, we turn to the proof of Theorem \ref{Fredholm-flex}. We first address the second part of the theorem. Fix
\[
  (\zeta,\nu)\in L^2(Y\times I,\Lambda^1\otimes E)\oplus L^2([0,1]\times I,\R^{2n}).
\]
Then $(\zeta,\nu)$ can be regarded as an $L^2$ map from $I$ to $\mathcal H$, and we denote the value of this map at $\theta \in I$ by $(\varphi_\theta,b_\theta,\nu_\theta)$.
Suppose for any compactly supported smooth $(\xi,\eta)\in E^1_{\fL}(I)$ the following inequality holds for a constant $\kappa$ independent of $(\xi,\eta)$:
\begin{equation}\label{weak-reg-D-A-S}
   \langle (\zeta,\nu),\mathcal D_{(A,S)}(\xi,\eta)\rangle_{L^2}\leq \kappa |\!|(\xi,\eta)|\!|_{L^2(I)}.
\end{equation}

Then (W1), (W2), (A1) and (A2) essentially imply that we may apply Theorem A.3 of \cite{SW:I-bdry} to show that $(\zeta,\nu)$ is in $E^1_{\fL}(J)$. One wrinkle is that the statement of Theorem A.3 of \cite{SW:I-bdry}, a priori, applies to the case that $I=J=\R$, and the operators $Q_\theta$ and $\fD_{(B_\theta,S_\theta)}$ satisfy the following additional assumptions.
\begin{itemize}
	\item[(W3)] There are Hilbert space isomorphisms $Q^{\pm}:\mathcal H\to \mathcal H$ such that $Q_\theta$ is convergent to $Q^\pm$ in $B(\mathcal H)$ as $\theta \to \pm \infty$.
	\item[(A3)] There are isomorphisms $\rD^\pm:\mathcal W_0\to \mathcal H$ such that $\rD_\theta$ is convergent to $\rD^\pm$ in $B(\mathcal W_0,\mathcal H)$ as $\theta \to \pm \infty$.
\end{itemize}

We may modify our setup slightly such that the conditions (W3) and (A3) are satisfied. First we replace the interval $I$ with $J$ and the interval $J$ with a smaller interval around $0$. Pick a smooth map $f: \R \to I$ that is identity in a neighborhood $K$ of the closure of $J$ in $I$ and is a constant map on the complement of $I$ in the domain. Similarly, pick $g:\R\to \R$ such that $g(\theta)=0$ if $\theta\in J$ and $g(\theta)=1$ if $\theta\in K$. For any $\theta\in \R$, define
\[
  \mathcal W_\theta':=\mathcal W_{f(\theta)}, \hspace{1cm}Q_\theta':=Q_{f(\theta)}, \hspace{1cm}\fD_{(B_\theta,S_\theta)}'=\fD_{(B_{f(\theta)},S_{f(\theta)})}-\delta g(\theta)\cdot {\rm Id},
\]
for a small positive real number $\delta$. As in Lemma \ref{Q-family-C1}, suppose also $\bQ':\R\to B(\mathcal H)$ is given by the operators $Q_\theta'$. 
Clearly the analogues of (W1), (W2), (A1) and (A2) are satisfied for these operators. Moreover, $Q_\theta'$ and $\fD_{(B_\theta,S_\theta)}'$ are constant with respect to $\theta$ once $|\theta|$ is large enough. In particular, (W3) clearly holds and Remark \ref{invertible-perturbation} implies that (A3) holds if $\delta$ is small enough. 
Suppose also 
\[
  (\zeta',\nu')\in L^2(Y\times \R,\Lambda^1\otimes E)\oplus L^2([0,1]\times \R,\R^{2n}).
\]
is given such that its restriction to $Y\times \{\theta\}$ and $I\times \{\theta\}$, denoted by $(\varphi_\theta',b_\theta',\nu_\theta')$,  is given by
\[
  (\varphi_\theta',b_\theta',\nu_\theta'):=(1-g(\theta))\cdot (\varphi_{f(\theta)},b_{f(\theta)},\nu_{f(\theta)}).
\]
As a consequence of \eqref{weak-reg-D-A-S}, we have
\[
  \int_{-\infty}^\infty \langle (\varphi_\theta',b_\theta',\nu_\theta'),(\frac{d}{d\theta}-\fD_{(B_\theta,S_\theta)}')(\psi_\theta,c_\theta,\eta_\theta)\rangle_{L^2} d\theta \leq 
 \kappa \(\int_{-\infty}^\infty|\!|(\psi_\theta,c_\theta,\eta_\theta)|\!|_{L^2}^2\)^{\frac{1}{2}}
\]
where $\{(\psi_\theta,c_\theta,\eta_\theta)\in \mathcal W_\theta' \}_{\theta\in \R}$ is a $1$-parameter family of triples such that the map $\theta \to Q_\theta'^{-1}(\psi_\theta,c_\theta,\eta_\theta)$ is an element of $L^2_1(\R,\mathcal H)\cap L^2(\R,\mathcal W_0)$. Then Theorem A.3 of \cite{SW:I-bdry} implies that $\bQ'(\zeta',\nu')$ belongs to $L^2_1(J,\mathcal H)\cap  L^2(I,\mathcal W_0)$. In particular, $(\zeta,\nu)\in E^1_{\fL}(J)$. Furthermore, (proof of) Lemma A.2 of \cite{SW:I-bdry} implies that
\begin{equation}\label{fred-prop-R}
	|\!|(\zeta',\nu')|\!|_{L^2_{1}}\leq \mathcal C\left(|\!|\mathcal D_{(A,S)}'(\zeta',\nu')|\!|_{L^2}+|\!|(\zeta',\nu')|\!|_{L^2}\right),
\end{equation}
where $\mathcal D_{(A,S)}'$ is the operator $\frac{d}{d\theta}-\fD_{(B_\theta,S_\theta)}'$ and the constant $\mathcal C$ depends continuously on $C_0$, $C_1$ in (W2), $C_1'$ in (A1) and $C_2'$ in (A2). In fact, an explicit formula for $\mathcal C$ can be found in the proof of Lemma A.2 of \cite{SW:I-bdry}. As an immediate consequence of \eqref{fred-prop-R}, we have
\begin{equation}\label{fred-prop-interval}
  |\!|(\zeta,\nu)|\!|_{E^{1}_{\fL}(J)}\leq \mathcal C' \left(|\!|\mathcal D_{(A,S)}(\zeta,\nu)|\!|_{L^2(I)}+|\!|(\zeta,\nu)|\!|_{L^2(I)}\right),
\end{equation}
where $\mathcal C'$ is determined by $\mathcal C$ and the intervals $K$ and $J$ through the choice of $g$.

\begin{remark}\label{uniform-fred-prop}
	The properties of the constants $\mathcal C$ and $\mathcal C'$ in the previous paragraph allow us to obtain an analogue of Remark \ref{constant-dep-k=0} for the operator $\mathcal D_{(A,S)}$,
	 as an extension of 
	Theorem \ref{Fredholm-flex}. To be more detailed, there are neighborhoods of $A$, $S$, $\{\mathcal L_\theta,L_\theta\}_{\theta\in I}$, defined with respect to some 
	Sobolev $L^2_l$ norm such that for any $A'$, $S'$, $\{\mathcal L_\theta',L_\theta'\}_{\theta\in I}$, the analogue of inequality \eqref{fred-prop-interval} holds with the same constant $\mathcal C'$.
\end{remark}

We prove the first part of Theorem \ref{Fredholm-flex} by induction on $k$. We already addressed the case that $k=1$. Now let $(\zeta,\nu)\in E^1_{\fL}(I)$ and $(\xi,\eta):=\mathcal D_{(A,S)}(\zeta,\nu)$ is in $L^2_{k-1}$ for $k\geq 2$. In particular, the induction hypothesis implies that $(\zeta,\nu)\in E^{k-1}_{\fL}(I)$ after shrinking the interval $I$, and we wish to show that $(\zeta,\nu)\in E^{k}_{\fL}(J)$. First we consider 
\[
  (\breve \zeta,\breve \nu):=\bQ\frac{d}{d\theta}\(\bQ^{-1}(\zeta,\nu)\).
\]
Then $(\breve \zeta,\breve \nu)\in E^{k-2}_{\fL}(I)$, and if $k\geq 3$, we have
\begin{equation}\label{DAS-breve}
  \mathcal D_{(A,S)}(\breve \zeta,\breve \nu)=\bQ\frac{d}{d\theta}\(\bQ^{-1}\mathcal D_{(A,S)}(\zeta,\nu)\)+\bP(\zeta,\nu),
\end{equation}
where $\bP$ is the commutator of $\mathcal D_{(A,S)}$ and $\bQ\circ \frac{d}{d\theta}\circ \bQ^{-1}$. In particular, the properties of $\bQ$ and the fact that the commutator of $\mathcal D_{(A,S)}$ and $ \frac{d}{d\theta}$ is a differential operator of degree $0$ imply that $\bP$ is a bounded linear map 
\[E^k_{\fL}(I)\to L^2_{k-1}(Y\times \R,\Lambda^1\otimes E)\oplus L^2_{k-1}([0,1]\times \R,\R^{2n})\] 
for any $k\geq 1$. Since $\mathcal D_{(A,S)}(\zeta,\nu)$ is in $L^2_{k-1}$, we conclude that $\mathcal D_{(A,S)}(\breve \zeta,\breve \nu)$ is in $L^2_{k-2}$. Thus, the induction hypothesis implies that $(\breve \zeta,\breve \nu)\in E^{k-1}_{\fL}(J)$. In the case that $k=2$, the right hand side of \eqref{DAS-breve} is still well-defined and is in $L^2$. We may use this to show that
\begin{equation}\label{ineq-weak-ind-step}
 \langle (\breve \zeta,\breve \nu),\mathcal D_{(A,S)}^*(\xi,\eta)\rangle_{L^2}\leq \kappa |\!|(\xi,\eta)|\!|_{L^2(I)},
\end{equation}
for any compactly supported $(\xi,\eta)\in E^{1}_{\fL}(I)$ where $\kappa$ is the $L^2$ norm of the right hand side of \eqref{DAS-breve}. 

To see this, take a sequence $\{(\zeta_i,\nu_i)\}$ of elements of $E^{2}_{\fL}(I)$ converging to $(\zeta,\nu)$ in $L^2_1$. Let $  (\breve \zeta_i,\breve \nu_i):=\bQ\frac{d}{d\theta}(\bQ^{-1}(\zeta_i,\nu_i))$, which is $L^2$ convergent to $(\breve \zeta,\breve \nu)$. Then $ \mathcal D_{(A,S)}(\breve \zeta_i,\breve \nu_i)$ is given by the analogue of \eqref{DAS-breve}, and hence we have
\begin{align*}
	\langle (\breve \zeta_i,\breve \nu_i),\mathcal D_{(A,S)}^*(\xi,\eta)\rangle
	&=\langle \mathcal D_{(A,S)}(\breve \zeta_i,\breve \nu_i),(\xi,\eta)\rangle\\
	&=\langle \bQ\frac{d}{d\theta}\(\bQ^{-1}\mathcal D_{(A,S)}(\zeta_i,\nu_i)\)+\bP(\zeta_i,\nu_i),(\xi,\eta)\rangle\\
	&=\langle \mathcal D_{(A,S)}(\zeta_i,\nu_i),(\bQ^*)^{-1}\frac{d}{d\theta}\(\bQ^*(\xi,\eta)\)\rangle_{L^2}+\langle \bP(\zeta_i,\nu_i),(\xi,\eta)\rangle.
\end{align*}
Here $\bQ^*$ is the $L^2$-adjoint of the operator $\bQ$, and we use integration by parts to obtain the last identity. By taking the limit as $i \to \infty$, we have
\begin{align*}
	\langle (\breve \zeta,\breve \nu),\mathcal D_{(A,S)}^*(\xi,\eta)\rangle
	&=\langle \mathcal D_{(A,S)}(\zeta,\nu),(\bQ^*)^{-1}\frac{d}{d\theta}\(\bQ^*(\xi,\eta)\)\rangle_{L^2}+\langle \bP(\zeta,\nu),(\xi,\eta)\rangle\\
	&=\langle \bQ\frac{d}{d\theta}\(\bQ^{-1}\mathcal D_{(A,S)}(\zeta,\nu)\)+\bP(\zeta,\nu),(\xi,\eta)\rangle_{L^2},
\end{align*}
where in the last identity we use integration by parts and the assumption that $ \mathcal D_{(A,S)}(\zeta,\nu)$ it is $L^2_1$. 
The inequality in \eqref{ineq-weak-ind-step} and the second part of Theorem \ref{Fredholm-flex} imply that $(\breve \zeta,\breve \nu)\in E^{1}_{\fL}(J)$. (Strictly speaking, we need the second part of Theorem \ref{Fredholm-flex} for the formal adjoint $\mathcal D_{(A,S)}^*$. As we explained there, Theorem \ref{Fredholm-flex} would be sufficient for this because $\mathcal D_{(A,S)}^*$ has the form required for the application of Theorem \ref{Fredholm-flex}.)

Our arguments in any of the above cases give rise to the following inequality 
\begin{align*}
	|\!|(\breve \zeta,\breve \nu)|\!|_{E^{k-1}_{\fL}(J)}&\leq C\(|\!|\bQ^{-1}\frac{d}{d\theta}\(\bQ\mathcal D_{(A,S)}(\zeta,\nu)\)|\!|_{L^2_{k-2}(I)}+|\!|\bP(\zeta,\nu)|\!|_{L^2_{k-2}(I)}+|\!|(\breve \zeta,\breve \nu)|\!|_{L^2(I)}\)\\
	&\leq C\(|\!|\mathcal D_{(A,S)}(\zeta,\nu)|\!|_{L^2_{k-1}(I)}+|\!|(\zeta,\nu)|\!|_{L^2_{k-1}(I)}+|\!|(\breve \zeta,\breve \nu)|\!|_{L^2(I)}\)\\
	&\leq C\(|\!|\mathcal D_{(A,S)}(\zeta,\nu)|\!|_{L^2_{k-1}(I)}+|\!|(\zeta,\nu)|\!|_{L^2_{k-1}(I)}\).
\end{align*} 
Thus, to complete the proof we need to show that all derivatives of $(\zeta,\nu)$ up to order $k$, that do not involve derivation with respect to $\theta$, are in $L^2$. That is to say, it suffices to show that $(\zeta,\nu)\in L^2(J,L^2_{k})$. By assumption and the above argument, $\mathcal D_{(A,S)}(\zeta,\nu)$ and $\frac{d}{d\theta}(\zeta,\nu)$ are both in $L^2_{k-1}$. Since $\mathcal D_{(A,S)}-\frac{d}{d\theta}$ maps $(\zeta,\nu)$ to a pair in $L^2_{k-1}$, we conclude that $\fD_{B_\theta,S_\theta}(\varphi_\theta,b_\theta,\nu_\theta)$ is in $L^2_{k-1}$ for almost every $\theta\in J$. Lemma \ref{3D-Fred-k=higher} implies that for these values of $\theta$, $(\varphi_\theta,b_\theta,\nu_\theta)\in L^2_k$ and we have
\[
   \vert\!\vert (\varphi_\theta,b_\theta,\nu_\theta)\vert\!\vert_{L^2_{k}}\leq C_{k-1} (\vert\!\vert  \fD_{(B,S)} (\varphi_\theta,b_\theta,\nu_\theta)\vert\!\vert_{L^2_{k-1}}
   +\vert\!\vert (\varphi_\theta,b_\theta,\nu_\theta)\vert\!\vert_{L^2})
\]
where the constant $C_{k-1}$ can be chosen to be independent of $\theta$ by Remark \ref{constant-dep-k=higher}. Therefore, we can write
\begin{align*}
	|\!|(\zeta, \nu)|\!|_{L^2(J,L^2_{k})}^2&=\int_{J}  \vert\!\vert (\varphi_\theta,b_\theta,\nu_\theta)\vert\!\vert_{L^2_{k}}^2d\theta\\
	&\leq C_{k-1}\int_{J} \vert\!\vert  \fD_{(B,S)} (\varphi_\theta,b_\theta,\nu_\theta)\vert\!\vert_{L^2_{k-1}}^2
   +\vert\!\vert (\varphi_\theta,b_\theta,\nu_\theta)\vert\!\vert_{L^2}^2 d\theta\\
	&\leq C_{k-1}\(|\!|\mathcal D_{(A,S)}(\zeta,\nu)|\!|_{L^2_{k-1}(I)}^2+|\!|(\zeta,\nu)|\!|_{L^2_{k-1}(I)}^2\).
\end{align*} 
As usual, we use the convention that the value of $C_{k-1}$ might increase from a line to the next one. This completes the proof of Theorem \ref{Fredholm-flex}.

\begin{remark}\label{uniform-fred-prop-higher}
	One can see easily from the above proof that an extension of Remark \ref{uniform-fred-prop} holds for higher Sobolev norms. That is to say, for any $k\geq 1$, there is a neighborhood of 
	$A$, $S$, $\{\mathcal L_\theta,L_\theta\}_{\theta\in I}$, defined with respect to some 
	Sobolev norm  $L^2_{l_k}$ such that for any element of this neighborhood, the analogue of \eqref{Fed-prop-k-I-J} holds with the same constant $C$.
\end{remark}

\subsection{Infinite mixed cylinders}

In this subsection, we consider the operator $\mathcal D_{(A,S)}$ in the case of an infinite cylinder, namely, $I=\R$. We simplify the setup by assuming that $A$ is the pull-back of a connection $B$ on the bundle $E$ over $Y$ and $S$ is constant in the $\R$ direction. That is to say, $S$ is the pull-back of a map from $[0,1]$ to the space of self-adjoint operators, which is denoted by the same notation. In particular, the operator $\mathcal D_{(A,S)}$ has the form
\begin{equation}\label{DAS-op-constant}
  \mathcal D_{(A,S)}=\frac{d}{d\theta}-\fD_{(B,S)}.
\end{equation}
We also fix a Lagrangian $L$ in $\R^{2n}$ and a canonical linearized Lagrangian correspondence $\mathcal L$ from $\Omega^1(\Sigma)$ to $\R^{2n}$ which is compatible with $\alpha$, the flat connection obtained from the restriction of $B$ to the boundary. Associated to $(\mathcal L, L)$, we have the Hilbert space $\mathcal W$ and we regard the operator in \eqref{DAS-op-constant} as a bounded Linear map from 
\begin{equation}\label{domain-constant-mixed}
  L^2_1(\R,\mathcal H)\cap L^2(\R,\mathcal W)
\end{equation}
to the space of $L^2$ pairs $(\zeta,\nu)$. Clearly, the space in \eqref{domain-constant-mixed} can be identified with $E^1_{\fL}(\R)$, defined using $(\mathcal L,L)$, which is regarded as a constant family with respect to $\theta$. We wish to show that the operator in \eqref{DAS-op-constant} is not just a Fredholm operator, but in fact an isomorphism at least in the case that $\fD_{(B,S)}$ is invertible. 

\begin{prop}\label{const-family-invertible}
	Suppose $L:\mathcal W\to \mathcal H$ is an invertible bounded operator.
	Then the operator 
	\[
	  \frac{d}{d\theta}-L:L^2_1(\R,\mathcal H)\cap L^2(\R,\mathcal W)\to L^2(\R,\mathcal H)
	\]
	is an isomorphism.
\end{prop}
\begin{proof}[Sketch of the proof]
	The proof is standard and we only sketch the main steps. 	See, for example, \cite{RD:SF-MI} or \cite[Chapter 3]{Don:YM-Floer} for more details. The composition of $L^{-1}$ and the the inclusion of $\mathcal W$ into $\mathcal H$ determines a compact self-adjoint operator. Thus, there is a complete eigenspace
	decomposition $\{e_i\}_i$ associated to the operator $L$ which provides an orthonormal basis for $\mathcal H$. Using this eigenspace decomposition, any element $(\zeta,\nu)$ of $L^2(\R,\mathcal H)$ can be written as
	\[
	  (\zeta,\nu)=\sum_{i}f_i(t)e_i,
	\]
	where $f_i(t)\in L^2(\R,\R)$ and
	\[
	  |\!|(\zeta,\nu)|\!|_{L^2}^2=\sum_{i}|\!|f_i(t)|\!|_{L^2}^2<\infty.
	\]	
	The norm on \eqref{domain-constant-mixed} is equivalent to
	\[
	  |\!|(\zeta,\nu)|\!|=\sqrt {\sum_{i}|\!|f_i'(t)|\!|_{L^2}^2+\lambda_i^2|\!|f_i(t)|\!|_{L^2}^2}.
	\]
	We have
	\[
	  \(\frac{d}{d\theta}-L\)(\sum_{i}f_i(t)e_i)=\sum_{i}(f_i'(t)+\lambda_if_i'(t))e_i,
	\]
	and one can write down an explicit inverse for this operator in terms of the eigenspace decomposition.
\end{proof}

\begin{remark}\label{perturbation-constant-family-invariance}
	As it is explained in Subsection \ref{unbdd-self-adj-3D}, we may assume that $\fD_{(B,S)}$ is invertible after adding a small multiple of the identity operator. Therefore Proposition \ref{const-family-invertible} is applicable to such perturbations of $\fD_{(B,S)}$. In fact, Proposition \ref{const-family-invertible} can be used in a more general setup where $L=\fD_{(B,S)}+h$ is an invertible operator for some  bounded self-adjoint operator
	$h:\mathcal H\to \mathcal H$.
	Such perturbations of $\fD_{(B,S)}$ appear in \cite{DFL:SO(3)-AF}, where we have to consider perturbations of the mixed equation. 
\end{remark}	

\appendix

\section{Elliptic regularity of bundle-valued 1-forms}\label{elliptic-reg-sec}

In this appendix, first we review some well-known results about regularity of the Laplace-Beltrrami operator. Then we consider slight variations to the case of bundle valued maps. Throughout this section, $M$ denotes a compact Riemannian manifold possibly with boundary. In this appendix, for any Riemannian manifold $M$ and differential $k$-forms $\alpha$ and $\beta$ on $M$, we slightly diverge from our notation in \eqref{inner-prod-diff-forms}, and denote the inner product of $\alpha$ and $\beta$ by
\[
 \int_M\langle \alpha,\beta \rangle.
\]
For any real number $r>1$, we also write $r^*$ for the conjugate of $r$ which satisfies
\[
  \frac{1}{r}+\frac{1}{r^*}=1.
\]
The following lemma is a standard fact about the Laplace-Beltrrami operator (see \cite[Theorems 9.14 and 9.15]{GT:elliptic}, \cite[Theorem 15.2]{ADN:elliptic} and \cite[Chapters 3 and Appendix D]{Weh:com}.)
\begin{lemma} \label{elliptic-reg-laplace}
	Let $k$ be a non-negative integer and $p>1$ be a real number. Let $u$ be an $L^p_k$ function on $M$.
	\begin{enumerate}
		\item[(i)] If $k\geq 1$, suppose there is an $L^p_{k-1}$ function $F$ on $M$ such that for 
		any smooth function $\varphi$ with $\varphi|_{\partial M}=0$, we have
		\begin{equation} \label{weak-Dirichlet-lemma}
		    \int_M \langle u, \Delta \varphi \rangle=
		    \int_M \langle F, \varphi \rangle.
		 \end{equation}
		 Then $u$ is in $L^p_{k+1}(M)$, and 
		 there is a constant $C$, independent of $u$, such that
		 \begin{equation}\label{weak-Dirichlet-conc}
		   |\!|u|\!|_{L^p_{k+1}(M)}\leq C(|\!|F|\!|_{L^p_{k-1}(M)}+|\!|u|\!|_{L^p(M)}).
		 \end{equation}
		 In the case that $k=0$, the assumption \eqref{weak-Dirichlet-lemma} has to be replaced with
		\begin{equation} \label{weak-Dirichlet-lemma-k=0}
		    |\int_M \langle u, \Delta \varphi \rangle|\leq \kappa
		    |\!|\varphi|\!|_{L^{p^*}_{1}(M)},
		 \end{equation}
		 and the conclusion \eqref{weak-Dirichlet-conc} has to be modified to:
		 \begin{equation}\label{weak-Dirichlet-conc}
		   |\!|u|\!|_{L^p_{1}(M)}\leq C(\kappa+|\!|u|\!|_{L^p(M)}).
		 \end{equation}
		 \item[(ii)] If $k\geq 1$, suppose there are functions $F$ and $G$ on $M$ such that for 
		any smooth function $\varphi$ with $\partial_\nu \varphi|_{\partial M}=0$ we have:
		\begin{equation} \label{weak-Neuman-lemma}
		    \int_M \langle u, \Delta \varphi \rangle=
		    \int_M \langle F, \varphi \rangle+\int_{\partial M} \langle G,  \varphi \rangle.
		 \end{equation}
		 If $F$ and $G$ are respectively in $L^p_{k-1}(M)$ and $L^p_{k}(M)$, then $u$ is in $L^p_{k+1}(M)$. Furthermore, 
		 there is a constant $C$, independent of $u$, such that
		 \begin{equation}\label{weak-Neuman-conc}
		   |\!|u|\!|_{L^p_{k+1}(M)}\leq C(|\!|F|\!|_{L^p_{k-1}(M)}+|\!|G|\!|_{L^p_{k}(M)}+|\!|u|\!|_{L^p(M)}).
		 \end{equation}
		 In the case that $k=0$, the assumption \eqref{weak-Neuman-lemma} has to be replaced with:
		\begin{equation} \label{weak-Neuman-lemma-k=0}
		    |\int_M \langle u, \Delta \varphi \rangle|\leq \kappa
		    |\!|\varphi|\!|_{L^{p^*}_{1}(M)},
		 \end{equation}
		 and the conclusion \eqref{weak-Neuman-conc} has to be modified to:
		 \begin{equation}\label{weak-Neuman-conc}
		   |\!|u|\!|_{L^p_{1}(M)}\leq C(\kappa+|\!|u|\!|_{L^p(M)}).
		 \end{equation}
	\end{enumerate}
\end{lemma}

We recall the following definition from Subsection \ref{regularity-thm-sub} about some functions spaces associated to the sections of a vector bundle.

\begin{definition} \label{test-func-spaces}
	Suppose $U$ is a (possibly non-compact) manifold with boundary and $E$ is a vector bundle over $U$. 
	Then the space of smooth sections of $E$ with compact support are denoted by $\Gamma_c(U,E)$. 
	The space of compactly supported sections of $E$, which vanish on the boundary of $E$, are denoted by 
	$\Gamma_\tau(U,E)$. Suppose a connection $A_0$ is fixed on $E$. 
	Then $\Gamma_\nu(U,E)$ is the space of all compactly supported sections $s$ of $E$ 
	such that the covariant derivative of $s$ in the normal directions to the boundary of $U$ vanish.
\end{definition}

The following Lemma is a slightly generalized version of \cite[Lemma A.2]{Weh:Lag-bdry-ana}.
\begin{lemma} \label{elliptic-reg}
	Let $k$ be a positive integer, and $r> 1$ be a real number.
	Let $M$ be a compact $n$-manifold with boundary and a Riemannian metric $g$,
	$U$ be an open subset of $M$, and $K$ be an open subspace of $U$ whose closure in $U$ is compact. 
	Let $E$ be an $SO(3)$-vector bundle over $M$ equipped with a smooth connection $A_0$. 
	Let $\sigma$ be a smooth vector field on $U$. Let $\Gamma_\circ(U,E)$ be one of the spaces $\Gamma_\tau(U,E)$ or $\Gamma_\nu(U,E)$, where
	$\Gamma_\nu(U,E)$ is defined using $A_0$.
	Then there is a constant $C$ such that the following holds. Let
	\[
	  f\in L^{r}_k(U,E),\hspace{.5cm}\alpha,\xi\in L^{r}_k(U,\Lambda^1(M)\otimes E),\hspace{.5cm}\zeta\in L^{r}_{k-1}(U,\Lambda^1(M)\otimes E),\hspace{.5cm}
	  \omega\in L^r_k(U,\Lambda^2(M)\otimes E),
	\]
	and for any $\phi \in \Gamma_c(U,E)$, $\psi \in \Gamma_\circ(U,E)$ we have
	\begin{equation} \label{weak-harm-1}
	  \int_M\langle \alpha, d_{A_0} \phi \rangle=\int_M \langle f, \phi\rangle, 
	\end{equation}
	\begin{equation} \label{weak-harm-2}  
	  \int_M\langle \alpha, d_{A_0}^*d_{A_0}(\psi\cdot \iota_\sigma g) \rangle=
	  \int_M\langle \omega, d_{A_0} (\psi\cdot \iota_\sigma g)\rangle+\int_M\langle \zeta, \psi\cdot \iota_\sigma g \rangle
	  +\int_{\partial M}\langle \xi, \psi\cdot \iota_\sigma g \rangle.
	\end{equation}
	Then $\alpha(\sigma)$ is an element of $L^{r}_{k+1}(K)$ and we have:
	\[
	  |\!| \alpha(\sigma)|\!|_{L^{r}_{k+1}(K)}\leq C (|\!|f|\!|_{L_k^{r}(U)}+|\!|\xi|\!|_{L_{k}^{r}(U)}
	  +|\!|\zeta|\!|_{L_{k-1}^{r}(U)}+|\!|\omega|\!|_{L_k^{r}(U)}+
	   |\!| \alpha |\!|_{L_k^{r}(U)}).
	\]
\end{lemma}	

\begin{proof}	 
	Without loss of generality, we may assume that $U$ is a precompact open subset of the half-space
	\[\mathbb H^n:=\{(x_1,\dots,x_n)\in \R^n\mid x_1\geq 0\},\] 
	 $E$ is trivialized over
	$U$ and the connection $A_0$ is given by a 1-form with values in $\R^3$. 
	We will denote this 1-form with $A_0$, too. We may pick this trivialization in a way that the normal covariant derivative with respect to the connection $A_0$ agrees with the ordinary derivative.
	That is to say, $\Gamma_\nu(U,\R^3)$ defined with respect to $A_0$ is the space of all compactly supported sections $\eta$ of $\R^3$ 
	such that $\partial_\nu \eta$ vanishes along the boundary. 
	 
	Fix a function $\rho:M\to \R$ which is supported in $U$
	and is equal to $1$ on $K$. Then we show that there are compactly supported 
	maps $F$ and $G$ from $U$ to $\R^3$ such that for any $\eta\in \Gamma_\circ(U,\R^3)$ we have
	\begin{equation} \label{weak-Dirichlet}
	    \int_M \langle \rho \alpha(\sigma), \Delta \eta \rangle=
	    \int_M \langle F, \eta \rangle+\int_{\partial M} \langle G,  \eta \rangle
	 \end{equation}
	 and $F$, $G$ respectively have finite $L^{r}_{k-1}$, $L^{r}_k$ norms.

	First we claim that 
	\begin{align} 
	   \int_M \langle \rho \alpha(\sigma), \Delta \eta \rangle
	   =&-\int_M \rho \langle  \alpha,d \iota_\sigma d\eta  \rangle- \int_M\langle \alpha, d^*(\rho  \iota_\sigma g \wedge d\eta)\rangle
	    -\int_M\rho \div(\sigma) \langle \alpha, d\eta\rangle \nonumber \\
	    &-\int_M \rho \langle B_g\alpha, d\eta\rangle 
	    -\int_M\langle \iota_\sigma (d\rho\wedge \alpha), d\eta\rangle\label{simplify-lap},
	 \end{align}	
	 where $B_g$ is defined by firstly taking the Lie derivative $\mathcal L_\sigma g$ of the Riemannian metric
	 $g$ and then requiring $B_g$ to satisfy the following identity for any pair of $1$-forms $\beta$ and $\beta'$:
	 \[
	  \mathcal L_\sigma (g)(\beta,\beta')=\langle B_g \beta,\beta' \rangle
	\]
	To see \eqref{simplify-lap}, 
	we pick a sequence $\{\alpha_i\}_{i\in \N}$ of smooth 1-forms on $U$ with values 
	in $\R^3$ such that $\alpha_i$ vanishes in a neighborhood of $U\cap \partial {\mathbb H}^n$ and 
	the sequence $\{\alpha_i\}$ is $L^r$-convergent to $\alpha$. 
	Then the left hand side of \eqref{simplify-lap} is equal to 
	\begin{align} \label{simplify-lap-verif}
	\lim_{i\to \infty}\int_M \langle\rho \alpha_i(\sigma),d^*d \eta \rangle
	   =&\lim_{i\to \infty} \int_M \langle d\iota_\sigma (\rho \alpha_i),d\eta \rangle
	   =\lim_{i\to \infty}\left[ \int_M \langle \mathcal L_\sigma (\rho \alpha_i),d \eta \rangle-
	   \langle \mathcal \iota_\sigma d(\rho \alpha_i),d \eta \rangle\right ]\nonumber\\
	   =&\lim_{i\to \infty}\left[-\int_M \langle \rho \alpha_i,\mathcal L_\sigma d \eta \rangle-
	    \int_M\div(\sigma) \langle \rho \alpha_i,d \eta \rangle-\int_M\mathcal L_\sigma(g)(\rho\alpha_i, d \eta)\right. \nonumber\\
	    &\left. -\int_M\langle d\alpha_i, \rho  \iota_\sigma g \wedge d\eta\rangle
	    -\int_M\langle \iota_\sigma (d\rho\wedge \alpha_i), d\eta\rangle\right] \nonumber \\
	   =&\lim_{i\to \infty}\left[-\int_M \rho\langle \alpha_i,d \iota_\sigma d\eta \rangle-
	    \int_M \rho\div(\sigma) \langle \alpha_i, d\eta \rangle -\int_M \rho\langle B_g \alpha_i, d\eta\rangle \right.\nonumber\\
	    &\left. -\int_M\langle \alpha_i, d^*(\rho  \iota_\sigma g \wedge d\eta)\rangle
	    -\int_M\langle \iota_\sigma (d\rho\wedge \alpha_i), d\eta\rangle\right].
	 \end{align}	 
	Now by taking the limit in \eqref{simplify-lap-verif} we obtain the desired identity. 

	The assumption \eqref{weak-harm-1} can be used to rewire the first term in the right hand side of 
	\eqref{simplify-lap} as
	 \begin{align} \label{1st-term}
	    \int_M \rho\langle \alpha,d \iota_\sigma  d\eta \rangle
	   =&\int_M \langle \alpha,d_{A_0}(\rho \iota_\sigma  d\eta) \rangle-\int_M \langle \alpha,\rho[A_0,\iota_\sigma  d \eta] \rangle-
	   \int_M \langle \alpha,\mathcal (d\rho)\cdot(\iota_\sigma  d\eta) \rangle \nonumber \\
	   =&\int_M \langle \rho f-*[\rho\alpha,*A_0]-*(\alpha\wedge *d\rho), \iota_\sigma  d\eta \rangle\nonumber \\
	   =&\int_M \langle \left (\rho f-*[\rho \alpha,*A_0]-*(\alpha\wedge *d\rho)\right )\iota_\sigma g, d\eta \rangle\nonumber\\
	    =&\int_M \langle d^*\( \left (\rho f-*[\rho \alpha,*A_0]-*(\alpha\wedge *d\rho)\right )\iota_\sigma g\),\eta \rangle\nonumber\\
	    &+\int_{\partial M} \langle *_{n-1}*(\left (\rho f-*[\rho \alpha,*A_0]-*(\alpha\wedge *d\rho)\right )\iota_\sigma g), \eta \rangle,
	 \end{align}	 
	where $*_{n-1}$ in the  last line denotes the Hodge operator on $\partial M$.
	
	We rewrite the second term in the right hand side of \eqref{simplify-lap} as
	 \begin{align*}
	    \int_M\langle \alpha, d^*(\rho  \iota_\sigma  g \wedge d\eta)\rangle
	    =&\int_M\langle \alpha, d^*( \eta d(\rho  \iota_\sigma  g))\rangle-\int_M\langle \alpha, d^*d(\rho  \iota_\sigma  g \eta)\rangle \nonumber\\
	   =&\int_M\langle d \alpha, \eta d(\rho  \iota_\sigma  g) \rangle-\int_{\partial M}\langle *_{n-1}(\alpha\wedge *d(\rho  \iota_\sigma  g)), \eta\rangle\nonumber\\
	   &-\int_M\langle \alpha, d_{A_0}^*d_{A_0}(\iota_\sigma  g \rho \eta)\rangle+
	     (-1)^{n-1}\int_M \langle \alpha, *[A_0,*d(\iota_\sigma  g \rho \eta)]\rangle\nonumber\\
	    &+\int_M\langle \alpha, d^*[A_0,\iota_\sigma  g \rho \eta]\rangle+(-1)^{n-1}\int_M \langle \alpha, *[A_0,*[A_0,\iota_\sigma  g \rho \eta]]\rangle\nonumber.
	 \end{align*}	
	Therefore, we can use \eqref{weak-harm-2}, to write
	 \begin{align}\label{2nd-term}
	    \int_M\langle \alpha, d^*(\rho  \iota_\sigma  g \wedge d\eta)\rangle
	   =&\int_M\langle *(d \alpha\wedge *d(\rho  \iota_\sigma  g)), \eta\rangle-\int_{\partial M}\langle *_{n-1}(\alpha\wedge *d(\rho  \iota_\sigma  g)), \eta\rangle\nonumber\\
	   &-\int_M\langle \omega,d_{A_0}(\iota_\sigma  g \rho \eta)\rangle-\int_M\langle \zeta,\iota_\sigma  g \rho \eta\rangle
	   -\int_{\partial M}\langle \xi,\iota_\sigma  g \rho \eta\rangle
	   \nonumber\\ 
	   &+(-1)^{n-1}\int_M \langle \alpha, *[A_0,*d(\iota_\sigma  g \rho \eta)]\rangle
	    +\int_M\langle \alpha, d^*[A_0,\iota_\sigma  g \rho \eta]\rangle\nonumber\\
	    &+(-1)^{n-1}\int_M \langle \alpha, *[A_0,*[A_0,\iota_\sigma  g \rho \eta]]\rangle.
	 \end{align}	 	 
	 Finally, the last three terms of \eqref{simplify-lap} are equal to
	\begin{equation} \label{3rd-term}
		-\int_M\langle d^*\left[(B_g+\div(\sigma ))\rho \alpha-\iota_\sigma  (\alpha\wedge d\rho)\right], \eta\rangle
		-\int_{\partial M}\langle *_{n-1}*[(B_g+\div(\sigma ))\rho \alpha-\iota_\sigma  (\alpha\wedge d\rho)],\eta\rangle
	\end{equation}
	 By applying further integration by parts to the expressions in \eqref{2nd-term}, we can 
	 find $F$ and $G$ satisfying \eqref{weak-Dirichlet}, which are respectively in $L^r_{k-1}$ and $L^r_{k}$, and satisfy
	 \[
	   |\!|F|\!|_{L^{r}_{k-1}}+|\!|G|\!|_{L^{r}_k(U)}\leq C'(|\!|f|\!|_{L_k^{r}(U)}+|\!|\omega|\!|_{L_k^{r}(U)}+
	   |\!|\zeta|\!|_{L_{k-1}^{r}(U)}+|\!|\xi|\!|_{L_{k}^{r}(U)}+
	   |\!| \alpha |\!|_{L_k^{r}(U)})
	 \]	 
	 for some constant $C'$ depending only on $A_0$, $g$, $\sigma$, $U$ and $K$. 
	 Therefore, Lemma \ref{elliptic-reg-laplace} (part (i) or (ii) depending on whether $\circ=\tau$ or $\nu$) implies that
	\[
	|\!| \rho\alpha(\sigma )|\!|_{L^{r}_{k+1}(U)}\\
	  \leq C (|\!|f|\!|_{L_k^{r}(U)}+|\!|\omega|\!|_{L_k^{r}(U)}
	  +|\!|\zeta|\!|_{L_{k-1}^{r}(U)}+|\!|\xi|\!|_{L_{k}^{r}(U)}+
	   |\!| \alpha |\!|_{L_k^{r}(U)}).
	\]
	This inequality proves the desired claim.
\end{proof}

The following lemma is an extension of the previous lemma to the case that $k=0$.
\begin{lemma}\label{elliptic-reg-k=0}
	Let $r$, $M$, $K$, $U$, $\sigma $, $E$ and $A_0$ be as in Lemma \ref{elliptic-reg}. 
	Let $\circ$ be either $\tau$ and $\nu$. There is a constant $C$ such that the following holds. Let $\alpha$ be an $L^r$
	section of $\Lambda^1\otimes E$ over the open subset $U$ of $M$ such that for any $\phi\in \Gamma_c(U,E)$
	and $\psi\in \Gamma_\circ(U,E)$:
	\begin{equation} \label{weak-harm}
	 | \int_M\langle \alpha, d_{A_0} \phi \rangle|\leq C_1 |\!|\phi|\!|_{L^{r^*}(U)}, \hspace{1cm} 
	 | \int_M\langle \alpha, d_{A_0}^*d_{A_0}(\psi\cdot \iota_\sigma g) \rangle|\leq 
	  C_2 |\!|\psi|\!|_{L_1^{r^*}(U)}.
	\end{equation}
	Then $\alpha(\sigma)$ belongs to $L^{r}_1(K)$ and 
	\begin{equation}\label{ineq-assumptions}
	  |\!| \alpha(\sigma )|\!|_{L^{r}_1(K)}\leq C(C_1+C_2+|\!| \alpha |\!|_{L^r(U)}).
	\end{equation}
\end{lemma}
\begin{proof}
	In the following $C$ is a constant independent of $\alpha$ which might increase from each line to the next one.
	As in the proof of Lemma \ref{elliptic-reg}, we can show that $\alpha$ satisfies \eqref{simplify-lap}. In particular, we have
	\begin{align} \label{simplify-lap-2}
	   |\int_M \langle \rho \alpha(\sigma ), \Delta \eta \rangle|\leq & \hspace{2mm}|\int_M \rho \langle  \alpha,d \iota_\sigma  d\eta  \rangle|+ 
	   |\int_M\langle \alpha, d^*(\rho  \iota_\sigma  g \wedge d\eta)\rangle|+
	   |\int_M\rho \div(\sigma ) \langle \alpha, d\eta\rangle| \nonumber\\
	   &+|\int_M \rho \langle B_g\alpha, d\eta\rangle|+
	    |\int_M\langle \iota_\sigma  (d(\rho)\wedge \alpha), d\eta\rangle|.
	 \end{align}
	The first term on the left hand side of the above inequality can be estimated as in \eqref{1st-term}:
	\begin{align} \label{1st-term-weak}
	    |\int_M \rho\langle \alpha,d \iota_\sigma  d\eta \rangle|
	   \leq &|\int_M \langle \alpha,d_{A_0}(\rho \iota_\sigma  d\eta) \rangle|+|\int_M \langle \alpha,\rho[A_0,\iota_\sigma  d \eta] \rangle|+
	   |\int_M \langle \alpha,\mathcal (d\rho)\cdot(\iota_\sigma  d\eta) \rangle| \nonumber \\
	   \leq &C(C_1+|\!| \alpha |\!|_{L^r(U)})|\!|\eta |\!|_{L^{r^*}_1(U)}.
	 \end{align}	
	 To. obtain the second inequality, we use the first assumption in \eqref{weak-harm}. Next, we find an upper bound for  the second term in \eqref{simplify-lap-2} using the second inequality in \eqref{weak-harm} following an argument similar to the 
	 previous lemma:
	  \begin{align}\label{2nd-term-weak}
	    |\int_M\langle \alpha, d^*(\rho  \iota_\sigma  g \wedge d\eta)\rangle|
	    \leq &|\int_M\langle \alpha, d^*( \eta d(\rho  \iota_\sigma  g))\rangle|+
	    |\int_M\langle \alpha, d_{A_0}^*d_{A_0}(\iota_\sigma  g \rho \eta)\rangle|\nonumber\\
	    &+|\int_M \langle \alpha, *[{A_0},*d(\iota_\sigma  g \rho \eta)]\rangle|+
	    |\int_M\langle \alpha, d^*[{A_0},\iota_\sigma  g \rho \eta]\rangle|\nonumber\\
	    &+|\int_M \langle \alpha, *[{A_0},*[{A_0},\iota_\sigma  g \rho \eta]]\rangle|\nonumber\\
	   \leq&C(C_2+|\!| \alpha |\!|_{L^r(U)})|\!|\eta |\!|_{L^{r^*}_1(U)}
	 \end{align}	 
	 It is straightforward to bound the remaining three terms in \eqref{simplify-lap-2} 
	 with $C|\!| \alpha |\!|_{L^r(U)}|\!|\eta |\!|_{L^{r^*}_1(U)}$. Consequently, Lemma \ref{elliptic-reg-laplace} implies that $\alpha(\sigma)$ is in $L^{r}_1(K)$ and \eqref{ineq-assumptions} holds.
\end{proof}

\begin{lemma}\label{weak-Dirac-lemma}
	Let $k$ be a non-negative integer and $r>1$ is a real number. Suppose $M$ is a Riemannian manifold possibly with boundary. Suppose $\Sigma$ is a closed surface and $F$ is an $\SO(3)$-bundle over $\Sigma$.
	Suppose $\beta=\{\beta_x\}_{x\in M}$ is a smooth family of connections on $F$ parametrized by $M$. 
	Suppose $f$ is an $L^r_k$ section of the bundle $T^*\Sigma\otimes F$ over $\Sigma\times M$. 
	If $k\geq 1$, suppose there are $L^r_{k}$ sections $\zeta_1$ and $\zeta_2$ of the pullback of $F$ over $\Sigma\times M$ such that for 
	any smooth section $\xi$ of the pullback of $F$ over $\Sigma\times M$, we have
	\begin{equation} \label{weak-Dirac-lemma-assumption}
	    \int_{M\times \Sigma} \langle f, d_{\beta} \xi \rangle=
	    \int_{M\times \Sigma} \langle \zeta_1, \xi \rangle,\hspace{1cm}
	    \int_{M\times \Sigma} \langle f, *_\Sigma d_\beta \xi \rangle=
	    \int_{M\times \Sigma} \langle \zeta_2, \xi \rangle.
	 \end{equation}
	 where $d_{\beta} \xi$ denotes the section of $T^*\Sigma\otimes F$ over $\Sigma\times M$ given by the exterior derivatives of $\xi$ in the $\Sigma$ direction with respect to the family of connections $\beta$.
	 Then $\nabla^\beta_\Sigma f$, the covariant derivative of $f$ in the $\Sigma$ direction with respect to $\beta$, is in $L^r_{k}$, and 
	 there is a constant $C$, independent of $f$, such that:
	 \begin{equation}\label{weak-Dirac-conc}
	   |\!|\nabla^{\beta}_\Sigma f|\!|_{L^r_{k}(M\times \Sigma)}\leq C(|\!|\xi_1|\!|_{L^r_{k}(M\times \Sigma)}+|\!|\xi_2|\!|_{L^r_{k}(M\times \Sigma)}+|\!|f|\!|_{L^r_{k}(M\times \Sigma)}).
	 \end{equation}
	 In the case that $k=0$, the assumption \eqref{weak-Dirac-lemma-assumption} has to be replaced with
	\begin{equation} \label{weak-Dirac-lemma-k=0}
	  |\int_{\Sigma\times X} \langle f, d_{\beta} \xi \rangle|+
	  |\int_{\Sigma\times X} \langle f,*_\Sigma d_{\beta}\xi \rangle|
	  \leq \kappa|\!|\xi|\!|_{L^{r^*}(\Sigma\times X)}.
	 \end{equation}
	 In this case, $\nabla^{\beta}_\Sigma f$ belongs to $L^r(X\times \Sigma)$ and 
	 \begin{equation}\label{weak-Dirac-conc-0}
	   |\!|\nabla^{\beta}_\Sigma f|\!|_{L^r(X\times\Sigma)}\leq C(\kappa+|\!|f|\!|_{L^r(X\times\Sigma)}).
	 \end{equation}
\end{lemma}

Lemma \ref{weak-Dirac-lemma} can be regarded as the family version of \ref{elliptic-reg-laplace} where we also replace the degree two elliptic operator $\Delta$ with the degree one operator $d_\beta\oplus d_\beta^*$. This proposition in the case that $F$ is the trivial bundle and $\beta$ is the trivial family of connections is proved in \cite[Lemma 2.9]{Weh:Lag-bdry-ana}. Clearly, this implies the lemma for the case that $F$ is trivial and $B$ is arbitrary. The proof in the case that $F$ is non-trivial is similar.

\section{Regularity of holomorphic curves in a Banach space}\label{Banach-space-valued}

Suppose $B$ is a Banach space and $M$ is a compact Riemannian manifold. In this appendix, we are interested in maps from $M$ to $B$. For $1<p<\infty$ and any non-negative integer $k$, we can define the Sobolev norm $|\!|\cdot|\!|_{L^p_k}$ on the space of such maps in the usual way. The completion of space of smooth maps from $M$ to $B$ with respect to this Sobolev norm is  denoted by $L^p_k(M,B)$. As an example, let $B=L^p(N)$ for a compact manifold $N$. Any function in $C^\infty(M\times N)$, determines an element of $L^p(M,B)$. In fact, the space of smooth functions on $M\times N$ is dense in $L^p(M,B)$ (see \cite{W:Ban-elliptic} and \cite{Max:GU-comp}). This gives us the following identifications of Sobolev spaces:
\[
  L^p(M,L^p(N))=L^p(N,L^p(M))=L^p(M\times N).
\]
More generally, $C^\infty(M\times N)$ is dense in $L^p_k(M,L^p(\Sigma))$ for any non-negative integer $k$, and we have (see \cite{W:Ban-elliptic,Max:GU-comp}):
\begin{equation} \label{fun-space-Ban}
  L^p_k(M\times N)=L^p_k(M,L^p(N))\cap L^p_k(N,L^p(M)),\hspace{1cm} L^p_k(M, L^p(N))=L^p(N,L^p_k(M)).
\end{equation}

For the rest of this appendix, we fix $B_p$ to be a Banach space that can be identified with a closed subspace of the space $L^p(N)$ for a closed manifold $N$.  In particular, the intersection $B_q:=B_p\cap L^q(N)$ with $q>p$ determines a closed subspace of $L^q(N)$. For $q<p$, $B_q$ is the closure of $B_p$ in $L^q(N)$.
\begin{lemma}[\cite{W:Ban-elliptic} and \cite{Max:GU-comp}]\label{elliptic-reg-laplace-Ban}
	Suppose $M$ is a Riemannian manifold with boundary. Let $k$ be a non-negative integer and $p>1$ be a real number. 
	Let
	$u\in L^p_k(M,B_p)$. Then the same claims as in parts
	(i) and (ii) of Lemma \ref{elliptic-reg-laplace} hold if we assume that $F$, $G$ and $\varphi$ are $B_p$-valued.	
\end{lemma}

\begin{proof}[Sketch of the Proof.]	
		Without loss of generality, we can assume that $B_p=L^p(N)$. Using the identifications in \eqref{fun-space-Ban}, we can regard $u$ as an $L^p$ map
		from $N$ to the Banach space $L^p_k(M)$. Next, we can apply the properties of the Laplacian operator acting on $L^p_k(M)$ to obtain the desired conclusions.
		For more details, we refer the reader to \cite[Lemma 2.1]{W:Ban-elliptic} and \cite[Subsection 3.3]{Max:GU-comp}.
\end{proof}

The proof of the following proposition about regularity of Banach valued Cauchy-Riemann equation can be found in \cite[Theorem 1.2]{W:Ban-elliptic} and \cite[Lemmas 27 and 28]{Max:GU-comp}. In this proposition,  ${\bB}_p$ denotes the direct sum $B_p\oplus B_p$. This space admits an obvious complex structure $J_0$ given by
\begin{equation}\label{J0}
  \mathcal J_0(v_0,v_1)=(-v_1,v_0).
\end{equation}
The subspace $\mathcal L:=0\oplus B_p$ defines a completely real subspace of $\bB_{p}$ with respect to $\mathcal J_0$. 

\begin{prop}\label{Ban-val-reg}
	Suppose $U$ is a bounded open subspace of \[\bbH^2:=\{(s,\theta)\in \R^2\mid s \geq 0\},\] and $U_\partial$ denotes the intersection $\bbH^2\cap U$.
	Suppose $\mathcal J:\bB_{p}\to {\rm End}(\bB_{p},\bB_{p})$ is a smooth family of complex structures such that $\mathcal J(x)=\mathcal J_0$ for $x\in \mathcal L$. 
	For $p>2$ and $k\geq 2$, suppose $u:U\to \bB_{p}$ is an $L^p_k$ map that satisfies
	\begin{equation}\label{Ban-CR}
	  \partial_\theta u-\mathcal J(u)\partial_s u=z\in L^p_k(U,\bB_{p}),
	\end{equation}
	and the boundary condition
	\begin{equation}\label{Ban-bd}
	  u|_{U_\partial}\subset \mathcal L.
	\end{equation}	
	Then for any open subspace $K\subset U$, whose closure in $U$ is compact, the map $u$ is in 
	$L^p_{k+1}(K)$. Moreover, there is a constant $C$, depending only on $K$, such that
	\begin{equation}
		|\!|u|\!|_{L^p_{k+1}(K)}\leq C(|\!|z|\!|_{L^p_{k}(U)}+|\!|u|\!|_{L^p}(U)).		
	\end{equation}
	If $u_i:U\to \bB_{p}$ is a sequence of $L^p_k$ map that satisfies
	\begin{equation}\label{Ban-CR-seq}
	  \partial_\theta u_i-\mathcal J(u_i)\partial_s u_i=z_i\in L^p_k(U,\bB_{p}),
	\end{equation}
	such that $u_i$ and $z_i$ are respectively $L^p_k$-convergent to $u$ and $z$, then $u_i$ restricted to $K$ is 
	$L^p_{k+1}$-convergent to the restriction of $u$ to $K$. In the case that $k=1$, similar results hold if we replace 
	$L^{p}_{k+1}$ with $L^{p/2}_{k+1}$.
\end{prop}
\begin{proof}[Sketch of the proof]
	For $k\geq 2$, suppose $u$ is a map that satisfies \eqref{Ban-CR} and \eqref{Ban-bd}. 
	We apply $\partial_\theta+\mathcal J(u)\partial_s$ to
	\eqref{Ban-CR}. Then we have:
	\begin{equation}\label{Ban-CR-squared}
		 \partial_s^2 u+\partial_\theta^2 u=\mathcal J(u)\partial_s(\mathcal J(u))\partial_su+
		 \partial_\theta(\mathcal  J(u))\partial_su
		 +\partial_\theta z+\mathcal J(u)\partial_sz
	\end{equation}
	Using the assumptions $k\geq 2$, $u\in L^p_k$ and $z\in L^p_k$, we can conclude that 
	the left hand side of the above identity is an element of $L^p_{k-1}$. The maps $u$ and $z$ can be written as 
	$(u_0,u_1)$ and $(z_0,z_1)$ with respect to the decomposition of $\bB_{p}$. 
	The boundary condition \eqref{Ban-bd} implies that $u_0|_{U_\partial}=0$ and 
	$\partial_su_1|_{U_\partial}=z_0|_{U_\partial}$. Therefore, we can invoke Lemma \ref{elliptic-reg-laplace-Ban} 
	to verify the claim. To be a bit more detailed, we use the assumption $k\geq 2$ to conclude that the products of two
	$L^p_{k-1}$ functions are still in $L^p_{k-1}$. In the case that $k=1$, the products of two $L^p(U,\bB_p)$ functions 
	is in $L^{p/2}(U,\bB_p)$, which in turn is a subspace of $L^{p/2}(U,\bB_{p/2})$.
	That allows us to use the same argument to prove the claim in this case.
	The sequential versions of these claims can be also treated similarly.
\end{proof}

We need a slight improvement of Proposition \ref{Ban-val-reg} to the case $k=0$ \cite[Lemma 29]{Max:GU-comp}.

\begin{prop}\label{Ban-val-reg-k=0}
	Suppose $U$ is given as in Proposition \ref{Ban-val-reg}. Suppose $\mathcal J:\bB_{p}\to {\rm End}(\bB_{p},\bB_{p})$ is a smooth family of complex structures such that $\mathcal J(x)=\mathcal J_0$ for $x\in \mathcal L$ and for any 
	$x\in \bB_{p}$, the space $\mathcal L$ is totally real with respect to $\mathcal J(x)$, i.e., $\bB_{p}=\mathcal L\oplus \mathcal J(x)\mathcal L$.
	For $p>2$, let $u:U\to \bB_{p}$ be in $L^p_1$. Suppose $q>p$ and $u$ is also an $L^q$ map from $U$ to $\bB_{q}$. Suppose $u$ satisfies 
	\begin{equation}\label{Ban-CR-q}
	  \partial_\theta u-\mathcal J(u)\partial_s u=z\in L^q(U,\bB_{q}),
	\end{equation}
	and the boundary condition \eqref{Ban-bd}. Then $u$ is an $L^q_1$ map from $U$ to $\bB_{q}$ and 
	\begin{equation}
		|\!|u|\!|_{L^q_{1}}\leq C(|\!|z|\!|_{L^q}+|\!|u|\!|_{L^q}).		
	\end{equation}	
	Moreover, if $u_i:U\to \bB_{p}$ are $L^q_1$ solutions of 
	\begin{equation}\label{Ban-CR-iq}
	  \partial_\theta u_i-\mathcal J(u_i)\partial_s u_i=z_i\in L^q(U,\bB_{q}),
	\end{equation}
	such that $u_i$ is convergent to $u$ in $L^p_1\cap L^q$ and $z_i$ is convergent to $z$ in $L^q$, then $u_i$ is convergent to $u$ in $L^q_1$.
\end{prop}
\begin{proof}
	Given $p>2$ and any bounded domain $\Omega$ in $\R^2$ with smooth boundary, let $L^p_1(\Omega,\bB_p)_\partial$ be the space of $L^p_1$ maps $u:\Omega\to \bB_p$
	such that the restriction of $u$ to the boundary is in $\mathcal L$. Then the Cauchy-Riemann operator 
	\begin{equation}\label{CR-constant}
	  \partial_\theta -\mathcal J_0\partial_s :L^p_1(\Omega,\bB_p)_\partial \to L^p(\Omega,\bB_p)
	\end{equation}
	is a surjective bounded operator with kernel being constant maps to $\mathcal L$. This can be seen in the same way as in Lemma \ref{elliptic-reg-laplace-Ban}.
	
	Now suppose $x\in \partial U$ and $D_r(x)=B_r(x)\cap \bbH^2$ is contained in $U$. Suppose $\Omega_r$ is the region given by rounding the corners of $D_r(x)$ such that it is contained in $D_r(x)$ and it contains $D_{r/2}(x)$. Since 
	$\mathcal J(u):U \to {\rm End}(\bB_{p},\bB_{p})$ is continuous and $\mathcal J(x)=J_0$, the operator $  \partial_\theta -\mathcal J(u)\partial_s:L^p_1(\Omega_r,\bB_p)_\partial \to L^p(\Omega_r,\bB_p)$ is surjective with kernel being 
	constant maps to $\mathcal L$ if $r$ is small enough. This holds because the operator $  \partial_\theta -\mathcal J(u)\partial_s$ is a deformation of the operator in \eqref{CR-constant} by a bounded operator of small norm for small values of 
	$r$. We assume that $r$ is chosen such that the same claim holds if we replace $q$ with $p$. Now let $\rho:\Omega_r\to \R$ be a smooth bump function that vanishes on the complement of $D_{r/2}(x)$ and equals $1$ on $D_{r/3}(x)$.
	Then our assumption implies that $\rho u$ is an element of $L^p_1(\Omega_r,\bB_p)_\partial$ and 
	\[
	   \partial_\theta(\rho u) -\mathcal J(u)\partial_s(\rho u)=\rho z+\partial_\theta(\rho) u -\mathcal J(u)\partial_s(\rho) u
	\]
	is in $L^q$. Thus there is $u'\in L^q_1(\Omega,\bB_q)_\partial$ such that 
	\[
	   \partial_\theta u' -\mathcal J(u)\partial_su'=\rho z+\partial_\theta(\rho) u -\mathcal J(u)\partial_s(\rho) u.
	\]
	This implies that $u'-\rho u$ is a constant map to $\mathcal L$. In particular, the restriction of $u$ to $D_{r/3}(x)$ is in $L^q_1(\Omega,\bB_q)_\partial$. For an interior point $x$, we may apply a similar argument to show that the restriction 
	of $u$ to a neighborhood of $x$ in is $L^q_1(\Omega,\bB_q)_\partial$. The only new point that we need is that we can find an isomorphism $T:\bB_p \to \bB_p$ such that $T^{-1} \mathcal J(x) T=\mathcal J_0$. In fact, we may take $T$ to be the 
	linear map that sends $(v_0,v_1)\in B_p\oplus B_p$ to $(v_0,0)+\mathcal J(x)(v_1,0)$. Since $\mathcal L$ is totally with respect to $\mathcal J(x)$, $T$ is an isomorphism.
\end{proof}

\bibliography{references}
\bibliographystyle{hplain}
\end{document}